\documentclass{amsart}

\usepackage{latexsym,amsfonts,amssymb,exscale,enumerate}
\usepackage{amsmath,amsthm,amscd}
\usepackage{hyperref}

\input xy
\usepackage[all]{xy}
\xyoption{line}
\xyoption{arrow}
\xyoption{color}
\SelectTips{cm}{}
\usepackage{tikz}
\usetikzlibrary{decorations.markings}
\usetikzlibrary{decorations.pathreplacing}
\newcommand{\hackcenter}[1]{
 \xy (0,0)*{#1}; \endxy}

\tikzset{->-/.style={decoration={
  markings,
  mark=at position #1 with {\arrow{>}}},postaction={decorate}}}

\tikzset{middlearrow/.style={
        decoration={markings,
            mark= at position 0.5 with {\arrow{#1}} ,
        },
        postaction={decorate}
    }
}

\usepackage{graphicx}
\usepackage{color}
\def\P{\mathsf{P}}
\def\Q{\mathsf{Q}}

\def\W{W_{1+\infty}}
\def\tW{\widetilde{W}_{1+\infty}}
\def\id{\mathrm{id}}
\def\Id{\mathrm{Id}}

\theoremstyle{plain}
\newtheorem{theorem}{Theorem}
\newtheorem{corollary}[theorem]{Corollary}
\newtheorem{proposition}[theorem]{Proposition}
\newtheorem{lemma}[theorem]{Lemma}

\newtheorem{example}[theorem]{Example}

\theoremstyle{definition}

\theoremstyle{remark}
\newtheorem{remark}[theorem]{Remark}

\newcommand{\Ucat}{\cal{U}}

\newcommand{\maps}{\colon}

\newcommand{\refequal}[1]{\xy {\ar@{=}^{#1}
(-1,0)*{};(1,0)*{}};
\endxy}
\newcommand{\cat}[1]{\ensuremath{\mbox{\bfseries {\upshape {#1}}}}}

\newcommand{\To}{\Rightarrow}

\renewcommand{\to}{\rightarrow}

\def\Ind{{\mathrm{Ind}}}

\def\dmod{{\mathrm{-mod}}}   %% finitely-generated modules
\def\Id{\mathrm{Id}}

\def\mf{\mathfrak}
\def\shuffle{\,\raise 1pt\hbox{$\scriptscriptstyle\cup{\mskip
               -4mu}\cup$}\,}

\numberwithin{equation}{section}

\def\emph#1{{\sl #1\/}}

\let\tilde=\widetilde

\let\theta=\vartheta
\let\epsilon=\varepsilon

\usepackage{bbm}
\def\C{{\mathbb{C}}}
\def\N{{\mathbbm N}}

\def\Z{{\mathbbm Z}}
\def\H{{\mathcal{H}}}

\def\cal#1{\mathcal{#1}}%
\def\1{\mathbbm{1}}%
\def\nn{\notag}

\def\la{\langle}
\def\ra{\rangle}

\newcommand{\lowrru}[1]{\xybox{%
  (-8,0)*{};
  (8,0)*{};
  (-6,-18)*{};(6,-9)*{} **\crv{(-6,-13) & (6,-15)} ?(1)*\dir{>};
  (6,-9)*{};(6,0)*{}  **\dir{-} ?(.3)*\dir{ }+(2,0)*{\scs {\bf j}};
}}

\newcommand{\lowllu}[1]{\xybox{%
  (-8,0)*{};
  (8,0)*{};
  (6,-18)*{};(-6,-9)*{} **\crv{(6,-13) & (-6,-15)} ?(1)*\dir{>};
  (-6,-9)*{};(-6,0)*{}  **\dir{-} ?(.3)*\dir{ }+(-2,0)*{\scs {\bf j}};
}}

\newcommand{\bbdl}[1]{\xybox{%
  (2,0);(0,-8) **\crv{(2,-2)&(0,-6)}; ?(.5)*\dir{>}
}}
\newcommand{\bbdlu}[1]{\xybox{%
  (2,0);(0,-8) **\crv{(2,-2)&(0,-6)}; ?(.5)*\dir{<}
}}
\newcommand{\bbdr}[1]{\xybox{%
  (-2,0);(0,-8) **\crv{(-2,-2)&(0,-6)}; ?(.5)*\dir{>}
}}
\newcommand{\bbdru}[1]{\xybox{%
  (-2,0);(0,-8) **\crv{(-2,-2)&(0,-6)}; ?(.5)*\dir{<}
}}

\DeclareGraphicsExtensions{.pdf,.eps}
\usepackage{psfrag}
\usepackage{bbm}

\def\cal#1{\mathcal{#1}}

\def \k {\mathbbm{k}}
\def \Z {\mathbbm{Z}}

\def \N {\mathbbm{N}}

\def \C {\mathcal{C}}
\def \Tr{\operatorname{Tr}}

\def \Span{\operatorname{Span}}
\def \Ob{\operatorname{Ob}}

\def \Id {{\rm Id}}

\def\k{\mathbbm{k}}

\newcommand\nc{\newcommand}
\nc\rnc{\renewcommand}
\nc\Kar{\operatorname{Kar}}
\nc\End{\operatorname{End}}
\nc\modQ {{\mathbb Q}}
\nc\modZ {{\mathbb Z}}
\nc\simeqto{\overset{\simeq}{\longrightarrow }}
\nc\modC {{\mathcal C}}
\nc\modD {{\mathcal D}}
\nc\K{\mathcal {K}}
\nc\CC{\mathbf{C}}
\newcommand{\scs}{\scriptstyle}
\nc\calU{\mathcal{U}}
\nc\cU{\calU}
\nc\col{\colon\thinspace}
\nc\calA{\mathcal{A}}
\nc\Ab{\mathbf{Ab}}
\nc\Ko{K_0}
\nc\TrhorCC{\Tr^{\mathrm{hor}}(\CC)}
\nc\AdCat{\mathbf{AdCat}}
\nc\TrCC{\Tr(\CC)}
\nc\Udot{\dot{\mathcal{U}}}
\nc\diag{\mathrm{d}}
\nc\modU {\mathcal{U}}
\nc\bfU{\mathbf{U}}
\nc\dU{\dot{\mathbf U}}
\nc\dUZ{{_\modZ\dot{\mathbf U}}}
\nc\UZ{{_\modZ \mathbf U} }
\nc\fsl{\mathfrak{sl}}
\nc\Uaa{{\bf U} (\mathfrak{sl}_2\otimes \Q[t,t^{-1}])}
\nc\UZslt{{_\modZ\mathbf{U}} (\mathfrak{sl}_2\otimes \Q[t])}
\nc\UdZslt{{_\modZ\dot{\mathbf{U}}} (\mathfrak{sl}_2\otimes \Q[t])}
\nc\LL{L^+\fsl_2}
\nc\UL{\mathbf U(\LL)}
\nc\UZL{\UZ(\LL)}
\nc\dUZL{\dUZ(\LL)}
\nc\dUL{\dU(\LL)}
\nc{\im}{\rm im}
\nc\Kom{\rm Kom}
\nc\GL{\rm{GL}}
\nc\g{\mathfrak{g}}

\nc\tG{\tilde{G}} \nc\tE{\tilde{E}}
\nc\Vect{\rm Vect}
\nc{\Gras}{{\rm {Gr}}}
\nc\FMod{\rm FMod}
\nc\yto[1]{\underset{#1}{\to}}
\nc\Ear{\yto{E}}

\def\l{\lambda}

\allowdisplaybreaks

\title{
W-algebras from Heisenberg categories
}

\begin{document}
\setcounter{tocdepth}{1}

\author{Sabin Cautis}
\email{cautis@math.ubc.ca}
\address{Department of Mathematics\\ University of British Columbia \\ Vancouver, Canada}

\author{Aaron D. Lauda}
\email{lauda@usc.edu}
\address{Department of Mathematics\\ University of Southern California \\ Los Angeles, CA}

\author{Anthony M. Licata}
\email{anthony.licata@anu.edu.au}
\address{Mathematical Sciences Institute\\ Australian National University \\ Canberra, Australia}

\author{Joshua Sussan}
\email{jsussan@mec.cuny.edu}
\address{Department of Mathematics \\ CUNY Medgar Evers \\ Brooklyn, NY}

% \date{\today}

\begin{abstract}
The trace (or zeroth Hochschild homology) of Khovanov's Heisenberg category is identified with a quotient of the algebra $W_{1+\infty}$. This induces an action of $W_{1+\infty}$ on symmetric functions.
\end{abstract}

\maketitle

\tableofcontents

\section{Introduction}

The algebra $\mbox{Sym}$ of symmetric functions in infinitely many variables is important in classical representation theory in part because it describes the characters of representations of symmetric groups in characteristic 0.  More precisely, there are isomorphisms
$$ \mbox{Sym} \cong \bigoplus_{n=0}^\infty Z(\mathbb{C}[S_n]) \cong \bigoplus_{n=0}^\infty K_0(\mathbb{C}[S_n]\dmod),$$
where $Z(\mathbb{C}[S_n])$ is the center of the group algebra, and $K_0(\mathbb{C}[S_n]\dmod)$ is the Grothendieck group of the category of finite dimensional representations. Since $Z(\mathbb{C}[S_n])$ is isomorphic to the center of the category $\mathbb{C}[S_n]\dmod$, the rightmost isomorphism says that two basic decategorifications of the categories $\mathbb{C}[S_n]\dmod$, the Grothendieck group and the center, are isomorphic to each other.

It is fruitful to consider a structure of interest for symmetric functions and then try to interpret it in the language of symmetric groups. For example, Gessinger realized the Hopf algebra structure on $\mbox{Sym}$ by considering the maps on characters induced by induction and restriction~\cite{G}.  Closely related to this Hopf algebra structure is the fact that $\mbox{Sym}$ is the underlying vector space of the canonical Fock space representation of the Heisenberg algebra $\mathfrak{h}$, where generators of $\mathfrak{h}$ act on the Grothendieck group as the maps induced by induction and restriction between different symmetric groups.  The algebra of symmetric functions $\mbox{Sym}$ thus becomes the canonical level one Fock space module, which we denote by $\mathcal{V}_{1,0}$, for the Heisenberg algebra.

From this point of view, it is also desirable to understand how to categorify the algebra $\mathfrak{h}$ itself (rather than just the representation
$\mathcal{V}_{1,0}$) and seek a monoidal category which acts directly on $\bigoplus_{n=0}^\infty \mathbb{C}[S_n]\dmod$. One definition of such a monoidal category, which is the central object of interest in the current paper, is due to Khovanov \cite{K}.  His Heisenberg category $\H$, which uses a graphical calculus of planar diagrams in the definition of the morphism spaces, has a Grothendieck group which contains the Heisenberg algebra $\mathfrak{h}$.\footnote{Conjecturally, there is an algebra isomorphism $K_0(\H) \cong \mathfrak{h}$.}

The canonical level one module $\mathcal{V}_{1,0}$ of the Heisenberg algebra can be equipped with the action of even larger infinite dimensional algebras.  Much of this symmetry, including the actions of both the Heisenberg and Virasoro algebras, is unified in the action of the W-algebra $\W$ on $\mathcal{V}_{1,0}$.  The Heisenberg and Virasoro actions on $\mathcal{V}_{1,0}$ have already been interpreted in the language of symmetric group characters: the Heisenberg action is essentially determined by the Hopf algebra structure on the polynomial algebra, and is thus implicit in \cite{G} (see also \cite{W} and references therein for closely related constructions using more general finite groups). A Virasoro algebra action on $ \bigoplus_{n=0}^\infty Z(\mathbb{C}[S_n])$ was described explicitly by Frenkel and Wang in \cite{FW}.  The natural question which remains is to extend this description to obtain a character-theoretic description of the action of the entire $\W$ action on $\mathcal{V}_{1,0}$.  Better still would be to obtain a character-theoretic construction of the algebra $\W$ itself, and then as a consequence of such a construction to recover a construction of the canonical module $\mathcal{V}_{1,0}$.  Our main aim in this paper is to give such a construction, and for this aim it is important to consider Khovanov's Heisenberg category $\H$, rather than just the representation categories of all symmetric groups. In this categorical context, the notion of character is captured by the trace $\Tr(\H)$, or  zeroth Hochschild homology. 
Let $ I\subset \W $ the two-sided ideal generated by $ C-1 $ and $w_{0,0} $ for the central elements $C,w_{0,0} \in \W$ (see Section \ref{secwalgebra}).
Our main theorem is the following.

%Such an action, which is a consequence of our main theorem below, has also recently been given in independent work of Shan-Varagnolo-Vasserot \cite{SVV}.

%The main aim in the current paper is to give a character theoretic construction of the algebra $\W$ itself. Then, as a consequence of such a construction, one can then recover a construction of the representation $\mathcal{V}_{1,0}$.  

\begin{theorem}\label{thm:main}
There is an isomorphism $ \Tr(\mathcal{H}) \cong W_{1+\infty}/I $.  The action of $\Tr(\mathcal{H})$ on the center $Z(\mathcal{H})$ is irreducible and faithful, and is thus identified with the canonical level one representation $\mathcal{V}_{1,0}$.
\end{theorem}

The theorem thus gives a graphical construction of both $\W/I$ itself, and of its canonical representation.
As a corollary, any categorical representation of Khovanov's Heisenberg category $\H$ gives rise to a representation of $\W/I$.\footnote{The
quotient of $\W$ we see here is not exactly the same as the quotients of $\W$ that have appeared in closely related literature.  For example, in \cite{SV}, the quotient of $\W$ which appears has $w_{0,0}=\frac{1}{2}$ instead of $w_{0,0}=0$.}
Thus, the categorical representation of $\H$ on $\bigoplus_{n=0}^\infty \mathbb{C}[S_n]\dmod$ also immediately gives rise to an action of $\W$ on $\bigoplus_{n=0}^\infty Z(\mathbb{C}[S_n])$.  Such an action has also recently been given in independent work of Shan, Varagnolo, and Vasserot in \cite{SVV}.   A closely related categorical action of $\H$ on the category of polynomial functors~\cite{HY} allows us to realize a $\W$-algebra action in this context as well. 
%Thus, alternatively, we could have worked directly in the representation of $\H$ on $\bigoplus_{n=0}^\infty \mathbb{C}[S_n]\dmod$ and obtained the action of $\W$ on $\bigoplus_{n=0}^\infty Z(\mathbb{C}[S_n])$, recovering the corresponding result from \cite{SVV}.  However, we prefer here to work directly in the category $\H$ and make use of the graphical description of hom spaces in this category.

In order to prove that $\Tr(\mathcal{H})$ is a quotient of $ \W$ we follow the proof in ~\cite{SV} of the isomorphism between a limit of spherical degenerate double affine Hecke algebras and a quotient of $ \W $. A complete set of generators and relations of the limit of spherical degenerate double affine Hecke algebras was given in ~\cite{AS}.  Composing the isomorphism in ~\cite{SV} with the isomorphism from Theorem \ref{thm:main}, we get a complete generators and relations description of $ \Tr(\mathcal{H})$.
Crucial in this comparison is the relationship between $ \mathcal{H} $ and the degenerate affine Hecke algebra $DH_n$.  Roughly speaking, $DH_n$ appears in the morphism spaces of in the upper half of $\mathcal{H}$. As a vector space $\Tr(DH_n)$ was determined by Solleveld ~\cite{S} and this calculation plays a role in our computation of $\Tr(\H)$.

From the point of view of Theorem \ref{thm:main}, it is perhaps equally reasonable to refer to $\H$ as a categorification of the $W$-algebra; the difference between this statement and the statement that $\H$ categorifies the Heisenberg algebra is that we have made a different choice in which decategorification procedure to invert.
It is interesting that the two natural decategorifications of the category $\H$ -- the Grothendieck group and the trace -- are quite different from one another.
This is in contrast to the situation for the category $\mathbb{C}[S_n]\dmod$, for example, where the Grothendieck group and the trace are isomorphic.  Our analysis shows that the trace decategorication can have structure missing from the Grothendieck group decategorification, since the algebra $\W$ contains the Heisenberg algebra and a lot more.   In addition to being able to directly identify $\Tr(\H)$ with $\W/I$, in the trace decategorification we are also able to find natural graphical realizations of the standard Heisenberg algebra generators.  By contrast, in the Grothendieck decategorification the standard Heisenberg generators do not admit an obvious categorical lift and therefore an alternative presentation for the Heisenberg algebra is required.  We should also mention that a completely different categorification problem -- to construct a monoidal category whose Grothendieck group is the algebra $\W$-- is open, though the Theorem \ref{thm:main} suggests that such a category might be related to Khovanov's category $\H$.

An analog of Theorem \ref{thm:main} (and of this difference between the Grothendieck group and the trace) also arises in the context of quantum groups associated to type ADE Kac-Moody algebras.  Associated to an arbitrary symmetrizable Kac-Moody algebras $\mf{g}$, a 2-category $\Ucat(\mf{g})$ was defined in \cite{KL} whose the split Grothendieck group is isomorphic to the integral form of Lusztig's indepotented form of the quantum group $\mathbf{U}_q(\mf{g})$. However, the trace of the 2-category $\Ucat(\mf{g})$ is then identified with the idempotented form of quantum current algebra $\mathbf{U}_q(\mf{g}[t])$~\cite{SVV,BHLZ,BHLW}.

\subsection{Further directions}

More generally, one can associate a Heisenberg category $\H_F$ to any finite dimensional Frobenius algebra $F$. This is the categorical analogue of the fact that one can associate a Heisenberg algebra ${\mathfrak{h}}_L$ to any $\mathbb{Z}$-lattice $L$. If one takes the simplest Frobenius algebra, namely the one dimensional algebra ${\mathbb C}$, then one recovers $\H$, and for this reason we first study the trace of $\H$. In future work we plan to compute the traces of these more general Heisenberg categories, and thus associate a W-algebra $W_F$ to any Frobenius algebra $F$.  Of particular interest are the Heisenberg categories studied in ~\cite{CL1} and associated to zig-zag algebras (these categorify the Heisenberg algebras associated to quantum lattices of type $A,D,E$). Via categorical vertex operators \cite{CL2,C} these Heisenberg categories were related to categorified quantum groups. So one might expect their trace to be related to \cite{BGHL,BHLZ} or to W-algebras associated to quantum groups \cite{FR}.  Finally, one may also consider the trace on the categorification of twisted Heisenberg algebras such as the twisted version of Khovanov's category from \cite{CS}. In this particular case the trace should be related to the twisted $\W$ algebra from ~\cite{KWY}.

In the recent work \cite{MS}, the algebra $\W$ appears in relation to the skein module of the torus.  It could also be interesting to relate directly the appearance of $\W$ in Heisenberg categorification to this skein module or to other invariants in quantum topology.

\subsection{Outline}
In Section ~\ref{hei} Khovanov's Heisenberg algebra $ \mathfrak{h} $ and its associated category $\H$ are introduced and its basic properties are reviewed.  In Section ~\ref{secwalgebra} the algebra $\W$ is defined, and some of its important features are recalled.  In Section ~\ref{secHH} the trace of $\H$ is determined in terms of $\W$.

\subsection*{Acknowledgements}
The authors would like to thank Peter Bouwknegt,  Mikhail Khovanov, Olivier Schiffmann, Davesh Maulik, and Weiqiang Wang for helpful conversations. S.C. was supported by an NSERC discovery grant and the John Templeton Foundation. A.D.L. was partially supported by NSF grant DMS-1255334 and by the John Templeton Foundation. A.M.L. was supported by an Australian Research Council Discovery Early Career fellowship. J.S. was supported by NSF grant DMS-1407394 and PSC-CUNY Award 67144-0045.

\section{The Heisenberg category $\H$}\label{hei}

\subsection{The Heisenberg algebra and Fock space}

Let $ \mathfrak{h} $ be the associative algebra over $\mathbb{C}$ with generators $ h_n$ for $ n \in \mathbb{Z}-\{0\}$
with relations $[h_m, h_n]=m \delta_{m,-n}$.

Another way to present $ \mathfrak{h}$ is as follows.  Set $ p^{(n)}=q^{(n)}=0$ for $ n<0$.
For $ n \geq 0 $ define $ p^{(n)} $ and $ q^{(n)} $ by
\begin{equation*}
\sum_{n \geq 0} p^{(n)} z^n = exp(\sum_{n \geq 1} \frac{h_{-n}}{n}z^n) \hspace{.5in}
\sum_{n \geq 0} q^{(n)} z^n = exp(\sum_{n \geq 1} \frac{h_{n}}{n}z^n).
\end{equation*}
In these generators the relations are:
\begin{align*}
p^{(n)} p^{(m)} &= p^{(m)} p^{(n)} \\
q^{(n)} q^{(m)} &= q^{(m)} q^{(n)} \\
q^{(n)}p^{(m)} &=\sum_{k \geq 0} p^{(m-k)} q^{(n-k)}.
\end{align*}

The Heisenberg algebra $\mathfrak{h}$ has a natural representation $\mathcal{F}$, known as the Fock space. Let $\mathfrak{h}^+ \subset \mathfrak{h}$ denote the subalgebra generated by the $q^{(n)}$ for $n \ge 0$. Let $\mbox{triv}_0$ denote the trivial representation of $\mathfrak{h}^+$, where all $q^{(n)}$ ($n>0$) act by zero. Then
$$\mathcal{F} := \Ind_{\mathfrak{h}^+}^{\mathfrak{h}}(\mbox{triv}_0)$$
is called the Fock space representation of $\mathfrak{h}$. It inherits a $\Z$ grading $\mathcal{F} = \oplus_{m \in \N} \mathcal{F}(m)$ by declaring $\mbox{triv}_0$ to have degree zero, $p^{(n)}$ degree $n$ and $q^{(n)}$ degree $-n$.

\subsection{Heisenberg category $\H$}\label{sec:hei}
In \cite{K} Khovanov introduced a categorical framework for the Heisenberg algebra $\mathfrak{h}$. This framework consists of a monoidal category $\H$ which is the Karoubi envelope of a monoidal category $\H'$ whose definition we now sketch (see \cite{K} for more details).

The category $\H'$ is generated by objects $\P$ and $\Q$. These can be denoted by an upward pointing strand and a downward pointing strand. Monoidal composition of such objects is then given by sideways concatenation of diagrams. The space of morphisms between products of $\P$'s and $\Q$'s is a $\mathbb{C}$-algebra described by certain string diagrams with relations. By convention, composition of morphisms is done vertically from the bottom and going up.

The morphisms are generated by crossings, caps and cups as shown below

\begin{equation}\label{eq:maps}
\begin{tikzpicture}[scale=.75]

\draw [->](0,0) -- (1,1) [thick];

\draw [->](1,0) -- (0,1) [thick];

\draw (3,.5) arc (180:360:.5)[->] [thick];

\draw (6.5,.25) arc (0:180:.5) [->][thick];

\draw (9,.5) arc (180:360:.5)[<-] [thick];

\draw (12.5,.25) arc (0:180:.5) [<-][thick];
\end{tikzpicture}
\end{equation}

Thus, for instance, the left crossing is a map in $\End(\P\P)$ while the right cap is a map $\P \Q \rightarrow \id$. These morphisms satisfy the following relations

\begin{equation}\label{eq:rel1}
\begin{tikzpicture}

\draw [shift={+(7,0)}](0,0) .. controls (1,1) .. (0,2)[->][thick] ;

\draw [shift={+(7,0)}](1,0) .. controls (0,1) .. (1,2)[->] [thick];

\draw [shift={+(7,0)}](1.5,1) node {=};

\draw [shift={+(7,0)}](2,0) --(2,2)[->][thick] ;

\draw [shift={+(7,0)}](3,0) -- (3,2)[->][thick] ;

\draw (0,0) -- (2,2)[->][thick];

\draw (2,0) -- (0,2)[->][thick];

\draw (1,0) .. controls (0,1) .. (1,2)[->][thick];

\draw (2.5,1) node {=};

\draw (3,0) -- (5,2)[->][thick];

\draw (5,0) -- (3,2)[->][thick];

\draw (4,0) .. controls (5,1) .. (4,2)[->][thick];
\end{tikzpicture}
\end{equation}

\begin{equation}\label{eq:rel2}
\begin{tikzpicture}
\draw (0,0) .. controls (1,1) .. (0,2)[<-][thick];

\draw (1,0) .. controls (0,1) .. (1,2)[->] [thick];

\draw (1.5,1) node {=};

\draw (2,0) --(2,2)[<-][thick];

\draw (3,0) -- (3,2)[->][thick];

\draw (3.8,1) node{$-$};

\draw (4,1.75) arc (180:360:.5) [thick];

\draw (4,2) -- (4,1.75) [thick];

\draw (5,2) -- (5,1.75) [thick][<-];

\draw (5,.25) arc (0:180:.5) [thick];

\draw (5,0) -- (5,.25) [thick];

\draw (4,0) -- (4,.25) [thick][<-];

\draw [shift={+(7,0)}](0,0) .. controls (1,1) .. (0,2)[->][thick];

\draw [shift={+(7,0)}](1,0) .. controls (0,1) .. (1,2)[<-] [thick];

\draw [shift={+(7,0)}](1.5,1) node {=};

\draw [shift={+(7,0)}](2,0) --(2,2)[->][thick];

\draw [shift={+(7,0)}](3,0) -- (3,2)[<-][thick];
\end{tikzpicture}
\end{equation}

\begin{equation}\label{eq:rel3}
\begin{tikzpicture}

\draw [shift={+(0,0)}](0,0) arc (180:360:0.5cm) [thick];

\draw [shift={+(0,0)}][->](1,0) arc (0:180:0.5cm) [thick];

\draw [shift={+(0,0)}](1.75,0) node{$= 1.$};

\draw  [shift={+(5,0)}](0,0) .. controls (0,.5) and (.7,.5) .. (.9,0) [thick];

\draw  [shift={+(5,0)}](0,0) .. controls (0,-.5) and (.7,-.5) .. (.9,0) [thick];

\draw  [shift={+(5,0)}](1,-1) .. controls (1,-.5) .. (.9,0) [thick];

\draw  [shift={+(5,0)}](.9,0) .. controls (1,.5) .. (1,1) [->] [thick];

\draw  [shift={+(5,0)}](1.5,0) node {$=$};

\draw  [shift={+(5,0)}](2,0) node {$0.$};
\end{tikzpicture}
\end{equation}

Moreover, two morphisms which differ by planar isotopies are equal. Relation (\ref{eq:rel1}) implies that there is a map $\mathbb{C}[S_n] \rightarrow \End(\P^n)$. Since $\H$ is assumed to be idempotent complete this means that we also get objects $\P^{(\l)}$, for any partition $\l \vdash n$, associated with the corresponding minimal idempotent $e_\l \in \mathbb{C}[S_n]$. Likewise we also have $\Q^{(\l)}$ for any $\l \vdash n$. We will denote by $(m)$ and $(1^m)$ the unique one-part and $m$-part partitions of $m$.

Let $ \H^{\geq} $ and $ \H^{\leq} $ be the full subcategories of $ \H$ generated by $ \P $ and $ \Q$ respectively.
Let $ \H^{0} $ be the full subcategory of of $ \H$ generated by the identity object $\Id$.

\begin{theorem}
[\cite{K}] \label{thm:1}

Inside $\H$ we have the following relations

\begin{enumerate}

\item $\P^{(\l)}$ and $\P^{(\mu)}$ commute for any partitions $\l,\mu$,

\item $\Q^{(\l)}$ and $\Q^{(\mu)}$ commute for any partitions $\l,\mu$,

\item $\Q^{(n)} \P^{(m)} \cong \bigoplus_{k \ge 0} \P^{(m-k)} \Q^{(n-k)}$ and $\Q^{(1^n)} \P^{(1^m)} \cong \bigoplus_{k \ge 0} \P^{(1^{m-k})} \Q^{(1^{n-k})}$,

\item $\Q^{(n)} \P^{(1^m)} \cong \P^{(1^m)} \Q^{(n)} \oplus \P^{(1^{m-1})} \Q^{(n-1)}$ and $\Q^{(1^n)} \P^{(m)} \cong \P^{(m)} \Q^{(1^n)} \oplus \P^{({m-1})} \Q^{(1^{n-1})}$.
\end{enumerate}
\end{theorem}

Thus, at the level of Grothendieck groups we have a map $\mathfrak{h} \rightarrow K_0(\H)$. This map is known to be injective but it is not known if it is surjective.

\subsection{The degenerate affine Hecke algebra}

Inside $\H$ consider the element $X_i \in \End(\P^n)$, acting on the $i$th factor $\P$, as illustrated in the left hand side of (\ref{eq:curl}). In \cite{K} this element was studied and it was encoded diagrammatically by a solid dot, as shown %on the right side of (\ref{eq:curl}).

\begin{equation}\label{eq:curl}
\begin{tikzpicture}
\draw  (1.9,0) .. controls (1.9,.5) and (1.3,.5) .. (1.1,0) [thick];
\draw  (1.9,0) .. controls (1.9,-.5) and (1.3,-.5) .. (1.1,0) [thick];
%\draw  (1.9,0) .. controls (1.7,-.5) and (1,-.5) .. (1,0) [thick];
\draw  (1,-1) .. controls (1,-.5) .. (1.1,0) [thick];
\draw  (1.1,0) .. controls (1,.5) .. (1,1) [->] [thick];
\draw  (2.5,0) node {$=$};
\draw (3,-1) -- (3,1) [thick][->];
\filldraw [black] (3,0) circle (2pt);
\end{tikzpicture}
\end{equation}

In \cite{K} it was shown that these $X_i$'s together with the symmetric group $\mathbb{C}[S_n] \subset \End(\P^n)$ generate a copy of the degenerate affine Hecke algebra.  In particular, using equation \eqref{eq:rel1}-- \eqref{eq:rel3} the equations
\begin{align} \label{eq:nil-dot}
\hackcenter{\begin{tikzpicture}[scale=0.8]
    \draw[thick, ->] (0,0) .. controls (0,.75) and (.75,.75) .. (.75,1.5)
        node[pos=.25, shape=coordinate](DOT){};
    \draw[thick, ->] (.75,0) .. controls (.75,.75) and (0,.75) .. (0,1.5);
    \filldraw  (DOT) circle (2.5pt);
\end{tikzpicture}}
\quad-\quad
\hackcenter{\begin{tikzpicture}[scale=0.8]
    \draw[thick, ->] (0,0) .. controls (0,.75) and (.75,.75) .. (.75,1.5)
        node[pos=.75, shape=coordinate](DOT){};
    \draw[thick, ->] (.75,0) .. controls (.75,.75) and (0,.75) .. (0,1.5);
    \filldraw  (DOT) circle (2.5pt);
\end{tikzpicture}}
&\quad=\quad
\hackcenter{\begin{tikzpicture}[scale=0.8]
    \draw[thick, ->] (0,0) -- (0,1.5);
    \draw[thick, ->] (.75,0) -- (.75,1.5);
\end{tikzpicture}}
\quad=\quad
\hackcenter{\begin{tikzpicture}[scale=0.8]
    \draw[thick, ->] (0,0) .. controls (0,.75) and (.75,.75) .. (.75,1.5);
    \draw[thick, ->] (.75,0) .. controls (.75,.75) and (0,.75) .. (0,1.5)
        node[pos=.75, shape=coordinate](DOT){};
    \filldraw  (DOT) circle (2.5pt);
\end{tikzpicture}}
\quad-\quad
\hackcenter{\begin{tikzpicture}[scale=0.8]
    \draw[thick, ->] (0,0) .. controls (0,.75) and (.75,.75) .. (.75,1.5);
    \draw[thick, ->] (.75,0) .. controls (.75,.75) and (0,.75) .. (0,1.5)
        node[pos=.25, shape=coordinate](DOT){};
      \filldraw  (DOT) circle (2.5pt);
\end{tikzpicture}}
\end{align}
follow. More precisely, denote a crossing of the $i$th and $(i+1)$st strands by $T_i$.

\begin{proposition}\cite{K}\label{prop:sergeevacting}
We have the following relations inside $\End_{\H}(\P^n)$:
\begin{align*}
T_i X_i &= X_{i+1}T_i + 1  \\
X_i T_i &= T_i X_{i+1} + 1 \\
X_i X_j &= X_j X_i \\
T_i^2 &=1 \\
T_i T_j &= T_j T_i \text{ if } |i-j|>1 \\
T_i T_{i+1} T_i &= T_{i+1} T_i T_{i+1}.
\end{align*}
\end{proposition}

The algebra generated by $ T_i $ for $i=1, \ldots, n-1$ and $ X_i $ for $ i=1,\ldots,n$ satisfying the relations in Proposition ~\ref{prop:sergeevacting} is the degenerate affine Hecke algebra $DH_n$.

We now define bubbles which are also endomorphisms of $\P^n$.

\begin{equation}\label{eq:bubbles}
\begin{tikzpicture}

\draw  [shift={+(0,0)}](-1,0) node {$c_n=$};
\draw  [shift={+(0,0)}](.2,.125) node {$n$};
\filldraw [shift={+(0,0)}][black] (0,.125) circle (2pt);
\draw [shift={+(0,0)}](0,0) arc (180:360:0.5cm) [thick];
\draw [shift={+(0,0)}][<-](1,0) arc (0:180:0.5cm) [thick];
\draw  [shift={+(5,0)}](-1,0) node {$\tilde{c}_n=$};
\draw  [shift={+(5,0)}](.2,.125) node {$n$};
\filldraw [shift={+(5,0)}][black] (0,.125) circle (2pt);
\draw [shift={+(5,0)}][->](0,0) arc (180:360:0.5cm) [thick];
\draw [shift={+(5,0)}][](1,0) arc (0:180:0.5cm) [thick];
\end{tikzpicture}
\end{equation}

\begin{proposition}\label{eq:bubblerelationship}\cite[Proposition 2]{K}
For $n>0$
\begin{equation*}
\tilde{c}_{n+1} = \sum_{i=0}^{n-1} \tilde{c}_i c_{n-1-i}.
\end{equation*}
\end{proposition}

\begin{proposition}
\label{KhovEndThm}
\cite[Proposition 4]{K}
There is an isomorphism of algebras
\begin{equation*}
\End(\P^n) \cong DH_n \otimes \mathbb{C}[c_0, c_1, \ldots].
\end{equation*}
In particular, when $n=0$ we have $\End(\id) \cong \mathbb{C}[c_0, c_1, \ldots]$.
\end{proposition}

Let $J_{m,n} $ be the ideal of $\End_{\H'}(\P^m\Q^n)$ generated by diagrams which contain at least one arc connecting a pair of upper points.

\begin{proposition}
\label{KhovEndThmPmQn}
\cite[Equation (19)]{K}
There exists a short exact sequence
\begin{equation*}
0 \rightarrow J_{m,n} \rightarrow \End_{\H'}(\P^m\Q^n) \rightarrow DH_m \otimes DH_n^{op} \otimes \mathbb{C}[c_0,c_1,\ldots] \rightarrow 0.
\end{equation*}
Furthermore, this sequence splits.
\end{proposition}

\section{The algebra $\W$}
\label{secwalgebra}
In this section we define the $W$-algebra of interest and list several important properties.
This algebra is defined to be the universal enveloping algebra of a central extension of differential operators on $\mathbb{C}^*$.
The $W$-algebra, $\W$ comes equipped with a $\Z$-grading and a compatible $\Z_{\geq 0}$-filtration.
We will also define a faithful irreducible representation of a quotient of $\W$.  Both of these features will be crucial in proving that the $W$-algebra is related to the trace of $\H$.
This section follows closely the exposition in ~\cite[Appendix F]{SV}.

Let $ \W$ be the $\mathbb{C}$-associative algebra generated by $w_{l,k}$ for $ l \in \Z $ and $ k \in \N$ and $C$ with relations
that $w_{0,0}$ and $C$ are central, and
\begin{equation}
w_{l,k}=t^l D^k \hspace{.5in} (l,k) \neq (0,0)
\end{equation}
\begin{equation}
\label{wdef}
[t^l exp(\alpha D), t^k exp(\beta D)]=(exp(k \alpha)- exp(l \beta))t^{l+k} exp(\alpha D + \beta D)
+ \delta_{l,-k} \frac{exp(-l \alpha)-exp(-k \beta)}{1-exp(\alpha+\beta)}C
\end{equation}
where $t$ is the parameter on ${\mathbb C}^*$ and $D = t \partial_t$.

Note that for $ l \in \mathbb{Z} - \{0 \} $ the set $ \{w_{l,0}  \} $ generates a Heisenberg subalgebra because
\begin{equation}
\label{Walgheisrelations}
[w_{l,0}, w_{k,0}]=l \delta_{l,-k}C.
\end{equation}
%For $ l \in \mathbb{Z} $ the set $ \{w_{l,1}  \} $ generates a Virasoro subalgebra because
%\begin{equation*}
%[w_{l,1}, w_{k,1}]=(k-l)w_{l+k,1} + \frac{k^3-k}{6} C \delta_{l,-k}.
%\end{equation*}

We will need the following relations which are direct consequences of ~\eqref{wdef}.
\begin{align}
[w_{l,1}, w_{k,1}]&=(k-l)w_{l+k,1} + \frac{k^3-k}{6} C \delta_{l,-k} \label{Walgl1k1} \\
[w_{-1,a},w_{1,b}]&=\sum_{r=1}^a \binom{a}{r}w_{0,a+b-r} - \sum_{s=1}^b (-1)^s \binom{b}{s}w_{0,a+b-s} +
\delta_{a,0}(-1)^{b+1}C \label{Walg-1a1b} \\
[w_{l,1},w_{k,0}]&=k w_{l+k,0} - \delta_{l,-k}\frac{k(k-1)}{2}C \label{Walgl1k0} \\
[w_{l,0},w_{0,2}]&=-2lw_{l,1} -l^2 w_{l,0}. \label{Walgl002}
\end{align}
We define Virasoro elements
\begin{equation*}
\bar{L}_l = -w_{l,1}-\frac{1}{2}(l+1)w_{l,0}.
\end{equation*}
Then it is easy to check that ~\eqref{wdef} gives the Virasoro relations
\begin{equation*}
[\bar{L}_l, \bar{L}_k]=(l-k)\bar{L}_{l+k} + \delta_{l,-k}\frac{l^3-l}{12}C.
\end{equation*}

Let $  \W^{>}, \W^0, \W^{<} $ be subalgebras of $ \W$ generated as follows:
\begin{align*}
%\W^{\geq} &:= \mathbb{C} \la C, w_{l,k} ; l,k \geq 0 \ra \\
\W^{>} &:= \mathbb{C} \la w_{l,k} ; l \geq 1, k \geq 0 \ra \\
\W^{0} &:= \mathbb{C} \la C, w_{0,l}; l \geq 0\} \ra \\
%\W^{\leq} &:= \mathbb{C} \la C, w_{-l,k} ; l,k \geq 0 \ra \\
\W^{<} &:= \mathbb{C} \la w_{-l,k} ; l \geq 1, k \geq 0 \ra.
\end{align*}

This algebra has a $ \Z$-grading called the rank grading where $w_{l,k}$ is in degree $l$ and $C$ is in degree $0$.  The algebra is also $\Z_{\geq 0}$-filtered where $ w_{l,k}$ is in degree $ \leq k$.  Let $ \W^{\omega}[r, \leq k] $ denote the set of elements in $\Z$-degree $r$ and $ \Z_{\geq 0}$-degree $k$ where $ \omega \in \{ <, >, 0, \emptyset \}$.

Denote the associated graded algebra of $ \W$ with respect to the $\Z_{\geq 0}$-filtration by
\begin{equation*}
\tW = gr(\W).
\end{equation*}
Since the $\Z$-grading on $\W$ is compatible with the $\Z_{\geq 0}$-filtering, the associated graded $ \tW$ is $ (\Z \times \Z_{\geq 0})$-graded.
Let $ \tW^{\omega}[r, k] $ denote the subspace of $ \tW^{\omega} $ in bidegree $(r,k)$.
Define a generating series for the graded dimension of $ \tW^{\omega} $ by
\begin{equation*}
P_{\tW^{\omega}}(t,q)=\sum_{r \in \Z} \sum_{k \in \Z_{\geq 0}} \dim \tW^{\omega}[r, k] t^r q^k.
\end{equation*}

\begin{proposition}
\label{poincareW}
The graded dimensions of $ \tW^>$ and $\tW^<$ are given by:
\begin{equation*}
P_{\tW^>}=\prod_{r>0} \prod_{k \geq 0} \frac{1}{1-t^r q^k}, \hspace{.3in}
P_{\tW^<}=\prod_{r<0} \prod_{k \geq 0} \frac{1}{1-t^r q^k}.
\end{equation*}
\end{proposition}

\begin{proof}
The defining relations of $\W$ given in ~\ref{wdef} imply that the associated graded algebras $\tW^>$ and $\tW^<$ are freely generated by the images of $w_{l,k}$.  The proposition now follows easily since the image of $w_{l,k}$ has bidegree $(l,k)$.
\end{proof}

\begin{lemma}
\label{lemmaf.5sv}
\cite[Lemma F.5]{SV}
\begin{enumerate}
\item $\W$ is generated by $w_{-1,0}, w_{1,0}, w_{0,2}$.
\item $\W^>$ is generated by $w_{1,l}$ for $ l \geq 0$.
\item $\W^<$ is generated by $w_{-1,l}$ for $ l \geq 0$.
\end{enumerate}
\end{lemma}

For $ c,d \in \mathbb{C}$ let $\mathbb{C}_{c,d}$ be a one dimensional representation of $ \W^{\geq} $ where $w_{k,l}$ acts by
zero for $(k,l)\neq (0,0)$, $C$ acts by $c$ and $ w_{0,0} $ acts by $d$.

Let
$\mathcal{M}_{c,d} := \Ind_{\W^{\geq}}^{\W}(\mathbb{C}_{c,d})$.

\begin{proposition} \label{formofV10} \cite{AFMO, FKRW}
The induced module $ \mathcal{M}_{c,d} $ has a unique irreducible quotient $ \mathcal{V}_{c,d} $.
As a vector space $ \mathcal{V}_{1,0} $ is isomorphic to $ \mathbb{C}[w_{-1,0}, w_{-2,0}, \ldots] $.
\end{proposition}

\begin{proposition} \cite[Proposition F.6]{SV}
The action of $ \W/(C-1,w_{0,0}) $ is faithful on $ \mathcal{V}_{1,0} $.
\end{proposition}
\begin{proof}
This is Proposition F.6 of ~\cite{SV} since their proof does not depend on how the central elements $C$ and $w_{0,0}$ act.  The proof uses the faithfulness of the action of the Heisenberg algebra on Fock space together with the filtration on $ \W$.
\end{proof}

\begin{lemma}
\label{nablaDAHA}
\cite[Lemma F.8]{SV}
The action of $ w_{0,2} $ on $ \mathcal{V}_{1,0} $ is given by the operator
\begin{equation*}
\sum_{k,l > 0} (w_{-l,0} w_{-k,0} w_{k+l,0} + w_{-l-k,0} w_{l,0} w_{k,0})-w_{0,1}.
\end{equation*}
\end{lemma}

\begin{proof}
This is a straightforward computation using the relations of $\W$ and the definition of $ \mathcal{V}_{1,0}$.
For more details see the proof of \cite[Lemma F.8]{SV}.  There is a correction term of $-w_{0,1}$ in the lemma because we take $w_{0,0}=0$ while in ~\cite{SV} it is taken to be $-\frac{1}{2}$.
\end{proof}

\section{Trace of $\H$} \label{secHH}

\subsection{Definitions and conventions} \label{HHconventions}

The trace, or zeroth Hochschild homology, $\Tr(\modC )$ of a $\Bbbk$-linear category $\modC$ is the $\Bbbk$-vector space given by
\begin{gather*}
  \Tr(\C )=
\left( \bigoplus_{X\in \Ob(\modC )}\modC (X,X) \right)/\Span_\k\{fg-gf\},
\end{gather*}
where $\modC(X,X)=\End_\C(X)$, $f$ and $g$ run through all pairs of morphisms $f\col X\to Y$, $g\col Y\to X$ with $X,Y\in \Ob(\modC )$.  The trace is invariant under passage to the Karoubi envelope $Kar(\modC)$.
\begin{proposition}
\label{HH0HequalsHH0Hprime}
\cite[Proposition 3.2]{BHLZ}
The natural map $ \Tr(Kar(\modC)) \rightarrow \Tr(\modC) $ induced by inclusion of categories is an isomorphism.
\end{proposition}

\begin{proposition}\cite[Lemma 2.1]{BHLW}
  \label{prop:indecomposables}
  Let $\C$ be a $\k$-linear additive category.
  Let $S\subset\Ob(\C)$ be a subset such that every object in $\C$ is
  isomorphic to the direct sum of finitely many copies of objects in
  $S$.  Let $\C|_S$ denote the full subcategory of $\C$ with
  $\Ob(\C|_S)=S$.  Then, the inclusion functor $\C|_S\to \C$ induces
  an isomorphism
  \begin{gather}
    \label{e9}
    \Tr(\modC|_S)\cong\Tr(\modC)
  \end{gather}
\end{proposition}

Recall that a categorical representation of the category $\modC$  is a $\Bbbk$-linear category $\cal{V}$ and a $\Bbbk$-linear functor $\modC \to \End(\cal{V})$ sending each object $X$ to an endofunctor $F(X)\maps \cal{V} \to \cal{V}$.  Each morphism $f \maps X \to Y$ is sent to a natural transformation of functors $F(f)\maps F(X) \to F(Y)$. A categorical representation gives rise to a representation
\begin{equation}
  \rho_F\maps  \Tr(\modC) \to \cat{Vect}_{\Bbbk}
\end{equation}
sending the class $[f]\in \Tr(\modC)$ corresponding to a map $f\maps X\to X$ in $\modC$ to the endomorphism of the $\Bbbk$-vector space $\Tr(\cal{V})$ defined by sending $\psi\maps U \to U$ in $\Tr(\cal{V})$ to
\begin{gather*}
  \rho_F([f])([\psi])
  =[F(X)(\psi)\circ F(f)_U]
  =[F(f)_U\circ F(X)(\psi)].
\end{gather*}
Here, note that $F(f)_U$ denotes the component of the natural transformation  $F(f) \maps
F(X) \To F(X) \maps\cal{V} \to \cal{V}$ corresponding to the object $U \in \Ob(\cal{V})$.  By naturality this map is well defined.

For $ f \in \End(\mathcal{H}) $ denote its class in $ \Tr $ by $ [f]$.
We view the element $[f] \in \Tr $ by drawing $ f $ in an annulus and then closing up the diagram for $f$.
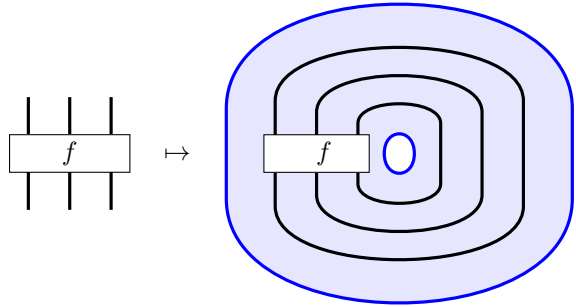
\[
\hackcenter{\begin{tikzpicture}
    \draw[very thick] (-.55,0) -- (-.55,1.5);
    \draw[very thick] (0,0) -- (0,1.5);
    \draw[very thick] (.55,0) -- (.55,1.5);
    \draw[fill=white!20,] (-.8,.5) rectangle (.8,1);
    \node () at (0,.75) {$f$};
\end{tikzpicture}}
\quad \mapsto
\quad
\hackcenter{\begin{tikzpicture}
      \path[draw,blue, very thick, fill=blue!10]
        (-2.3,-.6) to (-2.3,.6) .. controls ++(0,1.85) and ++(0,1.85) .. (2.3,.6)
         to (2.3,-.6)  .. controls ++(0,-1.85) and ++(0,-1.85) .. (-2.3,-.6);
        \path[draw, blue, very thick, fill=white]
            (-0.2,0) .. controls ++(0,.35) and ++(0,.35) .. (0.2,0)
            .. controls ++(0,-.35) and ++(0,-.35) .. (-0.2,0);
    \draw[very thick] (-1.65,-.7) -- (-1.65, .7).. controls ++(0,.95) and ++(0,.95) .. (1.65,.7)
        to (1.65,-.7) .. controls ++(0,-.95) and ++(0,-.95) .. (-1.65,-.7);
    \draw[very thick] (-1.1,-.55) -- (-1.1,.55) .. controls ++(0,.65) and ++(0,.65) .. (1.1,.55)
        to (1.1,-.55) .. controls ++(0,-.65) and ++(0,-.65) .. (-1.1, -.55);
    \draw[very thick] (-.55,-.4) -- (-.55,.4) .. controls ++(0,.35) and ++(0,.35) .. (.55,.4)
        to (.55, -.4) .. controls ++(0,-.35) and ++(0,-.35) .. (-.55,-.4);
    \draw[fill=white!20,] (-1.8,-.25) rectangle (-.4,.25);
    \node () at (-1,0) {$f$};
\end{tikzpicture}}
 \]

Let $ w \in S_n $
Define
\begin{equation*}
h_{w} \otimes (x_1^{j_1} \cdots x_n^{j_n}) := [f_{w; j_1, \ldots, j_n}]
\end{equation*}
where $ f_{w; j_1, \ldots, j_n} \in \End(\P^n) $ is given by:

\[
\hackcenter{\begin{tikzpicture}
\node () at (-2,.75) {$f_{w; j_1, \ldots, j_n}:=$};
    \draw[very thick][->] (-.55,0) -- (-.55,1.5);
    \draw[very thick][->] (.55,0) -- (.55,1.5);
    \draw[fill=white!20,] (-.8,.5) rectangle (.8,1);
    \node () at (0,.75) {$w$};
    \node () at (0,1.25) {$\cdots$};
    \node () at (0,.25) {$\cdots$};
\filldraw  (-.55,1.25) circle (2pt);
\filldraw  (.55,1.25) circle (2pt);
\draw (-.75,1.25) node {$j_1$};
\draw (.8,1.25) node {$j_n$};
\end{tikzpicture}}
\]

Define
\begin{equation*}
h_{-w} \otimes (x_1^{j_1} \cdots x_n^{j_n}) := [f_{-w; j_1, \ldots, j_n}]
\end{equation*}
where $ f_{-w; j_1, \ldots, j_n} \in \End(\Q^n) $ is given by:

\[
\hackcenter{\begin{tikzpicture}
\node () at (-2,.75) {$f_{-w; j_1, \ldots, j_n}:=$};
    \draw[very thick][<-] (-.55,0) -- (-.55,1.5);
    \draw[very thick][<-] (.55,0) -- (.55,1.5);
    \draw[fill=white!20,] (-.8,.5) rectangle (.8,1);
    \node () at (0,.75) {$w$};
    \node () at (0,1.25) {$\cdots$};
    \node () at (0,.25) {$\cdots$};
\filldraw  (-.55,1.25) circle (2pt);
\filldraw  (.55,1.25) circle (2pt);
\draw (-.8,1.25) node {$j_1$};
\draw (.8,1.25) node {$j_n$};
\end{tikzpicture}}
\]

Of particular importance is the element $ \sigma_n = s_1 \cdots s_{n-1} \in S_n$.
We then abbreviate
\begin{align}
h_n \otimes (x_1^{j_1} \cdots x_n^{j_n}) &:= h_{\sigma_n} \otimes (x_1^{j_1} \cdots x_n^{j_n})
= \left[ \;\hackcenter{
\begin{tikzpicture}[scale=0.8]
 %% Separate lines by 0.8
  \draw[thick,->] (3.2,0) .. controls (3.2,1.25) and (0,.25) .. (0,2)
     node[pos=0.85, shape=coordinate](X){};
  \draw[thick,->] (0,0) .. controls (0,1) and (.8,.8) .. (.8,2);
  \draw[thick,->] (.8,0) .. controls (.8,1) and (1.6,.8) .. (1.6,2);
  \draw[thick,->] (2.4,0) .. controls (2.4,1) and (3.2,.8) .. (3.2,2);
  \node at (1.6,.35) {$\dots$};
  \node at (2.4,1.65) {$\dots$};
  %\filldraw  (X) circle (2pt);
  \filldraw  (.05,1.6) circle (2pt);
  \filldraw  (.78,1.6) circle (2pt);  % I put the dot .02 to the left to make it look nice
  \filldraw  (1.58,1.6) circle (2pt);
  \filldraw  (3.18,1.6) circle (2pt);
  %% I cant remember a better way of doing this
  \node at (-0.18,1.7) {$\scs j_1$};
  \node at (.55,1.7) {$\scs j_2$};
  \node at (1.35,1.7) {$\scs j_3$};
   \node at (3.45,1.7) {$\scs j_n$};
\end{tikzpicture}}
\; \right]
\\
h_{-n} \otimes (x_1^{j_1} \cdots x_n^{j_n}) &:= h_{-\sigma_n} \otimes (x_1^{j_1} \cdots x_n^{j_n}) =
\left[ \;\hackcenter{
\begin{tikzpicture}[scale=0.8]
 %% Separate lines by 0.8
  \draw[thick,<-] (3.2,0) .. controls (3.2,1.25) and (0,.25) .. (0,2)
     node[pos=0.85, shape=coordinate](X){};
  \draw[thick,<-] (0,0) .. controls (0,1) and (.8,.8) .. (.8,2);
  \draw[thick,<-] (.8,0) .. controls (.8,1) and (1.6,.8) .. (1.6,2);
  \draw[thick,<-] (2.4,0) .. controls (2.4,1) and (3.2,.8) .. (3.2,2);
  \node at (1.6,.35) {$\dots$};
  \node at (2.4,1.65) {$\dots$};
  %\filldraw  (X) circle (2pt);
  \filldraw  (.05,1.6) circle (2pt);
  \filldraw  (.78,1.6) circle (2pt);  % I put the dot .02 to the left to make it look nice
  \filldraw  (1.58,1.6) circle (2pt);
  \filldraw  (3.18,1.6) circle (2pt);
  %% I cant remember a better way of doing this
  \node at (-0.18,1.7) {$\scs j_1$};
  \node at (.55,1.7) {$\scs j_2$};
  \node at (1.35,1.7) {$\scs j_3$};
   \node at (3.45,1.7) {$\scs j_n$};
\end{tikzpicture}}
\; \right]\nonumber
\end{align}

The next lemma allows us to express generators $h_n \otimes (x_1^{j_1} \cdots x_n^{j_n}) $ in terms of generators $ h_r \otimes x_1^k$.
\begin{lemma} \label{dotmovelemma}
For $n \ge 1$ and $1 \le i \le n-1$ we have
\begin{equation*}
h_{\pm n} \otimes x_i=h_{\pm n} \otimes x_{i+1} \pm (h_{\pm i} \otimes 1)(h_{\pm(n-i)}\otimes 1).
\end{equation*}
\end{lemma}
\begin{proof}
This is immediate from \eqref{eq:nil-dot} and properties of the trace.
\end{proof}

%\begin{lemma}
%\label{adh_0x^2}
%\begin{equation*}
%[h_n \otimes x_1^a, c_1]=-2n(h_n \otimes x_1^{a+1}) + \sum_{j=1}^{n-1} 2(n-j) (h_j \otimes x_1^a)(h_{n-j} \otimes 1).
%\end{equation*}
%\end{lemma}

%\begin{lemma}
%There is an equality:
%\label{h-1bh1a}
%\begin{equation*}
%[h_{-1} \otimes x_1^b, h_1 \otimes x_1^a]=\tilde{c}_{a+b}+\sum_{l=0}^{a+b-2} (a+b-1-l) \tilde{c}_l c_{a+b-2-l}.
%\end{equation*}
%\end{lemma}

The next lemma allows us to express generators $ [f_{w;j_1,\ldots,j_n}]$ in terms of the more elementary generators
$ h_r \otimes p $ where $p$ is a polynomial in variables $ x_1, \ldots, x_r$.

\begin{lemma}
\label{getridofwlemma}
Let $ w \in S_n$ and $(n_1, \ldots, n_r) $ a sequence of natural numbers which sum to $n$.  Then
\begin{equation*}
[f_{\pm w;j_1,\ldots,j_n}]=\sum d_{n_1, \ldots, n_r} (h_{n_1} \otimes p_{n_1}) \cdots (h_{n_r} \otimes p_{n_r})
\end{equation*}
for some constants $ d_{n_1, \ldots, n_r} \in \mathbb{C} $ and polynomials $ p_{n_i} $ in $i$ variables.
\end{lemma}

\begin{proof}
We prove this by induction on $j_1 + \cdots + j_n$.
The base case where all of the $j_i$ equal zero is trivial because then $ f_{w;j_1,\ldots,j_n} $ is just an element in the symmetric group and thus its class $ [f_{w;j_1,\ldots,j_n}] $ in the trace is determined by its conjucacy class in $S_n$ which of course could be written as a product of disjoint cycles.

Choose an element $ g \in S_n $ such that
\begin{equation*}
gwg^{-1}=(s_1 \cdots s_{n_1-1}) \cdots (s_{n_1+\cdots n_{r-1}} \cdots s_{n_1+\cdots n_r-1}).
\end{equation*}
Let $p=x_1^{j_1} \cdots x_n^{j_n}$.  Thus $ f_{w;j_1,\ldots,j_n}= pw$.
Now we conjugate this element by $ g $ to get
$ gpwg^{-1} = (g \centerdot p) gwg^{-1} + p_L wg^{-1} $
where $ p_L $ is a polynomial of degree less than $j_1+ \cdots j_n$ and
$ g \centerdot p $ is some other polynomial of degree $ j_1 + \cdots j_n $.
The lemma follows by applying induction to the second term and noting that in the first term $gwg^{-1} $ is a product of cycles.
\end{proof}

\begin{corollary}
Let $ w \in S_n$ and $(n_1, \ldots, n_r) $ a sequence of natural numbers which sum to $n$.  Then
\begin{equation*}
[f_{\pm w;j_1,\ldots,j_n}]=\sum d_{n_1, \ldots, n_r} (h_{\pm n_1} \otimes x_{1}^{l_1}) \cdots (h_{\pm n_r} \otimes x_{1}^{l_r})
\end{equation*}
for some constants $ d_{n_1, \ldots, n_r} \in \mathbb{C} $ and some non-negative integers $l_1, \ldots, l_r$.
\end{corollary}

\begin{proof}
This follows from Lemma ~\ref{dotmovelemma} and Lemma ~\ref{getridofwlemma}.
\end{proof}

\subsection{Elements $p^{(n)}$ and $ q^{(n)}$}
\label{generatorspq}
In this section we will see how the Heisenberg algebra $ \mathfrak{h} $ in the presentation given in terms of $ p^{(n)}, q^{(n)} $ from Section ~\ref{hei} appears in $ \Tr(\H) $.

Define $ p^{(n)} \otimes 1$ and $ q^{(n)} \otimes 1 $ in $ \Tr(\H) $ by
\begin{equation}
\label{pndef}
p^{(n)} \otimes 1 = \frac{1}{n!} \sum_{w \in S_n} [f_{w; 0, \ldots, 0}]
\end{equation}
\begin{equation}
\label{qndef}
q^{(n)} \otimes 1= \frac{1}{n!} \sum_{w \in S_n} [f_{-w; 0, \ldots, 0}]
\end{equation}
We depict the elements $ p^{(n)} \otimes 1$ and $ q^{(n)} \otimes 1$ by
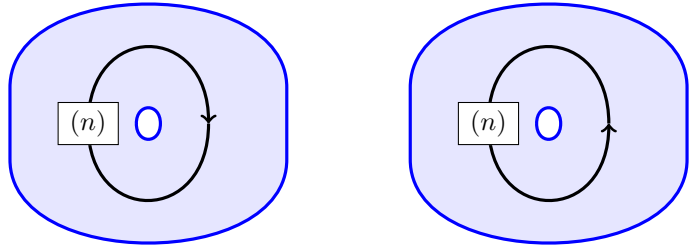
\[
\begin{tikzpicture}[scale=.8]
  %% Annulus
      \path[draw,blue, very thick, fill=blue!10]
        (-2.3,-.6) to (-2.3,.6) .. controls ++(0,1.85) and ++(0,1.85) .. (2.3,.6)
         to (2.3,-.6)  .. controls ++(0,-1.85) and ++(0,-1.85) .. (-2.3,-.6);
     \path[draw, blue, very thick, fill=white]
            (-0.2,0) .. controls ++(0,.35) and ++(0,.35) .. (0.2,0)
            .. controls ++(0,-.35) and ++(0,-.35) .. (-0.2,0);
\draw[very thick,->] (-1,0) .. controls ++(0,1.7) and ++(0,1.7).. (1,0);
\draw[very thick,->] (1,0) .. controls ++(0,-1.7) and ++(0,-1.7).. (-1,0);
\draw[ fill, white] (-1.5,-.35) rectangle (-.5,.35);
\draw  (-1.5,-.35) rectangle (-.5,.35);
\draw (-1,0) node {$(n)$};
\end{tikzpicture}
\;\;  \qquad \qquad
\begin{tikzpicture}[scale=.8]
  %% Annulus
      \path[draw,blue, very thick, fill=blue!10]
        (-2.3,-.6) to (-2.3,.6) .. controls ++(0,1.85) and ++(0,1.85) .. (2.3,.6)
         to (2.3,-.6)  .. controls ++(0,-1.85) and ++(0,-1.85) .. (-2.3,-.6);
     \path[draw, blue, very thick, fill=white]
            (-0.2,0) .. controls ++(0,.35) and ++(0,.35) .. (0.2,0)
            .. controls ++(0,-.35) and ++(0,-.35) .. (-0.2,0);
\draw[very thick,<-] (-1,0) .. controls ++(0,1.7) and ++(0,1.7).. (1,0);
\draw[very thick,<-] (1,0) .. controls ++(0,-1.7) and ++(0,-1.7).. (-1,0);
\draw[ fill, white] (-1.5,-.35) rectangle (-.5,.35);
\draw  (-1.5,-.35) rectangle (-.5,.35);
\draw (-1,0) node {$(n)$};
\end{tikzpicture}
\]

We have for example the relation $(q^{(1)} \otimes 1)(p^{(1)} \otimes 1)= (p^{(1)} \otimes 1)(q^{(1)} \otimes 1)+1$.
The composition $(q^{(1)} \otimes 1)(p^{(1)} \otimes 1)$ is
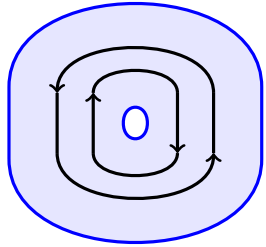
\[
\begin{tikzpicture}[scale=.8]
  %% Annulus
      \path[draw,blue, very thick, fill=blue!10]
        (-2.1,-.6) to (-2.1,.6) .. controls ++(0,1.85) and ++(0,1.85) .. (2.1,.6)
         to (2.1,-.6)  .. controls ++(0,-1.85) and ++(0,-1.85) .. (-2.1,-.6);
     \path[draw, blue, very thick, fill=white]
            (-0.2,0) .. controls ++(0,.35) and ++(0,.35) .. (0.2,0)
            .. controls ++(0,-.35) and ++(0,-.35) .. (-0.2,0);
      %% MIDDLE
      \draw[very thick ] (-1.3,.5) -- (-1.3,-.5);
      \draw[very thick ,<-] (-.7,.5) -- (-.7,-.5);
      \draw[very thick,->] (.7,.5) -- (.7,-.5);
      \draw[very thick] (1.3,.5) -- (1.3,-.5);
     %% TOP
     \draw[very thick] (-.7,.5) .. controls ++(0,.5) and ++(0,.5).. (.7,.5);
     \draw[very thick,->] (1.3,.5) .. controls ++(0,1) and ++(0,1).. (-1.3,.5);
    %% BOTTOM
    \draw[very thick ] (-.7,-.5) .. controls ++(0,-.5) and ++(0,-.5).. (.7,-.5);
    \draw[very thick,<-] (1.3,-.5) .. controls ++(0,-1) and ++(0,-1).. (-1.3,-.5);
\end{tikzpicture}
\]
Now, the graphical relations in $\H$ allow us to pull the inner circle outside as follows
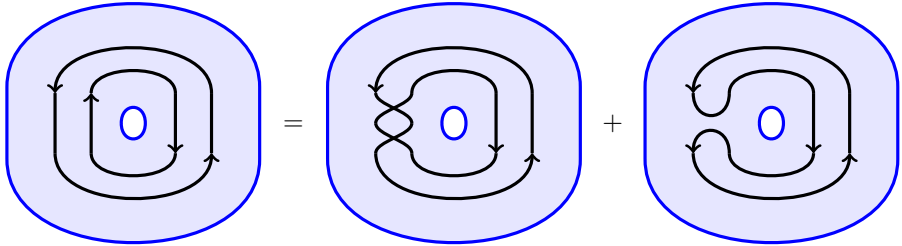
\[
\hackcenter{
\begin{tikzpicture}[scale=.8]
  %% Annulus
      \path[draw,blue, very thick, fill=blue!10]
        (-2.1,-.6) to (-2.1,.6) .. controls ++(0,1.85) and ++(0,1.85) .. (2.1,.6)
         to (2.1,-.6)  .. controls ++(0,-1.85) and ++(0,-1.85) .. (-2.1,-.6);
     \path[draw, blue, very thick, fill=white]
            (-0.2,0) .. controls ++(0,.35) and ++(0,.35) .. (0.2,0)
            .. controls ++(0,-.35) and ++(0,-.35) .. (-0.2,0);
      %% MIDDLE
      \draw[very thick ] (-1.3,.5) -- (-1.3,-.5);
      \draw[very thick ,<-] (-.7,.5) -- (-.7,-.5);
      \draw[very thick,->] (.7,.5) -- (.7,-.5);
      \draw[very thick] (1.3,.5) -- (1.3,-.5);
     %% TOP
     \draw[very thick] (-.7,.5) .. controls ++(0,.5) and ++(0,.5).. (.7,.5);
     \draw[very thick,->] (1.3,.5) .. controls ++(0,1) and ++(0,1).. (-1.3,.5);
    %% BOTTOM
    \draw[very thick ] (-.7,-.5) .. controls ++(0,-.5) and ++(0,-.5).. (.7,-.5);
    \draw[very thick,<-] (1.3,-.5) .. controls ++(0,-1) and ++(0,-1).. (-1.3,-.5);
\end{tikzpicture} }
\;\; = \;\;
\hackcenter{
\begin{tikzpicture}[scale=.8]
  %% Annulus
      \path[draw,blue, very thick, fill=blue!10]
        (-2.1,-.6) to (-2.1,.6) .. controls ++(0,1.85) and ++(0,1.85) .. (2.1,.6)
         to (2.1,-.6)  .. controls ++(0,-1.85) and ++(0,-1.85) .. (-2.1,-.6);
     \path[draw, blue, very thick, fill=white]
            (-0.2,0) .. controls ++(0,.35) and ++(0,.35) .. (0.2,0)
            .. controls ++(0,-.35) and ++(0,-.35) .. (-0.2,0);
      %% MIDDLE
      \draw[very thick ] (-1.3,.5) .. controls ++(0,-.2) and ++(0,+.2) .. (-0.7,0);
      \draw[very thick ] (-0.7,0) .. controls ++(0,-.2) and ++(0,+.2) .. (-1.3,-.5);
      \draw[very thick ] (-0.7,.5) .. controls ++(0,-.2) and ++(0,+.2) .. (-1.3,0);
      \draw[very thick ] (-1.3,0) .. controls ++(0,-.2) and ++(0,+.2) .. (-0.7,-.5);
      \draw[very thick,->] (.7,.5) -- (.7,-.5);
      \draw[very thick] (1.3,.5) -- (1.3,-.5);
     %% TOP
     \draw[very thick] (-.7,.5) .. controls ++(0,.5) and ++(0,.5).. (.7,.5);
     \draw[very thick,->] (1.3,.5) .. controls ++(0,1) and ++(0,1).. (-1.3,.5);
    %% BOTTOM
    \draw[very thick ] (-.7,-.5) .. controls ++(0,-.5) and ++(0,-.5).. (.7,-.5);
    \draw[very thick,<-] (1.3,-.5) .. controls ++(0,-1) and ++(0,-1).. (-1.3,-.5);
\end{tikzpicture} }
\;\; + \;\;
\hackcenter{
\begin{tikzpicture}[scale=.8]
  %% Annulus
      \path[draw,blue, very thick, fill=blue!10]
        (-2.1,-.6) to (-2.1,.6) .. controls ++(0,1.85) and ++(0,1.85) .. (2.1,.6)
         to (2.1,-.6)  .. controls ++(0,-1.85) and ++(0,-1.85) .. (-2.1,-.6);
     \path[draw, blue, very thick, fill=white]
            (-0.2,0) .. controls ++(0,.35) and ++(0,.35) .. (0.2,0)
            .. controls ++(0,-.35) and ++(0,-.35) .. (-0.2,0);
      %% MIDDLE
      \draw[very thick ] (-1.3,.5) .. controls ++(0,-.5) and ++(0,-.5) .. (-0.7,.5);
      \draw[very thick,<-] (-1.3,-.5) .. controls ++(0,.5) and ++(0,+.5) .. (-0.7,-.5);
      \draw[very thick,->] (.7,.5) -- (.7,-.5);
      \draw[very thick] (1.3,.5) -- (1.3,-.5);
     %% TOP
     \draw[very thick] (-.7,.5) .. controls ++(0,.5) and ++(0,.5).. (.7,.5);
     \draw[very thick,->] (1.3,.5) .. controls ++(0,1) and ++(0,1).. (-1.3,.5);
    %% BOTTOM
    \draw[very thick ] (-.7,-.5) .. controls ++(0,-.5) and ++(0,-.5).. (.7,-.5);
    \draw[very thick,<-] (1.3,-.5) .. controls ++(0,-1) and ++(0,-1).. (-1.3,-.5);
\end{tikzpicture} }
\]
Since a counterclockwise circle equals $1$, the circle in the last term above can be erased. On the other hand, the first term on the right can be simplified by sliding in the outside circle to obtain
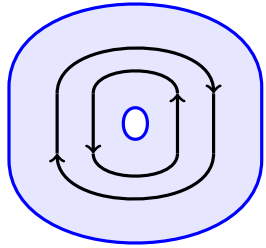
\[
\hackcenter{
\begin{tikzpicture}[scale=.8]
  %% Annulus
      \path[draw,blue, very thick, fill=blue!10]
        (-2.1,-.6) to (-2.1,.6) .. controls ++(0,1.85) and ++(0,1.85) .. (2.1,.6)
         to (2.1,-.6)  .. controls ++(0,-1.85) and ++(0,-1.85) .. (-2.1,-.6);
     \path[draw, blue, very thick, fill=white]
            (-0.2,0) .. controls ++(0,.35) and ++(0,.35) .. (0.2,0)
            .. controls ++(0,-.35) and ++(0,-.35) .. (-0.2,0);
      %% MIDDLE
      \draw[very thick] (-1.3,.5) -- (-1.3,-.5);
      \draw[very thick,->] (-.7,.5) -- (-.7,-.5);
      \draw[very thick,<-] (.7,.5) -- (.7,-.5);
      \draw[very thick] (1.3,.5) -- (1.3,-.5);
     %% TOP
     \draw[very thick,] (-.7,.5) .. controls ++(0,.5) and ++(0,.5).. (.7,.5);
     \draw[very thick,<-] (1.3,.5) .. controls ++(0,1) and ++(0,1).. (-1.3,.5);
    %% BOTTOM
    \draw[very thick] (-.7,-.5) .. controls ++(0,-.5) and ++(0,-.5).. (.7,-.5);
    \draw[very thick,->] (1.3,-.5) .. controls ++(0,-1) and ++(0,-1).. (-1.3,-.5);
\end{tikzpicture} }
\]
Thus we end up with
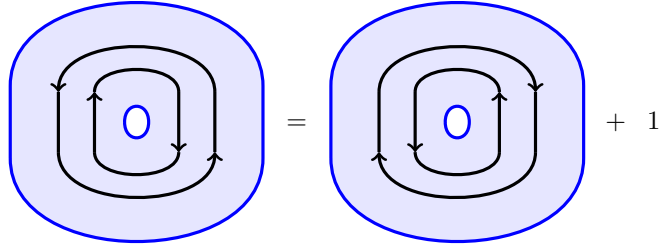
\[
\hackcenter{
\begin{tikzpicture}[scale=.8]
  %% Annulus
      \path[draw,blue, very thick, fill=blue!10]
        (-2.1,-.6) to (-2.1,.6) .. controls ++(0,1.85) and ++(0,1.85) .. (2.1,.6)
         to (2.1,-.6)  .. controls ++(0,-1.85) and ++(0,-1.85) .. (-2.1,-.6);
     \path[draw, blue, very thick, fill=white]
            (-0.2,0) .. controls ++(0,.35) and ++(0,.35) .. (0.2,0)
            .. controls ++(0,-.35) and ++(0,-.35) .. (-0.2,0);
      %% MIDDLE
      \draw[very thick ] (-1.3,.5) -- (-1.3,-.5);
      \draw[very thick ,<-] (-.7,.5) -- (-.7,-.5);
      \draw[very thick,->] (.7,.5) -- (.7,-.5);
      \draw[very thick] (1.3,.5) -- (1.3,-.5);
     %% TOP
     \draw[very thick] (-.7,.5) .. controls ++(0,.5) and ++(0,.5).. (.7,.5);
     \draw[very thick,->] (1.3,.5) .. controls ++(0,1) and ++(0,1).. (-1.3,.5);
    %% BOTTOM
    \draw[very thick ] (-.7,-.5) .. controls ++(0,-.5) and ++(0,-.5).. (.7,-.5);
    \draw[very thick,<-] (1.3,-.5) .. controls ++(0,-1) and ++(0,-1).. (-1.3,-.5);
\end{tikzpicture}}
\;\; = \;\;
\hackcenter{
\begin{tikzpicture}[scale=.8]
  %% Annulus
      \path[draw,blue, very thick, fill=blue!10]
        (-2.1,-.6) to (-2.1,.6) .. controls ++(0,1.85) and ++(0,1.85) .. (2.1,.6)
         to (2.1,-.6)  .. controls ++(0,-1.85) and ++(0,-1.85) .. (-2.1,-.6);
     \path[draw, blue, very thick, fill=white]
            (-0.2,0) .. controls ++(0,.35) and ++(0,.35) .. (0.2,0)
            .. controls ++(0,-.35) and ++(0,-.35) .. (-0.2,0);
      %% MIDDLE
      \draw[very thick] (-1.3,.5) -- (-1.3,-.5);
      \draw[very thick,->] (-.7,.5) -- (-.7,-.5);
      \draw[very thick,<-] (.7,.5) -- (.7,-.5);
      \draw[very thick] (1.3,.5) -- (1.3,-.5);
     %% TOP
     \draw[very thick,] (-.7,.5) .. controls ++(0,.5) and ++(0,.5).. (.7,.5);
     \draw[very thick,<-] (1.3,.5) .. controls ++(0,1) and ++(0,1).. (-1.3,.5);
    %% BOTTOM
    \draw[very thick] (-.7,-.5) .. controls ++(0,-.5) and ++(0,-.5).. (.7,-.5);
    \draw[very thick,->] (1.3,-.5) .. controls ++(0,-1) and ++(0,-1).. (-1.3,-.5);
\end{tikzpicture} }
\;\; + \;\; 1
\]

\begin{theorem}
\label{pqrelations}
The elements $ p^{(n)} \otimes 1, q^{(n)} \otimes 1 $ for $ n \in \mathbb{Z}_{\geq 0} $ satisfy the relations:
\begin{eqnarray*}
(p^{(n)} \otimes 1)(p^{(m)} \otimes 1) &=& (p^{(m)} \otimes 1)(p^{(n)} \otimes 1)
\text{ and }
(q^{(n)} \otimes 1)(q^{(m)} \otimes 1) = (q^{(m)} \otimes 1)(q^{(n)} \otimes 1) \\
(q^{(n)} \otimes 1)(p^{(m)} \otimes 1) &=&
\sum_{k \ge 0} (p^{(m-k)} \otimes 1)(q^{(n-k)} \otimes 1).
\end{eqnarray*}
\end{theorem}

\begin{proof}
This follows fairly directly from the proof of the categorified statement in \cite{K} and its analagous statement in \cite{CL1}.  We sketch the details for the computation of $ (q^{(n)} \otimes 1)(p^{(m)} \otimes 1) $.
By definition $ (q^{(n)} \otimes 1)(p^{(m)} \otimes 1) $ is
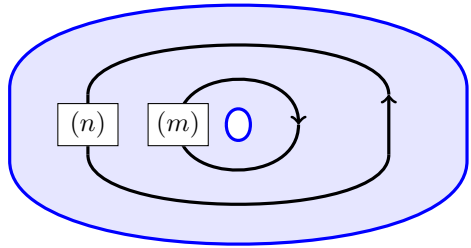
\begin{equation}
\label{qpdiag}
\begin{tikzpicture}[scale=.8]
  %% Annulus
      \path[draw,blue, very thick, fill=blue!10]
        (-3.8,-.6) to (-3.8,.6) .. controls ++(0,1.85) and ++(0,1.85) .. (3.8,.6)
         to (3.8,-.6)  .. controls ++(0,-1.85) and ++(0,-1.85) .. (-3.8,-.6);
     \path[draw, blue, very thick, fill=white]
            (-0.2,0) .. controls ++(0,.35) and ++(0,.35) .. (0.2,0)
            .. controls ++(0,-.35) and ++(0,-.35) .. (-0.2,0);
\draw[very thick,->] (-1,0) .. controls ++(0,1 ) and ++(0,1 ).. (1,0);
\draw[very thick,->] (1,0) .. controls ++(0,-1 ) and ++(0,-1 ).. (-1,0);
\draw[ fill, white] (-1.5,-.35) rectangle (-.5,.35);
\draw  (-1.5,-.35) rectangle (-.5,.35);
\draw (-1,0) node {$(m)$};
\draw[very thick ] (-2.5,.5) .. controls ++(0,1.1) and ++(0,1.1).. (2.5,.5);
\draw[very thick] (2.5,-.5) .. controls ++(0,-1.1) and ++(0,-1.1).. (-2.5,-.5);
\draw[very thick,->] (2.5,-.5) -- (2.5,.5);
\draw[very thick] (-2.5,-.5) -- (-2.5,.5);
\draw[ fill, white] (-3,-.35) rectangle (-2,.35);
\draw  (-3,-.35) rectangle (-2,.35);
\draw (-2.5,0) node {$(n)$};
\end{tikzpicture}
\end{equation}
Letting $ c_{m,n}^{k} = k! \binom{m}{k} \binom{n}{k}$
we get using a straightforward modification of \cite[Equation 31]{CL1} that ~\eqref{qpdiag} is equal to:
\begin{equation}
\label{qpdiag2}
\sum_{k} c_{m,n}^{k}
\hackcenter{
\begin{tikzpicture} [scale=.8]
  %% Annulus
      \path[shift={+(-1,2)}][draw,blue, very thick, fill=blue!10]
        (-3.8,-2.6) to (-3.8,2.9) .. controls ++(0,2) and ++(0,2) .. (5.5,2.9)
         to (5.5,-2.6)  .. controls ++(0,-2) and ++(0,-2) .. (-3.8,-2.6);
     \path[draw, blue, very thick, fill=white]
            (-1,2) .. controls ++(0,.35) and ++(0,.35) .. (-1.4,2)
            .. controls ++(0,-.35) and ++(0,-.35) .. (-1,2);
\draw  (-5,2.25) node {$ $};
\draw[ fill, white] (-.75,0) rectangle (1.25,.5);
\draw (-.75,0) rectangle (1.25,.5);
\draw (.25,.25) node {$(m)$};
\draw[ fill, white] (-.75,2) rectangle (1.25,2.5);
\draw (-.75,2) rectangle (1.25,2.5);
\draw (.25,2.25) node {$(n-k)$};
\draw [very thick,<-] (.25,.5) .. controls (.25,1) and (2.25,1.5) .. (2.25,2);
\draw (2.25,.5) arc (0:180:.75cm and .5cm)[->][very thick];
\draw (1.5,1) node [anchor=north] {$k$};
\draw[ fill, white] (1.5,0) rectangle (3.5,.5);
\draw (1.5,0) rectangle (3.5,.5);
\draw (2.5,.25) node {$(n)$};
\draw[ fill, white]  (1.5,2) rectangle (3.5,2.5);
\draw (1.5,2) rectangle (3.5,2.5);
\draw (2.5,2.25) node {$(m-k)$};
\draw[very thick,->] (2.75,.5) .. controls (2.75,1) and (.75,1.5) ..(.75,2);
\draw[shift={+(0,2)}][ fill, white] (-.75,2) rectangle (1.25,2.5);
\draw [shift={+(0,2)}](-.75,2) rectangle (1.25,2.5);
\draw [shift={+(0,2)}](.25,2.25) node {$(m)$};
\draw [shift={+(0,2)}][very thick,->] (.75,.5) .. controls (.75,1) and (2.75,1.5) .. (2.75,2);
\draw  [shift={+(0,2)}](.75,2) arc (180:360:.75cm and .5cm)[->][very thick];
\draw  [shift={+(0,2)}](1.5,1.5) node [anchor=south] {$k$};
\draw[shift={+(0,2)}][ fill, white] (1.5,2) rectangle (3.5,2.5);
\draw [shift={+(0,2)}] (1.5,2) rectangle (3.5,2.5);
\draw  [shift={+(0,2)}](2.5,2.25) node {$(n)$};
\draw  [shift={+(0,2)}] [very thick,<-](2.25,.5) .. controls (2.25,1) and (.25,1.5) .. (.25,2);
%\draw (.25,0) arc (-180:-360:-1cm) [very thick];
\draw (.25,0) arc (0:-180:1cm and .5cm)[][very thick];
%\draw (2.5,0) arc (-180:-360:-3cm) [very thick];
\draw (2.5,0) arc (0:-180:3cm and 1.5cm)[][very thick];
%\draw (.25,4.5) arc (180:360:-1cm) [very thick];
\draw (.25,4.5) arc (0:180:1cm and .5cm)[][very thick];
%\draw (2.5,4.5) arc (180:360:-3cm) [very thick];
\draw (2.5,4.5) arc (0:180:3cm and 1.5cm)[][very thick];
\draw (-1.75,0) -- (-1.75,4.5) [->][very thick];
\draw (-3.5,0) -- (-3.5,4.5) [<-][very thick];
\end{tikzpicture}}
\end{equation}
Slide the top two rectangles and the $k$ cups connecting them counterclockwise to the bottom of the diagram and apply a modification of \cite[Equation 32]{CL1}, to get that ~\eqref{qpdiag2} is equal to:
\[
\sum_{0 \leq l \leq k}\;\;
\hackcenter{
\begin{tikzpicture}[scale=.8]
  %% Annulus
      \path[draw,blue, very thick, fill=blue!10]
        (-5.6,-.8) to (-5.6,.8) .. controls ++(0,1.85) and ++(0,1.85) .. (3.5,.8)
         to (3.5,-.8)  .. controls ++(0,-1.85) and ++(0,-1.85) .. (-5.6,-.8);
     \path[draw, blue, very thick, fill=white]
            (-0.2,0) .. controls ++(0,.35) and ++(0,.35) .. (0.2,0)
            .. controls ++(0,-.35) and ++(0,-.35) .. (-0.2,0);
\draw[very thick] (-1.5,0) .. controls ++(0,1 ) and ++(0,1 ).. (1.5,0);
\draw[very thick,<-] (1.5,0) .. controls ++(0,-1 ) and ++(0,-1 ).. (-1.5,0);
%\draw[ fill, white] (-1.5,-.35) rectangle (-.5,.35);
\draw[ fill, white]  (-2.49,-.25) rectangle (-.51,.25);
\draw  (-2.5,-.3) rectangle (-.5,.3);
\draw (-1.5,0) node {$ (n-k)$};
\draw[very thick ] (-4,.5) .. controls ++(0,1.3) and ++(0,1.3).. (2.5,.5);
\draw[very thick] (2.5,-.5) .. controls ++(0,-1.3) and ++(0,-1.3).. (-4,-.5);
\draw[very thick,<- ] (2.5,-.5) -- (2.5,.5);
\draw[very thick] (-4,-.5) -- (-4,.5);
\draw[ fill, white] (-5,-.3) rectangle (-3,.3);
\draw  (-5,-.35) rectangle (-3,.35);
\draw (-4,0) node {$ (m-k)$};
\end{tikzpicture}}~.
\]
By definition this is equal to
$ \sum_k (p^{(m-k)} \otimes 1)(q^{(n-k)} \otimes 1)$.

\end{proof}

\subsection{Some diagrammatic lemmas}
For convenience, we work with classes $[f]$ of endomorphisms in $\Tr(\H)$ and omit the annuli from our diagrammatic computations.  The reader should keep in mind that all diagrams inside of closed brackets are interpreted on the annulus.

\begin{lemma}
For any $a>1$ we have the following identities
\begin{align}
\hackcenter{\begin{tikzpicture}[scale=0.8]
    \draw[thick, ->] (0,0) .. controls (0,.75) and (.75,.75) .. (.75,1.5)
        node[pos=.25, shape=coordinate](DOT){};
    \draw[thick, ->] (.75,0) .. controls (.75,.75) and (0,.75) .. (0,1.5);
    \filldraw  (DOT) circle (2pt);
    \node at (-.2,.3) {$\scs a$};
\end{tikzpicture}}
\quad-\quad
\hackcenter{\begin{tikzpicture}[scale=0.8]
    \draw[thick, ->] (0,0) .. controls (0,.75) and (.75,.75) .. (.75,1.5)
        node[pos=.75, shape=coordinate](DOT){};
    \draw[thick, ->] (.75,0) .. controls (.75,.75) and (0,.75) .. (0,1.5);
    \filldraw  (DOT) circle (2pt);
    \node at (.95,1.1) {$\scs a$};
\end{tikzpicture}}
&\quad=\quad
\hackcenter{\begin{tikzpicture}[scale=0.8]
    \draw[thick, ->] (0,0) .. controls (0,.75) and (.75,.75) .. (.75,1.5);
    \draw[thick, ->] (.75,0) .. controls (.75,.75) and (0,.75) .. (0,1.5)
        node[pos=.75, shape=coordinate](DOT){};
    \filldraw  (DOT) circle (2pt);
    \node at (-.2,1.1) {$\scs a$};
\end{tikzpicture}}
\quad-\quad
\hackcenter{\begin{tikzpicture}[scale=0.8]
    \draw[thick, ->] (0,0) .. controls (0,.75) and (.75,.75) .. (.75,1.5);
    \draw[thick, ->] (.75,0) .. controls (.75,.75) and (0,.75) .. (0,1.5)
        node[pos=.25, shape=coordinate](DOT){};
      \filldraw  (DOT) circle (2pt);
      \node at (.95,.3) {$\scs a$};
\end{tikzpicture}}
\quad=\quad
\sum_{f+g = a-1}
\hackcenter{\begin{tikzpicture}[scale=0.8]
    \draw[thick, ->] (0,0) -- (0,1.5);
    \draw[thick, ->] (.75,0) -- (.75,1.5);
    \node at (-.25,.85) {$\scs f$};
    \node at (1,.85) {$\scs g$};
    \filldraw  (0,.75) circle (2pt);
    \filldraw  (.75,.75) circle (2pt);
\end{tikzpicture}} \label{eq:inddot}
\\
\label{eq:dotted-curl}
\hackcenter{
\begin{tikzpicture}[scale=0.8]
    \draw  [thick](0,0) .. controls (0,.5) and (.7,.5) .. (.9,0);
    \draw  [thick](0,0) .. controls (0,-.5) and (.7,-.5) .. (.9,0);
    \draw  [thick](1,-1) .. controls (1,-.5) .. (.9,0);
    \draw  [thick,->](.9,0) .. controls (1,.5) .. (1,1) ;
     \filldraw  (.02,0) circle (2pt);
     \node at (-.2,.1) {$\scs a$};
\end{tikzpicture}}
&\;\; = \;\;
\hackcenter{
\begin{tikzpicture}[scale=0.8]
    \draw  [thick,->](0,-1)-- (0,1);
     \filldraw  (0,0) circle (2pt);
     \node at (-.45,.1) {$\scs a-1$};
\end{tikzpicture}}
\;\; + \;\;
\sum_{f+g=a-3}\; \tilde{c}_{f+2}\;
\hackcenter{
\begin{tikzpicture}[scale=0.8]
    \draw  [thick,->](0,-1)-- (0,1);
     \filldraw  (0,0) circle (2pt);
     \node at (-.25,.1) {$\scs g$};
\end{tikzpicture}}
\end{align}
\end{lemma}
\begin{proof}
The first claim follows inductively from \eqref{eq:nil-dot}.  The second claim follows from the first after observing that a left twist curl is zero and the $\tilde{c}_1=0$.
\end{proof}

\begin{lemma} \label{lem:mn-onetangle}
For any $a \geq 0$ we have
\[
\left[ \;
\hackcenter{
\begin{tikzpicture}[scale=0.7]
 %% Separate lines by 0.6
 %% UPWARD ORIENTED LINES
  \draw[thick,->] (6,2) .. controls ++(0,1.25) and ++(0,-1.75) .. (3,3.5);
  \draw[thick,->] (3,2) .. controls ++(0,1) and ++(0,-.7) .. (3.6,3.5);
  \draw[thick,->] (3.6,2) .. controls ++(0,1) and ++(0,-.7) .. (4.2,3.5);
  \draw[thick,->] (4.8,2) .. controls ++(0,1) and ++(0,-.7) .. (5.4,3.5);
  \draw[thick,->] (5.4,2) .. controls ++(0,1) and ++(0,-.7) .. (6,3.5);
  \node at (4.2,2.35) {$\dots$};
  \node at (4.8,3.15) {$\dots$};
  %%
  %% DOWNWARD ORIENTED LINES
  \draw[thick,<-] (2.4,2) .. controls ++(0,1.25) and ++(0,-1.75) .. (0,3.5) ;
  \draw[thick,<-] (0,2) .. controls ++(0,1) and ++(0,-.7) .. (.6,3.5);
  \draw[thick,<-] (.6,2) .. controls ++(0,1) and ++(0,-.7) .. (1.2,3.5);
  \draw[thick,<-] (1.8,2) .. controls ++(0,1) and ++(0,-.7) .. (2.4,3.5);
  \node at (1.2, 2.35) {$\dots$};
  \node at (1.8,3.15) {$\dots$};
%  %%
  %% BLUE LINES
    \draw[blue, dotted] (-0.4,2) -- (6.4,2);
  %% bottom
   \draw[thick,<-] (6.0,1) .. controls ++(0,1.05) and ++(0,-.75) .. (2.4,2);
   \draw[thick,->] (6.0,1) .. controls ++(0,-1.05) and ++(0,.75) .. (2.4,0);
   \draw[thick,->] (3,0) .. controls ++(0,.6) and ++(0,-.5) .. (2.4,1);
   \draw[thick ]   (2.4,1) .. controls ++(0,.6) and ++(0,-.5) .. (3,2);
   \draw[thick,->] (3.6,0) .. controls ++(0,.6) and ++(0,-.5) .. (3,1);
   \draw[thick ]   (3,1) .. controls ++(0,.6) and ++(0,-.5) .. (3.6,2);
   \draw[thick,->] (4.8,0) .. controls ++(0,.6) and ++(0,-.5) .. (4.2,1);
   \draw[thick ]   (4.2,1) .. controls ++(0,.6) and ++(0,-.5) .. (4.8,2);
    \draw[thick,->] (5.4,0) .. controls ++(0,.6) and ++(0,-.5) .. (4.8,1);
   \draw[thick ]   (4.8,1) .. controls ++(0,.6) and ++(0,-.5) .. (5.4,2);
    \draw[thick,->] (6,0) .. controls ++(0,.6) and ++(0,-.5) .. (5.4,1);
   \draw[thick ]   (5.4,1) .. controls ++(0,.6) and ++(0,-.5) .. (6,2);
   \draw[thick] (1.8,0) -- (1.8,2);
   \draw[thick] (0.6,0) -- (0.6,2);
   \draw[thick] (0.0,0) -- (0,2);
   \node at (1.2,.15) {$\dots$};
   \node at (4.2,.15) {$\dots$};
     \filldraw  (.05,3.05) circle (2pt);
     \node at (-0.2,3.15) {$\scs a$};
\end{tikzpicture}} \right]
  =
\left\{
  \begin{array}{ll}
    (h_{-m}\otimes x_1^a)(h_n\otimes 1), & \hbox{if $m > n \geq 1$;} \\
    (h_{-m}\otimes x_1^a)(h_n\otimes 1)-n\tilde{c}_{a}, & \hbox{if $m=n$;} \\
    (h_{-m}\otimes x_1^a)(h_n\otimes 1)
     - n \xy
  (0,.4)*{\sum};
  (0,-5.6)*{\scs a-1};
  (0,-2.9)*{\scs f+g=};
  \endxy\tilde{c}_{f} (h_{n-m}\otimes x_1^g)
     , & \hbox{if $n>m$,}
  \end{array}
\right.
\]
where the left side involves $m$ downward strands and $n$ upward strands.
\end{lemma}

\begin{proof} The claim follows by simplifying the diagram using the first equation in \eqref{eq:rel2}
\begin{equation}
\;\; = \;\;
\left[ \;
\hackcenter{
\begin{tikzpicture}[scale=0.8]
 %% Separate lines by 0.6
 %% UPWARD ORIENTED LINES
  \draw[thick,->] (6,2) .. controls ++(0,1.25) and ++(0,-1.75) .. (3,3.5);
  \draw[thick,->] (3,2) .. controls ++(0,1) and ++(0,-.7) .. (3.6,3.5);
  \draw[thick,->] (3.6,2) .. controls ++(0,1) and ++(0,-.7) .. (4.2,3.5);
  \draw[thick,->] (4.8,2) .. controls ++(0,1) and ++(0,-.7) .. (5.4,3.5);
  \draw[thick,->] (5.4,2) .. controls ++(0,1) and ++(0,-.7) .. (6,3.5);
  \node at (4.2,2.35) {$\dots$};
  \node at (4.8,3.15) {$\dots$};
  %%
  %% DOWNWARD ORIENTED LINES
  \draw[thick,<-] (2.4,2) .. controls ++(0,1.25) and ++(0,-1.75) .. (0,3.5) ;
  \draw[thick,<-] (0,2) .. controls ++(0,1) and ++(0,-.7) .. (.6,3.5);
  \draw[thick,<-] (.6,2) .. controls ++(0,1) and ++(0,-.7) .. (1.2,3.5);
  \draw[thick,<-] (1.8,2) .. controls ++(0,1) and ++(0,-.7) .. (2.4,3.5);
  \node at (1.2, 2.35) {$\dots$};
  \node at (1.8,3.15) {$\dots$};
%  %%
  %% BLUE LINES
    \draw[blue, dotted] (-0.4,2) -- (6.4,2);
  %% bottom
   \draw[thick,<-] (5.4,1) .. controls ++(0,1.05) and ++(0,-.75) .. (2.4,2);
   \draw[thick,->] (5.4,1) .. controls ++(0,-1.05) and ++(0,.75) .. (2.4,0);
   \draw[thick,->] (3,0) .. controls ++(0,.6) and ++(0,-.5) .. (2.4,1);
   \draw[thick ]   (2.4,1) .. controls ++(0,.6) and ++(0,-.5) .. (3,2);
   \draw[thick,->] (3.6,0) .. controls ++(0,.6) and ++(0,-.5) .. (3,1);
   \draw[thick ]   (3,1) .. controls ++(0,.6) and ++(0,-.5) .. (3.6,2);
   \draw[thick,->] (4.8,0) .. controls ++(0,.6) and ++(0,-.5) .. (4.2,1);
   \draw[thick ]   (4.2,1) .. controls ++(0,.6) and ++(0,-.5) .. (4.8,2);
    \draw[thick,->] (5.4,0) .. controls ++(0,.6) and ++(0,-.5) .. (4.8,1);
   \draw[thick ]   (4.8,1) .. controls ++(0,.6) and ++(0,-.5) .. (5.4,2);
    \draw[thick,->] (6,0) .. controls ++(0,.6) and ++(0,-.5) .. (6,2);
   \draw[thick] (1.8,0) -- (1.8,2);
   \draw[thick] (0.6,0) -- (0.6,2);
   \draw[thick] (0.0,0) -- (0,2);
   \node at (1.2,.15) {$\dots$};
   \node at (4.2,.15) {$\dots$};
     \filldraw  (.05,3.05) circle (2pt);
     \node at (-0.2,3.15) {$\scs a$};
\end{tikzpicture}} \right]
\;\; -\;\;
\left[ \;
\hackcenter{
\begin{tikzpicture}[scale=0.8]
 %% Separate lines by 0.6
 %% UPWARD ORIENTED LINES
  \draw[thick,->] (6,2) .. controls ++(0,1.25) and ++(0,-1.75) .. (3,3.5);
  \draw[thick,->] (3,2) .. controls ++(0,1) and ++(0,-.7) .. (3.6,3.5);
  \draw[thick,->] (3.6,2) .. controls ++(0,1) and ++(0,-.7) .. (4.2,3.5);
  \draw[thick,->] (4.8,2) .. controls ++(0,1) and ++(0,-.7) .. (5.4,3.5);
  \draw[thick,->] (5.4,2) .. controls ++(0,1) and ++(0,-.7) .. (6,3.5);
  \node at (4.2,2.35) {$\dots$};
  \node at (4.8,3.15) {$\dots$};
  %%
  %% DOWNWARD ORIENTED LINES
  \draw[thick,<-] (2.4,2) .. controls ++(0,1.25) and ++(0,-1.75) .. (0,3.5) ;
  \draw[thick,<-] (0,2) .. controls ++(0,1) and ++(0,-.7) .. (.6,3.5);
  \draw[thick,<-] (.6,2) .. controls ++(0,1) and ++(0,-.7) .. (1.2,3.5);
  \draw[thick,<-] (1.8,2) .. controls ++(0,1) and ++(0,-.7) .. (2.4,3.5);
  \node at (1.2, 2.35) {$\dots$};
  \node at (1.8,3.15) {$\dots$};
%  %%
  %% BLUE LINES
    \draw[blue, dotted] (-0.4,2) -- (6.4,2);
  %% bottom
   \draw[thick,<-] (6.0,2) .. controls ++(0,-.75) and ++(0,-.75) .. (2.4,2);
   \draw[thick,->] (6.0,0) .. controls ++(0,.75) and ++(0,1) .. (2.4,0);
   \draw[thick,->] (3,0) .. controls ++(0,.6) and ++(0,-.5) .. (2.4,1);
   \draw[thick ]   (2.4,1) .. controls ++(0,.6) and ++(0,-.5) .. (3,2);
   \draw[thick,->] (3.6,0) .. controls ++(0,.6) and ++(0,-.5) .. (3,1);
   \draw[thick ]   (3,1) .. controls ++(0,.6) and ++(0,-.5) .. (3.6,2);
   \draw[thick,->] (4.8,0) .. controls ++(0,.6) and ++(0,-.5) .. (4.2,1);
   \draw[thick ]   (4.2,1) .. controls ++(0,.6) and ++(0,-.5) .. (4.8,2);
    \draw[thick,->] (5.4,0) .. controls ++(0,.6) and ++(0,-.5) .. (4.8,1);
   \draw[thick ]   (4.8,1) .. controls ++(0,.6) and ++(0,-.5) .. (5.4,2);
   \draw[thick] (1.8,0) -- (1.8,2);
   \draw[thick] (0.6,0) -- (0.6,2);
   \draw[thick] (0.0,0) -- (0,2);
   \node at (1.2,.15) {$\dots$};
   \node at (4.2,.15) {$\dots$};
     \filldraw  (.05,3.05) circle (2pt);
     \node at (-0.2,3.15) {$\scs a$};
\end{tikzpicture}} \right]~. \label{eq:aaa}
\end{equation}
The second term unwinds substantially and simplifies to a diagram containing a left twist curl if $m>n$.  If $m=n$ this term reduces to a counter-clockwise bubble with $a$ dots.  Finally, if $n>m$ then the diagram reduces to $-h_{n-m} \otimes 1$ together with a left twist curl containing $a$ dots inside the curl on the first strand.  Reducing this dotted curl using \eqref{eq:dotted-curl} shows that in this case the diagram reduces to $-\sum_{f+g=a-1}\tilde{c}_f(h_{n-m}\otimes x_1^a)$.

Each time the downward oriented strand is pulled through an upward oriented strand in the bottom half of the first diagram in \eqref{eq:aaa} the resolution term will simplify to a diagram containing a left twist curl if $m>n$, to $\tilde{c}_a$ if $m=n$, and to the sum $-\sum_{f+g=a-1}\tilde{c}_f(h_{n-m}\otimes x_1^a)$ when $n>m$.  Since there are $n$ such upward oriented strands the result follows.
\end{proof}

\begin{lemma} \label{lem:alem2}
For any $a \geq 0$ we have
\[
\left[ \;
\hackcenter{
\begin{tikzpicture}[scale=0.7]
 %% Separate lines by 0.6
 %% UPWARD ORIENTED LINES
  \draw[thick] (3,0) .. controls ++(0,1.25) and ++(0,-1.75) .. (0,1.5);
  \draw[thick] (0,0) .. controls ++(0,1) and ++(0,-.7) .. (.6,1.5);
  \draw[thick] (.6,0) .. controls ++(0,1) and ++(0,-.7) .. (1.2,1.5);
  \draw[thick] (1.8,0) .. controls ++(0,1) and ++(0,-.7) .. (2.4,1.5);
  \draw[thick] (2.4,0) .. controls ++(0,1) and ++(0,-.7) .. (3,1.5);
  \node at (1.2,.35) {$\dots$};
  \node at (1.8,1.15) {$\dots$};
  %%
  %% DOWNWARD ORIENTED LINES
  \draw[thick] (6,0) .. controls ++(0,1.25) and ++(0,-1.75) .. (3.6,1.5) ;
  \draw[thick] (3.6,0) .. controls ++(0,1) and ++(0,-.7) .. (4.2,1.5);
  \draw[thick] (4.2,0) .. controls ++(0,1) and ++(0,-.7) .. (4.8,1.5);
  \draw[thick] (5.4,0) .. controls ++(0,1) and ++(0,-.7) .. (6,1.5);
  \node at (4.8,-.65) {$\dots$};
  \node at (5.4,1.85) {$\dots$};
  %%
  %% BLUE LINES
    \draw[blue, dotted] (-0.4,0) -- (6.4,0);
    \draw[blue, dotted] (-0.4,1.5) -- (6.4,1.5);
  %% TOP LEVEL
   \draw[thick,->]  (3,1.5) .. controls ++(0,0.5) and ++(0,-.5) .. (3.6,2.5);
   \draw[thick] (3.6,1.5) .. controls ++(0,1) and ++(0,-1) .. (0,2.5);
   \draw[thick,->]  (2.4,1.5) .. controls ++(0,0.5) and ++(0,-.5) .. (3,2.5);
   \draw[thick,->]  (1.2,1.5) .. controls ++(0,0.5) and ++(0,-.5) .. (1.8,2.5);
   \draw[thick,->]  (.6,1.5) .. controls ++(0,0.5) and ++(0,-.5) .. (1.2,2.5);
   \draw[thick,->]  (0,1.5) .. controls ++(0,0.5) and ++(0,-.5) .. (.6,2.5);
   \draw[thick] (4.2,1.5) -- (4.2,2.5);
   \draw[thick] (4.8,1.5) -- (4.8,2.5);
   \draw[thick] (6,1.5) -- (6,2.5);
     \filldraw  (.15,2.2) circle (2pt);
     \node at (-0.15,2.3) {$\scs a$};
  %% BOTTOM LEVEL
   \draw[thick]  (2.4,0) .. controls ++(0,-.5) and ++(0,.5) .. (3,-1);
   \draw[thick]  (1.8,0) .. controls ++(0,-.5) and ++(0,.5) .. (2.4,-1);
   \draw[thick]  (.6,0) .. controls ++(0,-.5) and ++(0,.5) .. (1.2,-1);
   \draw[thick]  (0,0) .. controls ++(0,-.5) and ++(0,.5) .. (.6,-1);
   \draw[thick]  (3,0) .. controls ++(0,-.5) and ++(0,.5) .. (3.6,-1);
  \draw[thick ,->] (3.6,0) .. controls ++(0,-1) and ++(0,+1) .. (0,-1);
   \draw[thick,->] (4.2,0) -- (4.2,-1);
   \draw[thick,->] (5.4,0) -- (5.4,-1);
   \draw[thick,->] (6,0) -- (6,-1);
\end{tikzpicture}}
\right]
 =
\left\{
  \begin{array}{ll}
    (h_{-m}\otimes x_1^a)(h_n\otimes 1), & \hbox{if $m > n \geq 1$;} \\
    (h_{-m}\otimes x_1^a)(h_n\otimes 1)-n\tilde{c}_{a}, & \hbox{if $m=n$;} \\
    (h_{-m}\otimes x_1^a)(h_n\otimes 1)
     - n   \xy
  (0,.4)*{\sum};
  (0,-5.6)*{\scs a-1};
  (0,-2.9)*{\scs f+g=};
  \endxy
   \tilde{c}_{f} (h_{n-m}\otimes x_1^g)
     , & \hbox{if $n>m$,}
  \end{array}
\right.
\]
where the left side involves $m$ downward strands and $n$ upward strands.
\end{lemma}

\begin{proof}
By repeatedly applying the second equation in \eqref{eq:rel2} together with the trace relation it follows that
\[
\left[ \;
\hackcenter{
\begin{tikzpicture}[scale=0.8]
 %% Separate lines by 0.6
 %% UPWARD ORIENTED LINES
  \draw[thick] (3,0) .. controls ++(0,1.25) and ++(0,-1.75) .. (0,1.5);
  \draw[thick] (0,0) .. controls ++(0,1) and ++(0,-.7) .. (.6,1.5);
  \draw[thick] (.6,0) .. controls ++(0,1) and ++(0,-.7) .. (1.2,1.5);
  \draw[thick] (1.8,0) .. controls ++(0,1) and ++(0,-.7) .. (2.4,1.5);
  \draw[thick] (2.4,0) .. controls ++(0,1) and ++(0,-.7) .. (3,1.5);
  \node at (1.2,.35) {$\dots$};
  \node at (1.8,1.15) {$\dots$};
  %%
  %%DOWNWARD ORIENTED LINES
  \draw[thick] (6,0) .. controls ++(0,1.25) and ++(0,-1.75) .. (3.6,1.5) ;
  \draw[thick] (3.6,0) .. controls ++(0,1) and ++(0,-.7) .. (4.2,1.5);
  \draw[thick] (4.2,0) .. controls ++(0,1) and ++(0,-.7) .. (4.8,1.5);
  \draw[thick] (5.4,0) .. controls ++(0,1) and ++(0,-.7) .. (6,1.5);
  \node at (4.8,-.65) {$\dots$};
  \node at (5.4,1.85) {$\dots$};
  %%
  %% BLUE LINES
    \draw[blue, dotted] (-0.4,0) -- (6.4,0);
    \draw[blue, dotted] (-0.4,1.5) -- (6.4,1.5);
  %% TOP LEVEL
   \draw[thick,->]  (3,1.5) .. controls ++(0,0.5) and ++(0,-.5) .. (3.6,2.5);
   \draw[thick] (3.6,1.5) .. controls ++(0,1) and ++(0,-1) .. (0,2.5);
   \draw[thick,->]  (2.4,1.5) .. controls ++(0,0.5) and ++(0,-.5) .. (3,2.5);
   \draw[thick,->]  (1.2,1.5) .. controls ++(0,0.5) and ++(0,-.5) .. (1.8,2.5);
   \draw[thick,->]  (.6,1.5) .. controls ++(0,0.5) and ++(0,-.5) .. (1.2,2.5);
   \draw[thick,->]  (0,1.5) .. controls ++(0,0.5) and ++(0,-.5) .. (.6,2.5);
   \draw[thick] (4.2,1.5) -- (4.2,2.5);
   \draw[thick] (4.8,1.5) -- (4.8,2.5);
   \draw[thick] (6,1.5) -- (6,2.5);
     \filldraw  (.15,2.2) circle (2pt);
     \node at (-0.15,2.3) {$\scs a$};
  %% BOTTOM LEVEL
   \draw[thick]  (2.4,0) .. controls ++(0,-.5) and ++(0,.5) .. (3,-1);
   \draw[thick]  (1.8,0) .. controls ++(0,-.5) and ++(0,.5) .. (2.4,-1);
   \draw[thick]  (.6,0) .. controls ++(0,-.5) and ++(0,.5) .. (1.2,-1);
   \draw[thick]  (0,0) .. controls ++(0,-.5) and ++(0,.5) .. (.6,-1);
   \draw[thick]  (3,0) .. controls ++(0,-.5) and ++(0,.5) .. (3.6,-1);
  \draw[thick ,->] (3.6,0) .. controls ++(0,-1) and ++(0,+1) .. (0,-1);
   \draw[thick,->] (4.2,0) -- (4.2,-1);
   \draw[thick,->] (5.4,0) -- (5.4,-1);
   \draw[thick,->] (6,0) -- (6,-1);
\end{tikzpicture}} \right]
\;\; = \;\;
\left[ \;
\hackcenter{
\begin{tikzpicture}[scale=0.8]
 %% Separate lines by 0.6
 %% UPWARD ORIENTED LINES
  \draw[thick] (3,0) .. controls ++(0,1.25) and ++(0,-1.75) .. (0,1.5);
  \draw[thick] (0,0) .. controls ++(0,1) and ++(0,-.7) .. (.6,1.5);
  \draw[thick] (.6,0) .. controls ++(0,1) and ++(0,-.7) .. (1.2,1.5);
  \draw[thick] (1.8,0) .. controls ++(0,1) and ++(0,-.7) .. (2.4,1.5);
  \draw[thick] (2.4,0) .. controls ++(0,1) and ++(0,-.7) .. (3,1.5);
  \node at (1.2,.25) {$\dots$};
  \node at (1.8,1.15) {$\dots$};
  %%
  %% DOWNWARD ORIENTED LINES
  \draw[thick] (6,0) .. controls ++(0,1.25) and ++(0,-1.75) .. (3.6,1.5) ;
  \draw[thick] (3.6,0) .. controls ++(0,1) and ++(0,-.7) .. (4.2,1.5);
  \draw[thick] (4.2,0) .. controls ++(0,1) and ++(0,-.7) .. (4.8,1.5);
  \draw[thick] (5.4,0) .. controls ++(0,1) and ++(0,-.7) .. (6,1.5);
  \node at (4.8,.25) {$\dots$};
  \node at (5.4,1.15) {$\dots$};
  %%
  %% BLUE LINES
    \draw[blue, dotted] (-0.4,0) -- (6.4,0);
    \draw[blue, dotted] (-0.4,1.5) -- (6.4,1.5);
  %% TOP LEVEL
   \draw[thick,<-] (3.6,1.5) .. controls ++(0,.55) and ++(0,-1.35) .. (0,3.5);
   \draw[thick,<-] (4.2,1.5) .. controls ++(0,.75) and ++(0,-1.15) .. (.6,3.5);
   \draw[thick,<-] (4.8,1.5) .. controls ++(0,.85) and ++(0,-.95) .. (1.2,3.5);;
   \draw[thick,<-] (6.0,1.5) .. controls ++(0,1.05) and ++(0,-.75) .. (2.4,3.5);;
   \draw[thick,->] (3.0,1.5) .. controls ++(0.1,0.65) and ++(-0.1,-.85) .. (6,3.5);
   \draw[thick,->] (2.4,1.5) .. controls ++(0.1,0.85) and ++(-0.1,-.65) .. (5.4,3.5);
   \draw[thick,->] (1.2,1.5) .. controls ++(0.1,1.05) and ++(-0.1,-.85) .. (4.2,3.5);
   \draw[thick,->] (0.6,1.5) .. controls ++(0.1,1.25) and ++(-0.1,-.65) .. (3.6,3.5);
   \draw[thick,->] (0.0,1.5) .. controls ++(0.1,1.35) and ++(-0.1,-.45) .. (3,3.5);
     \filldraw  (.05,3.2) circle (2pt);
     \node at (-0.2,3.3) {$\scs a$};
  %% BOTTOM LEVEL
   \draw[thick,<-]  (3,0) .. controls ++(0,-.65) and ++(0,.85) .. (6,-2);
   \draw[thick,<-]  (2.4,0) .. controls ++(0,-.75) and ++(0,.75) .. (5.4,-2);
   \draw[thick,<-]  (1.8,0) .. controls ++(0,-.85) and ++(0,.65) .. (4.8,-2);
   \draw[thick,<-]  (.6,0) .. controls ++(0,-1.15) and ++(0,.65) .. (3.6,-2);
   \draw[thick,<-]  (0,0) .. controls ++(0,-1.25) and ++(0,.45) .. (3,-2);
  \draw[thick ,->] (3.6,0) .. controls ++(0,-.55) and ++(0,1.35) .. (0,-2);
   \draw[thick,->] (4.2,0) .. controls ++(0,-.75) and ++(0,1.15) .. (0.6,-2);
   \draw[thick,->] (5.4,0) .. controls ++(0,-.85) and ++(0,.95) .. (1.8,-2);
   \draw[thick,->] (6,0)   .. controls ++(0,-1.05) and ++(0,.75) .. (2.4,-2);
\end{tikzpicture}} \right] ~.
\]
Then using triple point moves to slide the crossings in the middle third of the diagram through the strands at the top third, the diagram becomes
\[
\;\; = \;\;
\left[ \;
\hackcenter{
\begin{tikzpicture}[scale=0.8]
 %% Separate lines by 0.6
 %% UPWARD ORIENTED LINES
  \draw[thick,->] (6,2) .. controls ++(0,1.25) and ++(0,-1.75) .. (3,3.5);
  \draw[thick,->] (3,2) .. controls ++(0,1) and ++(0,-.7) .. (3.6,3.5);
  \draw[thick,->] (3.6,2) .. controls ++(0,1) and ++(0,-.7) .. (4.2,3.5);
  \draw[thick,->] (4.8,2) .. controls ++(0,1) and ++(0,-.7) .. (5.4,3.5);
  \draw[thick,->] (5.4,2) .. controls ++(0,1) and ++(0,-.7) .. (6,3.5);
  \node at (4.2,2.35) {$\dots$};
  \node at (4.8,3.15) {$\dots$};
  %%
  %% DOWNWARD ORIENTED LINES
  \draw[thick] (2.4,2) .. controls ++(0,1.25) and ++(0,-1.75) .. (0,3.5) ;
  \draw[thick] (0,2) .. controls ++(0,1) and ++(0,-.7) .. (.6,3.5);
  \draw[thick] (.6,2) .. controls ++(0,1) and ++(0,-.7) .. (1.2,3.5);
  \draw[thick] (1.8,2) .. controls ++(0,1) and ++(0,-.7) .. (2.4,3.5);
  \node at (1.2, 2.35) {$\dots$};
  \node at (1.8,3.15) {$\dots$};
%  %%
  %% BLUE LINES
    \draw[blue, dotted] (-0.4,0) -- (6.4,0);
    \draw[blue, dotted] (-0.4,2) -- (6.4,2);
  %% TOP LEVEL - now mid level
   \draw[thick,<-] (3.6,0) .. controls ++(0,.55) and ++(0,-1.35) .. (0,2);
   \draw[thick,<-] (4.2,0) .. controls ++(0,.75) and ++(0,-1.15) .. (.6,2);
   \draw[thick,<-] (5.4,0) .. controls ++(0,.85) and ++(0,-.95) .. (1.8,2);;
   \draw[thick,<-] (6.0,0) .. controls ++(0,1.05) and ++(0,-.75) .. (2.4,2);;
   \draw[thick] (3.0,0) .. controls ++(0.1,0.65) and ++(-0.1,-.85) .. (6,2);
   \draw[thick] (2.4,0) .. controls ++(0.1,0.85) and ++(-0.1,-.65) .. (5.4,2);
   \draw[thick] (1.8,0) .. controls ++(0.1,1.05) and ++(-0.1,-.85) .. (4.8,2);
   \draw[thick] (0.6,0) .. controls ++(0.1,1.25) and ++(-0.1,-.65) .. (3.6,2);
   \draw[thick] (0.0,0) .. controls ++(0.1,1.35) and ++(-0.1,-.45) .. (3,2);
   \node at (1.2,.15) {$\dots$};
   \node at (4.8,.15) {$\dots$};
     \filldraw  (.05,3.05) circle (2pt);
     \node at (-0.2,3.15) {$\scs a$};
  %% BOTTOM LEVEL
   \draw[thick,<-]  (3,0) .. controls ++(0,-.65) and ++(0,.85) .. (6,-2);
   \draw[thick,<-]  (2.4,0) .. controls ++(0,-.75) and ++(0,.75) .. (5.4,-2);
   \draw[thick,<-]  (1.8,0) .. controls ++(0,-.85) and ++(0,.65) .. (4.8,-2);
   \draw[thick,<-]  (.6,0) .. controls ++(0,-1.15) and ++(0,.65) .. (3.6,-2);
   \draw[thick,<-]  (0,0) .. controls ++(0,-1.25) and ++(0,.45) .. (3,-2);
  \draw[thick ,->] (3.6,0) .. controls ++(0,-.55) and ++(0,1.35) .. (0,-2);
   \draw[thick,->] (4.2,0) .. controls ++(0,-.75) and ++(0,1.15) .. (0.6,-2);
   \draw[thick,->] (5.4,0) .. controls ++(0,-.85) and ++(0,.95) .. (1.8,-2);
   \draw[thick,->] (6,0)   .. controls ++(0,-1.05) and ++(0,.75) .. (2.4,-2);
    \node at (1.2,-1.75) {$\dots$};
   \node at (4.2,-1.75) {$\dots$};
\end{tikzpicture}} \right]
\;\; = \;\;
\left[ \;
\hackcenter{
\begin{tikzpicture}[scale=0.8]
 %% Separate lines by 0.6
 %% UPWARD ORIENTED LINES
  \draw[thick,->] (6,2) .. controls ++(0,1.25) and ++(0,-1.75) .. (3,3.5);
  \draw[thick,->] (3,2) .. controls ++(0,1) and ++(0,-.7) .. (3.6,3.5);
  \draw[thick,->] (3.6,2) .. controls ++(0,1) and ++(0,-.7) .. (4.2,3.5);
  \draw[thick,->] (4.8,2) .. controls ++(0,1) and ++(0,-.7) .. (5.4,3.5);
  \draw[thick,->] (5.4,2) .. controls ++(0,1) and ++(0,-.7) .. (6,3.5);
  \node at (4.2,2.35) {$\dots$};
  \node at (4.8,3.15) {$\dots$};
  %%
  %% DOWNWARD ORIENTED LINES
  \draw[thick] (2.4,2) .. controls ++(0,1.25) and ++(0,-1.75) .. (0,3.5) ;
  \draw[thick] (0,2) .. controls ++(0,1) and ++(0,-.7) .. (.6,3.5);
  \draw[thick] (.6,2) .. controls ++(0,1) and ++(0,-.7) .. (1.2,3.5);
  \draw[thick] (1.8,2) .. controls ++(0,1) and ++(0,-.7) .. (2.4,3.5);
  \node at (1.2, 2.35) {$\dots$};
  \node at (1.8,3.15) {$\dots$};
%  %%
  %% BLUE LINES
    \draw[blue, dotted] (-0.4,3.5) -- (6.4,3.5);
    \draw[blue, dotted] (-0.4,2) -- (6.4,2);
  %% TOP LEVEL - now mid level
   \draw[thick,<-] (3.6,0) .. controls ++(0,.55) and ++(0,-1.35) .. (0,2);
   \draw[thick,<-] (4.2,0) .. controls ++(0,.75) and ++(0,-1.15) .. (.6,2);
   \draw[thick,<-] (5.4,0) .. controls ++(0,.85) and ++(0,-.95) .. (1.8,2);;
   \draw[thick,<-] (6.0,0) .. controls ++(0,1.05) and ++(0,-.75) .. (2.4,2);;
   \draw[thick] (3.0,0) .. controls ++(0.1,0.65) and ++(-0.1,-.85) .. (6,2);
   \draw[thick] (2.4,0) .. controls ++(0.1,0.85) and ++(-0.1,-.65) .. (5.4,2);
   \draw[thick] (1.8,0) .. controls ++(0.1,1.05) and ++(-0.1,-.85) .. (4.8,2);
   \draw[thick] (0.6,0) .. controls ++(0.1,1.25) and ++(-0.1,-.65) .. (3.6,2);
   \draw[thick] (0.0,0) .. controls ++(0.1,1.35) and ++(-0.1,-.45) .. (3,2);
   \node at (1.2,.15) {$\dots$};
   \node at (4.8,.15) {$\dots$};
     \filldraw  (.05,3.05) circle (2pt);
     \node at (-0.2,3.15) {$\scs a$};
  %% OLD BOTTOM LEVEL - new top layer
   \draw[thick,<-]  (3.0,5.5) .. controls ++(0,-.65) and ++(0,.85) .. (6,3.5);
   \draw[thick,<-]  (2.4,5.5) .. controls ++(0,-.75) and ++(0,.75) .. (5.4,3.5);
   \draw[thick,<-]  (1.2,5.5) .. controls ++(0,-.85) and ++(0,.65) .. (4.2,3.5);
   \draw[thick,<-]  (0.6,5.5) .. controls ++(0,-1.15) and ++(0,.65) ..(3.6,3.5);
   \draw[thick,<-]  (0.0,5.5) .. controls ++(0,-1.25) and ++(0,.45) ..(3.0,3.5);
  \draw[thick ,->]  (3.6,5.5) .. controls ++(0,-.55) and ++(0,1.35) ..(0.0,3.5);
   \draw[thick,->]  (4.2,5.5) .. controls ++(0,-.75) and ++(0,1.15) ..(0.6,3.5);
   \draw[thick,->]  (4.8,5.5) .. controls ++(0,-.85) and ++(0,.95) .. (1.2,3.5);
   \draw[thick,->]  (6.0,5.5) .. controls ++(0,-1.05) and ++(0,.75) ..(2.4,3.5);
    \node at (1.8,3.75) {$\dots$};
   \node at (4.8,3.75) {$\dots$};
\end{tikzpicture}} \right]
\]
Sliding one layer of crossings from the top third to the bottom third using the trace relation gives
\[
\;\; = \;\;
\left[ \;
\hackcenter{
\begin{tikzpicture}[scale=0.8]
 %% Separate lines by 0.6
 %% UPWARD ORIENTED LINES
  \draw[thick,->] (6,2) .. controls ++(0,.85) and ++(0,-.9) .. (3,3);
  \draw[thick,->] (3,2) .. controls ++(0,.6) and ++(0,-.5) .. (3.6,3);
  \draw[thick,->] (3.6,2) .. controls ++(0,.6) and ++(0,-.5) .. (4.2,3);
  \draw[thick,->] (4.8,2) .. controls ++(0,.6) and ++(0,-.5) .. (5.4,3);
  \draw[thick,->] (5.4,2) .. controls ++(0,.6) and ++(0,-.5) .. (6,3);
  \node at (4.2,2.15) {$\dots$};
  %\node at (4.8,3.15) {$\dots$};
  %%
  %% DOWNWARD ORIENTED LINES
  \draw[thick] (2.4,2) .. controls ++(0,.85) and ++(0,-.9) .. (0,3) ;
  \draw[thick] (0,2) .. controls ++(0,.6) and ++(0,-.5) .. (.6,3);
  \draw[thick] (.6,2) .. controls ++(0,.6) and ++(0,-.5) .. (1.2,3);
  \draw[thick] (1.8,2) .. controls ++(0,.6) and ++(0,-.5) .. (2.4,3);
  \node at (1.2, 2.15) {$\dots$};
  \node at (1.8,3.15) {$\dots$};
%  %%
  %% BLUE LINES
    \draw[blue, dotted] (-0.4,3) -- (6.4,3);
    \draw[blue, dotted] (-0.4,2) -- (6.4,2);
  %% TOP LEVEL
   \draw[thick,<-] (3.6,0) .. controls ++(0,.55) and ++(0,-1.35) .. (0,2);
   \draw[thick,<-] (4.2,0) .. controls ++(0,.75) and ++(0,-1.15) .. (.6,2);
   \draw[thick,<-] (5.4,0) .. controls ++(0,.85) and ++(0,-.95) .. (1.8,2);;
   \draw[thick,<-] (6.0,0) .. controls ++(0,1.05) and ++(0,-.75) .. (2.4,2);;
   \draw[thick] (3.0,0) .. controls ++(0.1,0.65) and ++(-0.1,-.85) .. (6,2);
   \draw[thick] (2.4,0) .. controls ++(0.1,0.85) and ++(-0.1,-.65) .. (5.4,2);
   \draw[thick] (1.8,0) .. controls ++(0.1,1.05) and ++(-0.1,-.85) .. (4.8,2);
   \draw[thick] (0.6,0) .. controls ++(0.1,1.25) and ++(-0.1,-.65) .. (3.6,2);
   \draw[thick] (0.0,0) .. controls ++(0.1,1.35) and ++(-0.1,-.45) .. (3,2);
   \node at (1.2,.15) {$\dots$};
   \node at (4.8,.15) {$\dots$};
     \filldraw  (0,3.1) circle (2pt);
     \node at (-0.24,3.2) {$\scs a$};
  %% OLD BOTTOM LEVEL - new top layer
   \draw[thick,<-]  (3.0,5) .. controls ++(0,-.65) and ++(0,.85) .. (6,3);
   \draw[thick,<-]  (2.4,5) .. controls ++(0,-.75) and ++(0,.75) .. (5.4,3);
   \draw[thick,<-]  (1.2,5) .. controls ++(0,-.85) and ++(0,.65) .. (4.2,3);
   \draw[thick,<-]  (0.6,5) .. controls ++(0,-1.15) and ++(0,.65) ..(3.6,3);
   \draw[thick,<-]  (0.0,5) .. controls ++(0,-1.25) and ++(0,.45) ..(3.0,3);
  \draw[thick ,->]  (3.6,5) .. controls ++(0,-.55) and ++(0,1.35) ..(0.0,3);
   \draw[thick,->]  (4.2,5) .. controls ++(0,-.75) and ++(0,1.15) ..(0.6,3);
   \draw[thick,->]  (4.8,5) .. controls ++(0,-.85) and ++(0,.95) .. (1.2,3);
   \draw[thick,->]  (6.0,5) .. controls ++(0,-1.05) and ++(0,.75) ..(2.4,3);
    \node at (1.8,4.75) {$\dots$};
   \node at (5.4,4.75) {$\dots$};
\end{tikzpicture}} \right]
\;\; = \;\;
\left[ \;
\hackcenter{
\begin{tikzpicture}[scale=0.8]
 %% Separate lines by 0.6
 %% UPWARD ORIENTED LINES
  \draw[thick,->] (6,2) .. controls ++(0,.85) and ++(0,-.9) .. (3,3);
  \draw[thick,->] (3,2) .. controls ++(0,.6) and ++(0,-.5) .. (3.6,3);
  \draw[thick,->] (3.6,2) .. controls ++(0,.6) and ++(0,-.5) .. (4.2,3);
  \draw[thick,->] (4.8,2) .. controls ++(0,.6) and ++(0,-.5) .. (5.4,3);
  \draw[thick,->] (5.4,2) .. controls ++(0,.6) and ++(0,-.5) .. (6,3);
  \node at (4.2,2.15) {$\dots$};
  %\node at (4.8,3.15) {$\dots$};
  %%
  %% DOWNWARD ORIENTED LINES
  \draw[thick] (2.4,2) .. controls ++(0,.85) and ++(0,-.9) .. (0,3) ;
  \draw[thick] (0,2) .. controls ++(0,.6) and ++(0,-.5) .. (.6,3);
  \draw[thick] (.6,2) .. controls ++(0,.6) and ++(0,-.5) .. (1.2,3);
  \draw[thick] (1.8,2) .. controls ++(0,.6) and ++(0,-.5) .. (2.4,3);
  \node at (1.2, 2.15) {$\dots$};
  \node at (1.8,3.15) {$\dots$};
%  %%
  %% BLUE LINES
    \draw[blue, dotted] (-0.4,3) -- (6.4,3);
    \draw[blue, dotted] (-0.4,2) -- (6.4,2);
  %%  bottom
   \draw[thick] (3.6,0) .. controls ++(0,.55) and ++(0,-1.35) .. (0,2);
   \draw[thick] (4.2,0) .. controls ++(0,.75) and ++(0,-1.15) .. (.6,2);
   \draw[thick] (5.4,0) .. controls ++(0,.85) and ++(0,-.95) .. (1.8,2);;
   \draw[thick] (6.0,0) .. controls ++(0,1.05) and ++(0,-.75) .. (2.4,2);;
   \draw[thick] (3.0,0) .. controls ++(0.1,0.65) and ++(-0.1,-.85) .. (6,2);
   \draw[thick] (2.4,0) .. controls ++(0.1,0.85) and ++(-0.1,-.65) .. (5.4,2);
   \draw[thick] (1.8,0) .. controls ++(0.1,1.05) and ++(-0.1,-.85) .. (4.8,2);
   \draw[thick] (0.6,0) .. controls ++(0.1,1.25) and ++(-0.1,-.65) .. (3.6,2);
   \draw[thick] (0.0,0) .. controls ++(0.1,1.35) and ++(-0.1,-.45) .. (3,2);
   \node at (1.2,.15) {$\dots$};
   \node at (4.8,.15) {$\dots$};
     \filldraw  (0,3.1) circle (2pt);
     \node at (-0.24,3.2) {$\scs a$};
  %%  top layer
   \draw[thick,<-]  (3.6,5) .. controls ++(0,-.65) and ++(0,.85) .. (6,3);
   \draw[thick,<-]  (3,5) .. controls ++(0,-.75) and ++(0,.75) .. (5.4,3);
   \draw[thick,<-]  (1.8,5) .. controls ++(0,-.85) and ++(0,.65) .. (4.2,3);
   \draw[thick,<-]  (1.2,5) .. controls ++(0,-1.15) and ++(0,.65) ..(3.6,3);
   \draw[thick,<-]  (0.6,5) .. controls ++(0,-1.25) and ++(0,.45) ..(3.0,3);
  \draw[thick ,->]  (0,5) -- (0.0,3);
   \draw[thick,->]  (4.2,5) .. controls ++(0,-.75) and ++(0,1.15) ..(0.6,3);
   \draw[thick,->]  (4.8,5) .. controls ++(0,-.85) and ++(0,.95) .. (1.2,3);
   \draw[thick,->]  (6.0,5) .. controls ++(0,-1.05) and ++(0,.75) ..(2.4,3);
    \node at (2.4,4.75) {$\dots$};
   \node at (5.4,4.75) {$\dots$};
      %% SUPER Bottom
    \draw[thick,->]  (3.6,0) .. controls ++(0,-.65) and ++(0,.85) .. (0,-1);
    \draw[thick,<-]     (0,0) .. controls ++(0,-.6) and ++(0,.5) .. (.6,-1);
    \draw[thick,<-]    (0.6,0) .. controls ++(0,-.6) and ++(0,.5) .. (1.2,-1);
    \draw[thick,<-]    (1.8,0) .. controls ++(0,-.6) and ++(0,.5) .. (2.4,-1);
    \draw[thick,<-]    (2.4,0) .. controls ++(0,-.6) and ++(0,.5) .. (3,-1);
   \draw[thick,<-]    (3,0) .. controls ++(0,-.6) and ++(0,.5) .. (3.6,-1);
    \draw[thick,<-] (4.2,-1) -- (4.2,0);
    \draw[thick,<-] (5.4,-1) -- (5.4,0);
    \draw[thick,<-] (6,-1) -- (6,0);
\end{tikzpicture}} \right]
\]
Applying the first equation in \eqref{eq:rel2} in the bottom third of the diagram only one term survives
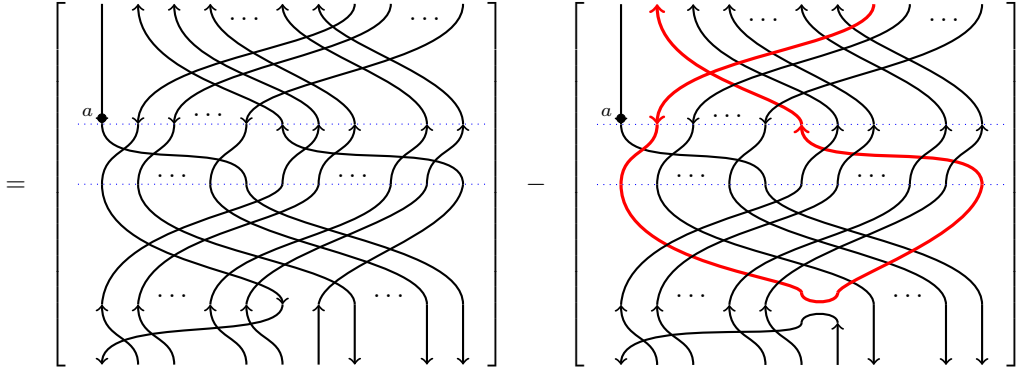
\[
\;\; = \;\;
\left[ \;
\hackcenter{
\begin{tikzpicture}[scale=0.8]
 %% Separate lines by 0.6
 %% UPWARD ORIENTED LINES
  \draw[thick,->] (6,2) .. controls ++(0,.85) and ++(0,-.9) .. (3,3);
  \draw[thick,->] (3,2) .. controls ++(0,.6) and ++(0,-.5) .. (3.6,3);
  \draw[thick,->] (3.6,2) .. controls ++(0,.6) and ++(0,-.5) .. (4.2,3);
  \draw[thick,->] (4.8,2) .. controls ++(0,.6) and ++(0,-.5) .. (5.4,3);
  \draw[thick,->] (5.4,2) .. controls ++(0,.6) and ++(0,-.5) .. (6,3);
  \node at (4.2,2.15) {$\dots$};
  %\node at (4.8,3.15) {$\dots$};
  %%
  %% DOWNWARD ORIENTED LINES
  \draw[thick] (2.4,2) .. controls ++(0,.85) and ++(0,-.9) .. (0,3) ;
  \draw[thick] (0,2) .. controls ++(0,.6) and ++(0,-.5) .. (.6,3);
  \draw[thick] (.6,2) .. controls ++(0,.6) and ++(0,-.5) .. (1.2,3);
  \draw[thick] (1.8,2) .. controls ++(0,.6) and ++(0,-.5) .. (2.4,3);
  \node at (1.2, 2.15) {$\dots$};
  \node at (1.8,3.15) {$\dots$};
%  %%
  %% BLUE LINES
    \draw[blue, dotted] (-0.4,3) -- (6.4,3);
    \draw[blue, dotted] (-0.4,2) -- (6.4,2);
  %%  bottom
   \draw[thick,<-] (3.0,0) .. controls ++(0,.55) and ++(0,-1.35) .. (0,2);
   \draw[thick] (4.2,0) .. controls ++(0,.75) and ++(0,-1.15) .. (.6,2);
   \draw[thick] (5.4,0) .. controls ++(0,.85) and ++(0,-.95) .. (1.8,2);;
   \draw[thick] (6.0,0) .. controls ++(0,1.05) and ++(0,-.75) .. (2.4,2);;
   \draw[thick] (3.6,0) .. controls ++(0.1,0.65) and ++(-0.1,-.85) .. (6,2);
   \draw[thick] (2.4,0) .. controls ++(0.1,0.85) and ++(-0.1,-.65) .. (5.4,2);
   \draw[thick] (1.8,0) .. controls ++(0.1,1.05) and ++(-0.1,-.85) .. (4.8,2);
   \draw[thick] (0.6,0) .. controls ++(0.1,1.25) and ++(-0.1,-.65) .. (3.6,2);
   \draw[thick] (0.0,0) .. controls ++(0.1,1.35) and ++(-0.1,-.45) .. (3,2);
   \node at (1.2,.15) {$\dots$};
   \node at (4.8,.15) {$\dots$};
     \filldraw  (0,3.1) circle (2pt);
     \node at (-0.24,3.2) {$\scs a$};
  %%  top layer
   \draw[thick,<-]  (3.6,5) .. controls ++(0,-.65) and ++(0,.85) .. (6,3);
   \draw[thick,<-]  (3,5) .. controls ++(0,-.75) and ++(0,.75) .. (5.4,3);
   \draw[thick,<-]  (1.8,5) .. controls ++(0,-.85) and ++(0,.65) .. (4.2,3);
   \draw[thick,<-]  (1.2,5) .. controls ++(0,-1.15) and ++(0,.65) ..(3.6,3);
   \draw[thick,<-]  (0.6,5) .. controls ++(0,-1.25) and ++(0,.45) ..(3.0,3);
  \draw[thick ,->]  (0,5) -- (0.0,3);
   \draw[thick,->]  (4.2,5) .. controls ++(0,-.75) and ++(0,1.15) ..(0.6,3);
   \draw[thick,->]  (4.8,5) .. controls ++(0,-.85) and ++(0,.95) .. (1.2,3);
   \draw[thick,->]  (6.0,5) .. controls ++(0,-1.05) and ++(0,.75) ..(2.4,3);
    \node at (2.4,4.75) {$\dots$};
   \node at (5.4,4.75) {$\dots$};
      %% SUPER Bottom
    \draw[thick,->]  (3.0,0) .. controls ++(0,-.65) and ++(0,.85) .. (0,-1);
    \draw[thick,<-]     (0,0) .. controls ++(0,-.6) and ++(0,.5) .. (.6,-1);
    \draw[thick,<-]    (0.6,0) .. controls ++(0,-.6) and ++(0,.5) .. (1.2,-1);
    \draw[thick,<-]    (1.8,0) .. controls ++(0,-.6) and ++(0,.5) .. (2.4,-1);
    \draw[thick,<-]    (2.4,0) .. controls ++(0,-.6) and ++(0,.5) .. (3,-1);
   \draw[thick,<-]    (3.6,0) .. controls ++(0,-.6) and ++(0,.5) .. (3.6,-1);
    \draw[thick,<-] (4.2,-1) -- (4.2,0);
    \draw[thick,<-] (5.4,-1) -- (5.4,0);
    \draw[thick,<-] (6,-1) -- (6,0);
\end{tikzpicture}} \right]
\;\; - \;\;
\left[ \;
\hackcenter{
\begin{tikzpicture}[scale=0.8]
 %% Separate lines by 0.6
 %% UPWARD ORIENTED LINES
  \draw[very thick, red,->] (6,2) .. controls ++(0,.85) and ++(0,-.9) .. (3,3);
  \draw[thick,->] (3,2) .. controls ++(0,.6) and ++(0,-.5) .. (3.6,3);
  \draw[thick,->] (3.6,2) .. controls ++(0,.6) and ++(0,-.5) .. (4.2,3);
  \draw[thick,->] (4.8,2) .. controls ++(0,.6) and ++(0,-.5) .. (5.4,3);
  \draw[thick,->] (5.4,2) .. controls ++(0,.6) and ++(0,-.5) .. (6,3);
  \node at (4.2,2.15) {$\dots$};
  %\node at (4.8,3.15) {$\dots$};
  %%
  %% DOWNWARD ORIENTED LINES
  \draw[thick] (2.4,2) .. controls ++(0,.85) and ++(0,-.9) .. (0,3) ;
  \draw[very thick, red] (0,2) .. controls ++(0,.6) and ++(0,-.5) .. (.6,3);
  \draw[thick] (.6,2) .. controls ++(0,.6) and ++(0,-.5) .. (1.2,3);
  \draw[thick] (1.8,2) .. controls ++(0,.6) and ++(0,-.5) .. (2.4,3);
  \node at (1.2, 2.15) {$\dots$};
  \node at (1.8,3.15) {$\dots$};
%  %%
  %% BLUE LINES
    \draw[blue, dotted] (-0.4,3) -- (6.4,3);
    \draw[blue, dotted] (-0.4,2) -- (6.4,2);
  %%  bottom
   \draw[very thick, red] (3.0,.2) .. controls ++(0,.25) and ++(0,-1.35) .. (0,2);
   \draw[thick] (4.2,0) .. controls ++(0,.75) and ++(0,-1.15) .. (.6,2);
   \draw[thick] (5.4,0) .. controls ++(0,.85) and ++(0,-.95) .. (1.8,2);;
   \draw[thick] (6.0,0) .. controls ++(0,1.05) and ++(0,-.75) .. (2.4,2);;
   \draw[very thick, red] (3.6,.2) .. controls ++(0.1,0.25) and ++(-0.1,-.85) .. (6,2);
   \draw[thick] (2.4,0) .. controls ++(0.1,0.85) and ++(-0.1,-.65) .. (5.4,2);
   \draw[thick] (1.8,0) .. controls ++(0.1,1.05) and ++(-0.1,-.85) .. (4.8,2);
   \draw[thick] (0.6,0) .. controls ++(0.1,1.25) and ++(-0.1,-.65) .. (3.6,2);
   \draw[thick] (0.0,0) .. controls ++(0.1,1.35) and ++(-0.1,-.45) .. (3,2);
   \node at (1.2,.15) {$\dots$};
   \node at (4.8,.15) {$\dots$};
     \filldraw  (0,3.1) circle (2pt);
     \node at (-0.24,3.2) {$\scs a$};
  %%  top layer
   \draw[thick,<-]  (3.6,5) .. controls ++(0,-.65) and ++(0,.85) .. (6,3);
   \draw[thick,<-]  (3,5) .. controls ++(0,-.75) and ++(0,.75) .. (5.4,3);
   \draw[thick,<-]  (1.8,5) .. controls ++(0,-.85) and ++(0,.65) .. (4.2,3);
   \draw[thick,<-]  (1.2,5) .. controls ++(0,-1.15) and ++(0,.65) ..(3.6,3);
   \draw[very thick, red,<-]  (0.6,5) .. controls ++(0,-1.25) and ++(0,.45) ..(3.0,3);
  \draw[thick ,->]  (0,5) -- (0.0,3);
   \draw[very thick, red,->]  (4.2,5) .. controls ++(0,-.75) and ++(0,1.15) ..(0.6,3);
   \draw[thick,->]  (4.8,5) .. controls ++(0,-.85) and ++(0,.95) .. (1.2,3);
   \draw[thick,->]  (6.0,5) .. controls ++(0,-1.05) and ++(0,.75) ..(2.4,3);
    \node at (2.4,4.75) {$\dots$};
   \node at (5.4,4.75) {$\dots$};
      %% SUPER Bottom
      \draw[very thick, red]    (3.0,.2) .. controls ++(0,-.2) and ++(0,-.2) .. (3.6,.2);
      \draw[thick ]    (3.0,-.3) .. controls ++(0,.2) and ++(0,.2) .. (3.6,-.3);
    \draw[thick,->]  (3.0,-.3) .. controls ++(0,-.35) and ++(0,.85) .. (0,-1);
    \draw[thick,<-]     (0,0) .. controls ++(0,-.6) and ++(0,.5) .. (.6,-1);
    \draw[thick,<-]    (0.6,0) .. controls ++(0,-.6) and ++(0,.5) .. (1.2,-1);
    \draw[thick,<-]    (1.8,0) .. controls ++(0,-.6) and ++(0,.5) .. (2.4,-1);
    \draw[thick,<-]    (2.4,0) .. controls ++(0,-.6) and ++(0,.5) .. (3,-1);
   \draw[thick,<-]    (3.6,-.3) .. controls ++(0,-.3) and ++(0,.5) .. (3.6,-1);
    \draw[thick,<-] (4.2,-1) -- (4.2,0);
    \draw[thick,<-] (5.4,-1) -- (5.4,0);
    \draw[thick,<-] (6,-1) -- (6,0);
\end{tikzpicture}} \right]
\]
since the second diagram can be simplified to one containing a left twist curl which is zero.  It is easy to see that the same thing happens as the left most downward oriented strand at the bottom is pulled through all of the upward oriented lines.

To further simplify the diagram, slide all of the crossings involving the second downward oriented line from the top third of the diagram to the bottom third and simplify using the first equation in \eqref{eq:rel2} as above.  Each time this equation is applied the resolution term can be simplified so that it contains a left twist curl.  Hence, all of these terms vanish and we are left with the following
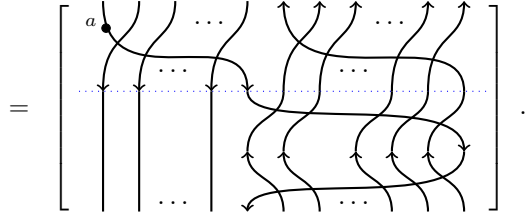
\[
\;\; = \;\;
\left[ \;
\hackcenter{
\begin{tikzpicture}[scale=0.8]
 %% Separate lines by 0.6
 %% UPWARD ORIENTED LINES
  \draw[thick,->] (6,2) .. controls ++(0,1.25) and ++(0,-1.75) .. (3,3.5);
  \draw[thick,->] (3,2) .. controls ++(0,1) and ++(0,-.7) .. (3.6,3.5);
  \draw[thick,->] (3.6,2) .. controls ++(0,1) and ++(0,-.7) .. (4.2,3.5);
  \draw[thick,->] (4.8,2) .. controls ++(0,1) and ++(0,-.7) .. (5.4,3.5);
  \draw[thick,->] (5.4,2) .. controls ++(0,1) and ++(0,-.7) .. (6,3.5);
  \node at (4.2,2.35) {$\dots$};
  \node at (4.8,3.15) {$\dots$};
  %%
  %% DOWNWARD ORIENTED LINES
  \draw[thick,<-] (2.4,2) .. controls ++(0,1.25) and ++(0,-1.75) .. (0,3.5) ;
  \draw[thick,<-] (0,2) .. controls ++(0,1) and ++(0,-.7) .. (.6,3.5);
  \draw[thick,<-] (.6,2) .. controls ++(0,1) and ++(0,-.7) .. (1.2,3.5);
  \draw[thick,<-] (1.8,2) .. controls ++(0,1) and ++(0,-.7) .. (2.4,3.5);
  \node at (1.2, 2.35) {$\dots$};
  \node at (1.8,3.15) {$\dots$};
%  %%
  %% BLUE LINES
    \draw[blue, dotted] (-0.4,2) -- (6.4,2);
  %% bottom
   \draw[thick,<-] (6.0,1) .. controls ++(0,1.05) and ++(0,-.75) .. (2.4,2);
   \draw[thick,->] (6.0,1) .. controls ++(0,-1.05) and ++(0,.75) .. (2.4,0);
   \draw[thick,->] (3,0) .. controls ++(0,.6) and ++(0,-.5) .. (2.4,1);
   \draw[thick ]   (2.4,1) .. controls ++(0,.6) and ++(0,-.5) .. (3,2);
   \draw[thick,->] (3.6,0) .. controls ++(0,.6) and ++(0,-.5) .. (3,1);
   \draw[thick ]   (3,1) .. controls ++(0,.6) and ++(0,-.5) .. (3.6,2);
   \draw[thick,->] (4.8,0) .. controls ++(0,.6) and ++(0,-.5) .. (4.2,1);
   \draw[thick ]   (4.2,1) .. controls ++(0,.6) and ++(0,-.5) .. (4.8,2);
    \draw[thick,->] (5.4,0) .. controls ++(0,.6) and ++(0,-.5) .. (4.8,1);
   \draw[thick ]   (4.8,1) .. controls ++(0,.6) and ++(0,-.5) .. (5.4,2);
    \draw[thick,->] (6,0) .. controls ++(0,.6) and ++(0,-.5) .. (5.4,1);
   \draw[thick ]   (5.4,1) .. controls ++(0,.6) and ++(0,-.5) .. (6,2);
   \draw[thick] (1.8,0) -- (1.8,2);
   \draw[thick] (0.6,0) -- (0.6,2);
   \draw[thick] (0.0,0) -- (0,2);
   \node at (1.2,.15) {$\dots$};
   \node at (4.2,.15) {$\dots$};
     \filldraw  (.05,3.05) circle (2pt);
     \node at (-0.2,3.15) {$\scs a$};
\end{tikzpicture}} \right]~.
\]
The result follows by Lemma~\ref{lem:mn-onetangle}.
\end{proof}

\begin{lemma} \label{lem_nplusm}
For $1 \leq i < n$ and $a \geq 0$ the identity
\[
[(x_1^a T_{i+1}T_{i+2} \dots T_n)
(T_1 \dots T_n)
(T_{n+1} \dots T_{n+m})
(T_n T_{n-1} \dots T_i)] = h_{n+m} \otimes x_1^a
\]
holds in the trace.
\end{lemma}

\begin{proof}
The left side of this relation can be expressed by the diagram
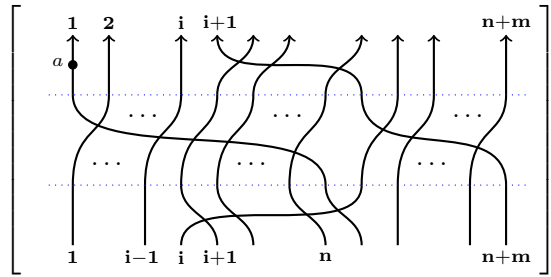
\[
\left[ \;\;\;
\hackcenter{
\begin{tikzpicture}[scale=0.8]
 %% Separate lines by 0.6
 %% DOWNWARD ORIENTED LINES
  \draw[thick] (3,0) .. controls ++(0,1.25) and ++(0,-1.25) .. (-1.2,1.5);
  \draw[thick] (-1.2,0) .. controls ++(0,1) and ++(0,-.7) .. (-.6,1.5);
  \draw[thick] (0,0) .. controls ++(0,1) and ++(0,-.7) .. (.6,1.5);
  \draw[thick] (.6,0) .. controls ++(0,1) and ++(0,-.7) .. (1.2,1.5);
  \draw[thick] (1.2,0) .. controls ++(0,1) and ++(0,-.7) .. (1.8,1.5);
  \draw[thick] (2.4,0) .. controls (2.4,1) and (3,.8) .. (3,1.5);
  \node at (-.6,.35) {$\dots$};
  \node at (0,1.15) {$\dots$};
  \node at (1.8,.35) {$\dots$};
  \node at (2.4,1.15) {$\dots$};
  %%
  %% UPWARD ORIENTED LINES
  \draw[thick] (6,0) .. controls (6,1.25) and (3.6,.25) .. (3.6,1.5)
     node[pos=0.85, shape=coordinate](X){};
  \draw[thick] (3.6,0) .. controls (3.6,1) and (4.2,.8) .. (4.2,1.5);
  \draw[thick] (4.2,0) .. controls (4.2,1) and (4.8,.8) .. (4.8,1.5);
  \draw[thick] (5.4,0) .. controls (5.4,1) and (6,.8) .. (6,1.5);
  \node at (4.8,.35) {$\dots$};
  \node at (5.4,1.15) {$\dots$};
  %%
  %% BLUE LINES
    \draw[blue, dotted] (-1.6,0) -- (6.4,0);
    \draw[blue, dotted] (-1.6,1.5) -- (6.4,1.5);
  %% TOP LEVEL
   \draw[thick,->] (3.6,1.5) .. controls ++(0,1) and ++(0,-1) .. (1.2,2.5);
   \draw[thick,->] (3,1.5) .. controls ++(0,.5) and ++(0,-.5) .. (3.6,2.5);
   \draw[thick,->] (-1.2,1.5) -- (-1.2,2.5);
   \draw[thick,->] (-0.6,1.5) -- (-0.6,2.5);
   %\draw[thick,->] (.6,1.5) -- (.6,2.5);
   \draw[thick,->] (.6,1.5) -- (.6,2.5);
   \draw[thick,->] (1.2,1.5) .. controls ++(0,.5) and ++(0,-.5) .. (1.8,2.5);
   \draw[thick,->] (1.8,1.5) .. controls ++(0,.5) and ++(0,-.5) .. (2.4,2.5);
   \draw[thick,->] (4.2,1.5) -- (4.2,2.5);
   \draw[thick,->] (4.8,1.5) -- (4.8,2.5);
   \draw[thick,->] (6,1.5) -- (6,2.5);
  %% BOTTOM LEVEL
   \draw[thick]  (3.6,0) .. controls ++(0,-1) and ++(0,1) .. (.6,-1);
   \draw[thick] (-1.2,0) -- (-1.2,-1);
   \draw[thick]  (0,0) -- (0,-1);
   \draw[thick]  (.6,0) .. controls ++(0,-.5) and ++(0,+.5) .. (1.2,-1);
   \draw[thick]  (1.2,0) .. controls ++(0,-.5) and ++(0,+.5) .. (1.8,-1);
   \draw[thick]  (2.4,0) .. controls ++(0,-.5) and ++(0,+.5) .. (3,-1);
   \draw[thick]  (3,0) .. controls ++(0,-.5) and ++(0,+.5) .. (3.6,-1);
   \draw[thick] (4.2,0) -- (4.2,-1);
   \draw[thick] (5.4,0) -- (5.4,-1);
   \draw[thick] (6,0) -- (6,-1);
   \node at (-1.2,2.7) {$\scs \mathbf{1}$};
   \node at (-.6,2.7) {$\scs \mathbf{2}$};
   \node at (.6,2.7) {$\scs \mathbf{i}$};
   \node at (1.25,2.7) {$\scs \mathbf{i+1}$};
   \node at (6,2.7) {$\scs \mathbf{n+m}$};
   \node at (-1.2,-1.2) {$\scs \mathbf{1}$};
   \node at (-0.05,-1.2) {$\scs \mathbf{i-1}$};
   \node at (.6,-1.2) {$\scs \mathbf{i}$};
   \node at (1.25,-1.2) {$\scs \mathbf{i+1}$};
   \node at (3,-1.2) {$\scs \mathbf{n}$};
   \node at (6,-1.2) {$\scs \mathbf{n+m}$};
   \filldraw  (-1.2,2) circle (2pt);
   \node at (-1.45,2.05) {$\scs a$};
\end{tikzpicture}}
\right]
\]
The $T_{i+1} \dots T_n$ at the top can be reordered to $T_{n} \dots T_{i+1}$ by first sliding the $T_{i+1}$ around the annulus using the trace relation and then back to the top of diagram using equations \eqref{eq:rel1} to produce $T_{i+2}T_{i+1} T_{i+3} \dots T_n$. Next use the trace relation and equations \eqref{eq:rel1} to bring the crossings $T_{i+2}T_{i+1}$ through the bottom of the diagram back to the top portion of the diagram producing $T_{i+3}T_{i+2}T_{i+1}T_{i+4} \dots T_n$.  Continuing in this way all of the crossings at the top third of the diagram can be reordered to produce
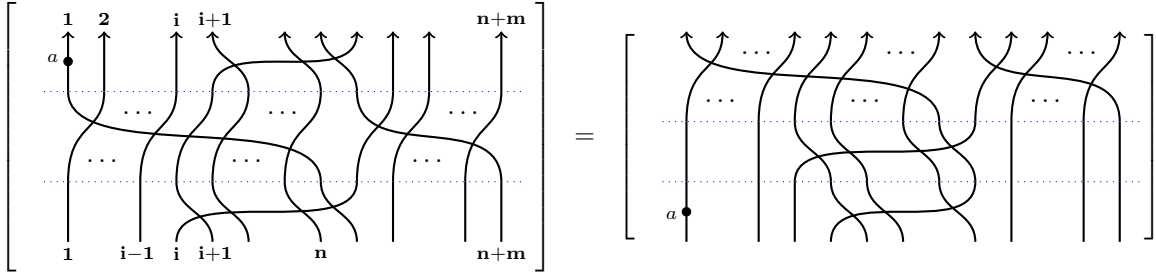
\[
\left[ \;\;\;
\hackcenter{
\begin{tikzpicture}[scale=0.8]
 %% Separate lines by 0.6
 %% DOWNWARD ORIENTED LINES
  \draw[thick] (3,0) .. controls ++(0,1.25) and ++(0,-1.25) .. (-1.2,1.5);
  \draw[thick] (-1.2,0) .. controls ++(0,1) and ++(0,-.7) .. (-.6,1.5);
  \draw[thick] (0,0) .. controls ++(0,1) and ++(0,-.7) .. (.6,1.5);
  \draw[thick] (.6,0) .. controls ++(0,1) and ++(0,-.7) .. (1.2,1.5);
  \draw[thick] (1.2,0) .. controls ++(0,1) and ++(0,-.7) .. (1.8,1.5);
  \draw[thick] (2.4,0) .. controls (2.4,1) and (3,.8) .. (3,1.5);
  \node at (-.6,.35) {$\dots$};
  \node at (0,1.15) {$\dots$};
  \node at (1.8,.35) {$\dots$};
  \node at (2.4,1.15) {$\dots$};
  %%
  %% UPWARD ORIENTED LINES
  \draw[thick] (6,0) .. controls (6,1.25) and (3.6,.25) .. (3.6,1.5)
     node[pos=0.85, shape=coordinate](X){};
  \draw[thick] (3.6,0) .. controls (3.6,1) and (4.2,.8) .. (4.2,1.5);
  \draw[thick] (4.2,0) .. controls (4.2,1) and (4.8,.8) .. (4.8,1.5);
  \draw[thick] (5.4,0) .. controls (5.4,1) and (6,.8) .. (6,1.5);
  \node at (4.8,.35) {$\dots$};
  \node at (5.4,1.15) {$\dots$};
  %%
  %% BLUE LINES
    \draw[blue, dotted] (-1.6,0) -- (6.4,0);
    \draw[blue, dotted] (-1.6,1.5) -- (6.4,1.5);
  %% TOP LEVEL
   \draw[thick,->] (1.2,1.5) .. controls ++(0,1) and ++(0,-1) .. (3.6,2.5);
   \draw[thick,->] (3.6,1.5) .. controls ++(0,.5) and ++(0,-.5) .. (3,2.5);
   \draw[thick,->] (3,1.5) .. controls ++(0,.5) and ++(0,-.5) .. (2.4,2.5);
   \draw[thick,->] (-1.2,1.5) -- (-1.2,2.5);
   \draw[thick,->] (-0.6,1.5) -- (-0.6,2.5);
   %\draw[thick,->] (.6,1.5) -- (.6,2.5);
   \draw[thick,->] (.6,1.5) -- (.6,2.5);
   %\draw[thick,->] (1.2,1.5) .. controls ++(0,.5) and ++(0,-.5) .. (1.8,2.5);
   \draw[thick,->] (1.8,1.5) .. controls ++(0,.5) and ++(0,-.5) .. (1.2,2.5);
   \draw[thick,->] (4.2,1.5) -- (4.2,2.5);
   \draw[thick,->] (4.8,1.5) -- (4.8,2.5);
   \draw[thick,->] (6,1.5) -- (6,2.5);
  %% BOTTOM LEVEL
   \draw[thick]  (3.6,0) .. controls ++(0,-1) and ++(0,1) .. (.6,-1);
   \draw[thick] (-1.2,0) -- (-1.2,-1);
   \draw[thick]  (0,0) -- (0,-1);
   \draw[thick]  (.6,0) .. controls ++(0,-.5) and ++(0,+.5) .. (1.2,-1);
   \draw[thick]  (1.2,0) .. controls ++(0,-.5) and ++(0,+.5) .. (1.8,-1);
   \draw[thick]  (2.4,0) .. controls ++(0,-.5) and ++(0,+.5) .. (3,-1);
   \draw[thick]  (3,0) .. controls ++(0,-.5) and ++(0,+.5) .. (3.6,-1);
   \draw[thick] (4.2,0) -- (4.2,-1);
   \draw[thick] (5.4,0) -- (5.4,-1);
   \draw[thick] (6,0) -- (6,-1);
   \node at (-1.2,2.7) {$\scs \mathbf{1}$};
   \node at (-.6,2.7) {$\scs \mathbf{2}$};
   \node at (.6,2.7) {$\scs \mathbf{i}$};
   \node at (1.25,2.7) {$\scs \mathbf{i+1}$};
   \node at (6,2.7) {$\scs \mathbf{n+m}$};
   \node at (-1.2,-1.2) {$\scs \mathbf{1}$};
   \node at (-0.05,-1.2) {$\scs \mathbf{i-1}$};
   \node at (.6,-1.2) {$\scs \mathbf{i}$};
   \node at (1.25,-1.2) {$\scs \mathbf{i+1}$};
   \node at (3,-1.2) {$\scs \mathbf{n}$};
   \node at (6,-1.2) {$\scs \mathbf{n+m}$};
   \filldraw  (-1.2,2) circle (2pt);
   \node at (-1.45,2.05) {$\scs a$};
\end{tikzpicture}}
\right]
\;\; = \;\;
\left[ \;\;\;
\hackcenter{
\begin{tikzpicture}[scale=0.8]
 %% Separate lines by 0.6
 %% DOWNWARD ORIENTED LINES
  \draw[thick,->] (3,0) .. controls ++(0,1.25) and ++(0,-1.25) .. (-1.2,1.5);
  \draw[thick,->] (-1.2,0) .. controls ++(0,1) and ++(0,-.7) .. (-.6,1.5);
  \draw[thick,->] (0,0) .. controls ++(0,1) and ++(0,-.7) .. (.6,1.5);
  \draw[thick,->] (.6,0) .. controls ++(0,1) and ++(0,-.7) .. (1.2,1.5);
  \draw[thick,->] (1.2,0) .. controls ++(0,1) and ++(0,-.7) .. (1.8,1.5);
  \draw[thick,->] (2.4,0) .. controls (2.4,1) and (3,.8) .. (3,1.5);
  \node at (-.6,.35) {$\dots$};
  \node at (0,1.15) {$\dots$};
  \node at (1.8,.35) {$\dots$};
  \node at (2.4,1.15) {$\dots$};
  %%
  %% Right part of diagram
  \draw[thick,->] (6,0) .. controls (6,1.25) and (3.6,.25) .. (3.6,1.5);
  \draw[thick,->] (3.6,0) .. controls (3.6,1) and (4.2,.8) .. (4.2,1.5);
  \draw[thick,->] (4.2,0) .. controls (4.2,1) and (4.8,.8) .. (4.8,1.5);
  \draw[thick,->] (5.4,0) .. controls (5.4,1) and (6,.8) .. (6,1.5);
  \node at (4.8,.35) {$\dots$};
  \node at (5.4,1.15) {$\dots$};
  %%
  %% BLUE LINES
    \draw[blue, dotted] (-1.6,0) -- (6.4,0);
    \draw[blue, dotted] (-1.6,-1) -- (6.4,-1);
  %% OLD TOP LEVEL = NEW BOTTOM LEVEL
   \draw[thick] (1.2,-2) .. controls ++(0,1) and ++(0,-1) .. (3.6,-1);
   \draw[thick] (3.6,-2) .. controls ++(0,.5) and ++(0,-.5) .. (3,-1);
   \draw[thick] (2.4,-2) .. controls ++(0,.5) and ++(0,-.5) .. (1.8,-1);
   \draw[thick] (-1.2,-2) -- (-1.2,-1);
   \draw[thick] (-0,-2) -- (-0,-1);
   \draw[thick] (.6,-2) -- (.6,-1);
   \draw[thick] (1.8,-2) .. controls ++(0,.5) and ++(0,-.5) .. (1.2,-1);
   \draw[thick] (4.2,-2) -- (4.2,-1);
   \draw[thick] (5.4,-2) -- (5.4,-1);
   \draw[thick] (6,-2) -- (6,-1);
  %% BOTTOM LEVEL
   \draw[thick]  (3.6,0) .. controls ++(0,-1) and ++(0,1) .. (.6,-1);
   \draw[thick] (-1.2,0) -- (-1.2,-1);
   \draw[thick]  (0,0) -- (0,-1);
   \draw[thick]  (.6,0) .. controls ++(0,-.5) and ++(0,+.5) .. (1.2,-1);
   \draw[thick]  (1.2,0) .. controls ++(0,-.5) and ++(0,+.5) .. (1.8,-1);
   \draw[thick]  (2.4,0) .. controls ++(0,-.5) and ++(0,+.5) .. (3,-1);
   \draw[thick]  (3,0) .. controls ++(0,-.5) and ++(0,+.5) .. (3.6,-1);
   \draw[thick] (4.2,0) -- (4.2,-1);
   \draw[thick] (5.4,0) -- (5.4,-1);
   \draw[thick] (6,0) -- (6,-1);
   %%
%   \node at (-1.2,2.7) {$\scs \mathbf{1}$};
%   \node at (-.6,2.7) {$\scs \mathbf{2}$};
%   \node at (.6,2.7) {$\scs \mathbf{i}$};
%   \node at (1.25,2.7) {$\scs \mathbf{i+1}$};
%   \node at (6,2.7) {$\scs \mathbf{n+m}$};
%   %%
%   \node at (-1.2,-1.2) {$\scs \mathbf{1}$};
%   \node at (-0.05,-1.2) {$\scs \mathbf{i-1}$};
%   \node at (.6,-1.2) {$\scs \mathbf{i}$};
%   \node at (1.25,-1.2) {$\scs \mathbf{i+1}$};
%   \node at (3,-1.2) {$\scs \mathbf{n}$};
%   \node at (6,-1.2) {$\scs \mathbf{n+m}$};
   %%
   \filldraw  (-1.2,-1.5) circle (2pt);
   \node at (-1.45,-1.55) {$\scs a$};
\end{tikzpicture}}
\right]
\]
where we have moved the entire top third of the diagram to the bottom using the trace relation.  It is now easy to see that using both relations in \eqref{eq:rel1} together with the trace relatios simplifies the diagram to $h_{n+m} \otimes x_1^a$.
\end{proof}

\subsection{A Heisenberg-Virasoro subalgebra}
\label{sec:HV}
We exhibit generators of $\Tr(\H)$ which satisfy the Heisenberg relations for the original generators $h_n$.
There is another set of generators which satisfy Virasoro relations with central charge $1$.
The Virasoro algebra naturally acts on the Heisenberg algebra and the resulting semi-direct product algebra
is known as a Heisenberg-Virasoro algebra.
These generators come from images of diagrams which contain at most one right-hand curl.

We recover the original Heisenberg generators by the following result.

\begin{lemma}
\label{origheisrelations}
The elements $ h_n \otimes 1$ for $ n \in \Z - \{0 \} $ satisfy the following Heisenberg relations
\begin{equation*}
[h_m \otimes 1, h_n \otimes 1]=n \delta_{m,-n}.
\end{equation*}
\end{lemma}

\begin{proof}
For $mn>0$ the underlying morphisms for $ (h_m \otimes 1)(h_n \otimes 1) $ and $ (h_n \otimes 1)(h_m \otimes 1) $ in the category $ \H$ are in the same conjugacy class in $S_{n+m} $. Thus their traces are equal.

Suppose then that $m<0$ and $n>0$.  Then consider the diagram for $(h_n\otimes1)(h_m\otimes 1)$.  Using the trace condition with $n$ applications of the second equation in \eqref{eq:rel2} it is easy to see that
\[
(h_n\otimes1)(h_m\otimes 1) =
\left[ \;
\hackcenter{
\begin{tikzpicture}[scale=0.8]
 %% Separate lines by 0.6
 %% UPWARD ORIENTED LINES
  \draw[thick] (3,0) .. controls ++(0,1.25) and ++(0,-1.75) .. (0,1.5);
  \draw[thick] (0,0) .. controls ++(0,1) and ++(0,-.7) .. (.6,1.5);
  \draw[thick] (.6,0) .. controls ++(0,1) and ++(0,-.7) .. (1.2,1.5);
  \draw[thick] (1.8,0) .. controls ++(0,1) and ++(0,-.7) .. (2.4,1.5);
  \draw[thick] (2.4,0) .. controls ++(0,1) and ++(0,-.7) .. (3,1.5);
  \node at (1.2,.35) {$\dots$};
  \node at (1.8,1.15) {$\dots$};
  %%
  %% DOWNWARD ORIENTED LINES
  \draw[thick] (6,0) .. controls ++(0,1.25) and ++(0,-1.75) .. (3.6,1.5) ;
  \draw[thick] (3.6,0) .. controls ++(0,1) and ++(0,-.7) .. (4.2,1.5);
  \draw[thick] (4.2,0) .. controls ++(0,1) and ++(0,-.7) .. (4.8,1.5);
  \draw[thick] (5.4,0) .. controls ++(0,1) and ++(0,-.7) .. (6,1.5);
  \node at (4.8,-.65) {$\dots$};
  \node at (5.4,1.85) {$\dots$};
  %%
  %% BLUE LINES
    \draw[blue, dotted] (-0.4,0) -- (6.4,0);
    \draw[blue, dotted] (-0.4,1.5) -- (6.4,1.5);
  %% TOP LEVEL
   \draw[thick,->]  (3,1.5) .. controls ++(0,0.5) and ++(0,-.5) .. (3.6,2.5);
   \draw[thick] (3.6,1.5) .. controls ++(0,1) and ++(0,-1) .. (0,2.5);
   \draw[thick,->]  (2.4,1.5) .. controls ++(0,0.5) and ++(0,-.5) .. (3,2.5);
   \draw[thick,->]  (1.2,1.5) .. controls ++(0,0.5) and ++(0,-.5) .. (1.8,2.5);
   \draw[thick,->]  (.6,1.5) .. controls ++(0,0.5) and ++(0,-.5) .. (1.2,2.5);
   \draw[thick,->]  (0,1.5) .. controls ++(0,0.5) and ++(0,-.5) .. (.6,2.5);
   \draw[thick] (4.2,1.5) -- (4.2,2.5);
   \draw[thick] (4.8,1.5) -- (4.8,2.5);
   \draw[thick] (6,1.5) -- (6,2.5);
  %% BOTTOM LEVEL
   \draw[thick]  (2.4,0) .. controls ++(0,-.5) and ++(0,.5) .. (3,-1);
   \draw[thick]  (1.8,0) .. controls ++(0,-.5) and ++(0,.5) .. (2.4,-1);
   \draw[thick]  (.6,0) .. controls ++(0,-.5) and ++(0,.5) .. (1.2,-1);
   \draw[thick]  (0,0) .. controls ++(0,-.5) and ++(0,.5) .. (.6,-1);
   \draw[thick]  (3,0) .. controls ++(0,-.5) and ++(0,.5) .. (3.6,-1);
  \draw[thick ,->] (3.6,0) .. controls ++(0,-1) and ++(0,+1) .. (0,-1);
   \draw[thick,->] (4.2,0) -- (4.2,-1);
   \draw[thick,->] (5.4,0) -- (5.4,-1);
   \draw[thick,->] (6,0) -- (6,-1);
\end{tikzpicture}}
\right]
\]
which is equal to $(h_m\otimes1)(h_n\otimes 1) -n\delta_{n,m}$ by Lemma~\ref{lem:alem2}.
\end{proof}

The next Lemma will roughly lead to ``half" of a Virasoro algebra.

\begin{lemma}
\label{lemmam1n1}
For $m,n \in \mathbb{Z} $ with $mn > 0$,
\begin{equation*}
[h_m \otimes x_1, h_n \otimes x_1]=(n-m)(h_{n+m} \otimes x_1).
\end{equation*}
\end{lemma}

\begin{proof}
Let $ \beta_n=h_n \otimes x_1 $.
We prove by induction on $m$ that
$ [\beta_m, \beta_n]=(n-m)\beta_{m+n}$.
The case $m=1$ has the following graphical proof.

Without loss of generality assume $n>0$ so that
\[
(h_n \otimes x_1)(h_1 \otimes x_1) =
\left[ \;\;\;
\hackcenter{
\begin{tikzpicture}[scale=0.8]
  \draw[thick,->] (2.4,0) .. controls (2.4,1.25) and (0,.25) .. (0,2)
     node[pos=0.85, shape=coordinate](X){};
  \draw[thick,->] (0,0) .. controls (0,1) and (0.6,.8) .. (0.6,2);
  \draw[thick,->] (0.6,0) .. controls (0.6,1) and (1.2,.8) .. (1.2,2);
  \draw[thick,->] (1.8,0) .. controls (1.8,1) and (2.4,.8) .. (2.4,2);
  \node at (1.2,.35) {$\dots$};
  \node at (1.8,1.65) {$\dots$};
  \draw[thick,->] (3,0) -- (3,2);
  \filldraw  (3,1.2) circle (2pt);
  \filldraw  (.05,1.6) circle (2pt);
\end{tikzpicture}}
\; \;\;\right] .
\]
Then conjugating by $T_n$ we have
\[
\left[ \;\;\;
\hackcenter{
\begin{tikzpicture}[scale=0.8]
  \draw[thick] (2.4,0) .. controls (2.4,1.25) and (0,.25) .. (0,1.5)
     node[pos=0.85, shape=coordinate](X){};
  \draw[thick] (0,0) .. controls (0,1) and (0.6,.8) .. (0.6,1.5);
  \draw[thick] (0.6,0) .. controls (0.6,1) and (1.2,.8) .. (1.2,1.5);
  \draw[thick] (1.8,0) .. controls (1.8,1) and (2.4,.8) .. (2.4,1.5);
  \node at (1.2,.35) {$\dots$};
  \node at (1.8,1.65) {$\dots$};
  \draw[thick] (3,0) -- (3,1.5);
  %% DOTS
  \filldraw  (3,1.2) circle (2pt);
  \filldraw  (.05,1.3) circle (2pt);
 %% BLUE LINES
    \draw[blue, dotted] (-0.4,0) -- (3.4,0);
    \draw[blue, dotted] (-0.4,1.5) -- (3.4,1.5);
    %% TOP
   \draw[thick,->]  (2.4,1.5) .. controls (2.4,2) and (3,2) .. (3,2.5);
   \draw[thick,->] (3,1.5) .. controls (3,2) and (2.4,2) .. (2.4,2.5);
   \draw[thick,->] (0,1.5) -- (0,2.5);
   \draw[thick,->] (.6,1.5) -- (.6,2.5);
   \draw[thick,->] (1.2,1.5) -- (1.2,2.5);
   %% Bottom
   \draw[thick]  (2.4,0) .. controls (2.4,-.5) and (3,-.5) .. (3,-1);
   \draw[thick ] (3,0) .. controls (3,-.5) and (2.4,-.5) .. (2.4,-1);
   \draw[thick] (0,0) -- (0,-1);
   \draw[thick] (.6,0) -- (.6,-1);
   \draw[thick] (1.8,0) -- (1.8,-1);
\end{tikzpicture}}
\; \;\;\right]
\;\; \refequal{\eqref{eq:nil-dot}} \;\;
\left[ \;\;\;
\hackcenter{
\begin{tikzpicture}[scale=0.8]
  \draw[thick] (2.4,0) .. controls (2.4,1.25) and (0,.25) .. (0,1.5)
     node[pos=0.85, shape=coordinate](X){};
  \draw[thick] (0,0) .. controls (0,1) and (0.6,.8) .. (0.6,1.5);
  \draw[thick] (0.6,0) .. controls (0.6,1) and (1.2,.8) .. (1.2,1.5);
  \draw[thick] (1.8,0) .. controls (1.8,1) and (2.4,.8) .. (2.4,1.5);
  \node at (1.2,.35) {$\dots$};
  \node at (1.8,1.65) {$\dots$};
  \draw[thick] (3,0) -- (3,1.5);
  %% DOTS
  \filldraw  (2.45,2.2) circle (2pt);
  \filldraw  (.03,1.3) circle (2pt);
 %% BLUE LINES
    \draw[blue, dotted] (-0.4,0) -- (3.4,0);
    \draw[blue, dotted] (-0.4,1.5) -- (3.4,1.5);
    %% TOP
   \draw[thick,->]  (2.4,1.5) .. controls (2.4,2) and (3,2) .. (3,2.5);
   \draw[thick,->] (3,1.5) .. controls (3,2) and (2.4,2) .. (2.4,2.5);
   \draw[thick,->] (0,1.5) -- (0,2.5);
   \draw[thick,->] (.6,1.5) -- (.6,2.5);
   \draw[thick,->] (1.2,1.5) -- (1.2,2.5);
   %% Bottom
   \draw[thick]  (2.4,0) .. controls (2.4,-.5) and (3,-.5) .. (3,-1);
   \draw[thick ] (3,0) .. controls (3,-.5) and (2.4,-.5) .. (2.4,-1);
   \draw[thick] (0,0) -- (0,-1);
   \draw[thick] (.6,0) -- (.6,-1);
   \draw[thick] (1.8,0) -- (1.8,-1);
\end{tikzpicture}}
\; \;\;\right]
\;\; - \;\;
\left[ \;\;\;
\hackcenter{
\begin{tikzpicture}[scale=0.8]
  \draw[thick] (2.4,0) .. controls (2.4,1.25) and (0,.25) .. (0,1.5)
     node[pos=0.85, shape=coordinate](X){};
  \draw[thick] (0,0) .. controls (0,1) and (0.6,.8) .. (0.6,1.5);
  \draw[thick] (0.6,0) .. controls (0.6,1) and (1.2,.8) .. (1.2,1.5);
  \draw[thick] (1.8,0) .. controls (1.8,1) and (2.4,.8) .. (2.4,1.5);
  \node at (1.2,.35) {$\dots$};
  \node at (1.8,1.65) {$\dots$};
  \draw[thick] (3,0) -- (3,1.5);
  %% DOTS
  \filldraw  (.03,1.3) circle (2pt);
 %% BLUE LINES
    \draw[blue, dotted] (-0.4,0) -- (3.4,0);
    \draw[blue, dotted] (-0.4,1.5) -- (3.4,1.5);
    %% TOP
   \draw[thick,->]  (2.4,1.5) -- (2.4,2.5);
   \draw[thick,->] (3,1.5) -- (3,2.5);
   \draw[thick,->] (0,1.5) -- (0,2.5);
   \draw[thick,->] (.6,1.5) -- (.6,2.5);
   \draw[thick,->] (1.2,1.5) -- (1.2,2.5);
   %% Bottom
   \draw[thick]  (2.4,0) .. controls (2.4,-.5) and (3,-.5) .. (3,-1);
   \draw[thick ] (3,0) .. controls (3,-.5) and (2.4,-.5) .. (2.4,-1);
   \draw[thick] (0,0) -- (0,-1);
   \draw[thick] (.6,0) -- (.6,-1);
   \draw[thick] (1.8,0) -- (1.8,-1);
\end{tikzpicture}}
\; \;\;\right]
\]
The second term on the right-hand side is just $-(h_{n+1} \otimes x_1)$.  The first term can be simplified further by conjugating by $T_{n-1}$, sliding the dot using \eqref{eq:nil-dot}, and observing using Lemma~\ref{lem_nplusm} that the crossing resolution term is also equal to $-(h_{n+1} \otimes x_1)$.  Continuing this procedure, conjugating by $T_{n-2}$, sliding the dot, and simplifying the crossing resolution term using Lemma~\ref{lem_nplusm}, we see that $(h_n \otimes x_1)(h_1 \otimes x_1)$ is equal to
\[
\left[ \;\;\;
\hackcenter{
\begin{tikzpicture}[scale=0.8]
  \draw[thick] (2.4,0) .. controls (2.4,1.25) and (0,.25) .. (0,1.5)
     node[pos=0.85, shape=coordinate](X){};
  \draw[thick] (0,0) .. controls (0,1) and (0.6,.8) .. (0.6,1.5);
  \draw[thick] (0.6,0) .. controls (0.6,1) and (1.2,.8) .. (1.2,1.5);
  \draw[thick] (1.8,0) .. controls (1.8,1) and (2.4,.8) .. (2.4,1.5);
  \node at (1.2,.35) {$\dots$};
  \node at (1.8,1.65) {$\dots$};
  \draw[thick] (3,0) -- (3,1.5);
  %% DOTS
  \filldraw  (0.05,2.25) circle (2pt);
  \filldraw  (0.03,1.3) circle (2pt);
 %% BLUE LINES
    \draw[blue, dotted] (-0.4,0) -- (3.4,0);
    \draw[blue, dotted] (-0.4,1.5) -- (3.4,1.5);
    %% TOP
   \draw[thick,->]  (3,1.5) .. controls ++(0,0.95) and ++(0,-1.15) .. (0,2.5);
   \draw[thick,->]  (2.4,1.5) .. controls ++(0,0.5) and ++(0,-0.5) .. (3,2.5);
   \draw[thick,->]  (0,1.5) .. controls ++(0,0.5) and ++(0,-0.5) .. (0.6,2.5);
   \draw[thick,->]  (0.6,1.5) .. controls ++(0,0.5) and ++(0,-0.5) .. (1.2,2.5);
   \draw[thick,->]  (1.2,1.5) .. controls ++(0,0.5) and ++(0,-0.5) .. (1.8,2.5);
   %% Bottom
   \draw[thick]  (3,0) .. controls ++(0,-0.95) and ++(0,+0.95) .. (0,-1);
   \draw[thick]  (0,0) .. controls ++(0,-0.5) and ++(0,+0.5) .. (0.6,-1);
   \draw[thick]  (0.6,0) .. controls ++(0,-0.5) and ++(0,+0.5) .. (1.2,-1);
    \draw[thick] (1.8,0) .. controls ++(0,-0.5) and ++(0,+0.5) .. (2.4,-1);
   \draw[thick]  (2.4,0) .. controls (2.4,-.5) and (3,-.5) .. (3,-1);
\end{tikzpicture}}
\; \;\;\right]
\;\; - \;\; n(h_{n+1} \otimes x_1).
\]
The first term can be simplified by sliding the lower dot
\[
\left[ \;\;\;
\hackcenter{
\begin{tikzpicture}[scale=0.8]
  \draw[thick] (2.4,0) .. controls (2.4,1.25) and (0,.25) .. (0,1.5)
     node[pos=0.85, shape=coordinate](X){};
  \draw[thick] (0,0) .. controls (0,1) and (0.6,.8) .. (0.6,1.5);
  \draw[thick] (0.6,0) .. controls (0.6,1) and (1.2,.8) .. (1.2,1.5);
  \draw[thick] (1.8,0) .. controls (1.8,1) and (2.4,.8) .. (2.4,1.5);
  \node at (1.2,.35) {$\dots$};
  \node at (1.8,1.65) {$\dots$};
  \draw[thick] (3,0) -- (3,1.5);
  %% DOTS
  \filldraw  (0.05,2.25) circle (2pt);
  \filldraw  (0.03,1.3) circle (2pt);
 %% BLUE LINES
    \draw[blue, dotted] (-0.4,0) -- (3.4,0);
    \draw[blue, dotted] (-0.4,1.5) -- (3.4,1.5);
    %% TOP
   \draw[thick,->]  (3,1.5) .. controls ++(0,0.95) and ++(0,-0.95) .. (0,2.5);
   \draw[thick,->]  (2.4,1.5) .. controls ++(0,0.5) and ++(0,-0.5) .. (3,2.5);
   \draw[thick,->]  (0,1.5) .. controls ++(0,0.5) and ++(0,-0.5) .. (0.6,2.5);
   \draw[thick,->]  (0.6,1.5) .. controls ++(0,0.5) and ++(0,-0.5) .. (1.2,2.5);
   \draw[thick,->]  (1.2,1.5) .. controls ++(0,0.5) and ++(0,-0.5) .. (1.8,2.5);
   %% Bottom
   \draw[thick]  (3,0) .. controls ++(0,-0.95) and ++(0,+1.15) .. (0,-1);
   \draw[thick]  (0,0) .. controls ++(0,-0.5) and ++(0,+0.5) .. (0.6,-1);
   \draw[thick]  (0.6,0) .. controls ++(0,-0.5) and ++(0,+0.5) .. (1.2,-1);
    \draw[thick] (1.8,0) .. controls ++(0,-0.5) and ++(0,+0.5) .. (2.4,-1);
   \draw[thick]  (2.4,0) .. controls (2.4,-.5) and (3,-.5) .. (3,-1);
\end{tikzpicture}}
\; \;\;\right]
\;\; = \;\;
\left[ \;\;\;
\hackcenter{
\begin{tikzpicture}[scale=0.8]
  \draw[thick] (2.4,0) .. controls (2.4,1.25) and (0,.25) .. (0,1.5)
     node[pos=0.85, shape=coordinate](X){};
  \draw[thick] (0,0) .. controls (0,1) and (0.6,.8) .. (0.6,1.5);
  \draw[thick] (0.6,0) .. controls (0.6,1) and (1.2,.8) .. (1.2,1.5);
  \draw[thick] (1.8,0) .. controls (1.8,1) and (2.4,.8) .. (2.4,1.5);
  \node at (1.2,.35) {$\dots$};
  \node at (1.8,1.65) {$\dots$};
  \draw[thick] (3,0) -- (3,1.5);
  %% DOTS
  \filldraw  (0.05,2.25) circle (2pt);
  \filldraw  (0.55,2.25) circle (2pt);
 %% BLUE LINES
    \draw[blue, dotted] (-0.4,0) -- (3.4,0);
    \draw[blue, dotted] (-0.4,1.5) -- (3.4,1.5);
    %% TOP
   \draw[thick,->]  (3,1.5) .. controls ++(0,0.95) and ++(0,-1.15) .. (0,2.5);
   \draw[thick,->]  (2.4,1.5) .. controls ++(0,0.5) and ++(0,-0.5) .. (3,2.5);
   \draw[thick,->]  (0,1.5) .. controls ++(0,0.5) and ++(0,-0.5) .. (0.6,2.5);
   \draw[thick,->]  (0.6,1.5) .. controls ++(0,0.5) and ++(0,-0.5) .. (1.2,2.5);
   \draw[thick,->]  (1.2,1.5) .. controls ++(0,0.5) and ++(0,-0.5) .. (1.8,2.5);
   %% Bottom
   \draw[thick]  (3,0) .. controls ++(0,-0.95) and ++(0,+0.95) .. (0,-1);
   \draw[thick]  (0,0) .. controls ++(0,-0.5) and ++(0,+0.5) .. (0.6,-1);
   \draw[thick]  (0.6,0) .. controls ++(0,-0.5) and ++(0,+0.5) .. (1.2,-1);
    \draw[thick] (1.8,0) .. controls ++(0,-0.5) and ++(0,+0.5) .. (2.4,-1);
   \draw[thick]  (2.4,0) .. controls (2.4,-.5) and (3,-.5) .. (3,-1);
\end{tikzpicture}}
\; \;\;\right]
\;\; + \;\;
\left[ \;\;\;
\hackcenter{
\begin{tikzpicture}[scale=0.8]
  \draw[thick] (2.4,0) .. controls (2.4,1.25) and (0,.25) .. (0,1.5)
     node[pos=0.85, shape=coordinate](X){};
  \draw[thick] (0,0) .. controls (0,1) and (0.6,.8) .. (0.6,1.5);
  \draw[thick] (0.6,0) .. controls (0.6,1) and (1.2,.8) .. (1.2,1.5);
  \draw[thick] (1.8,0) .. controls (1.8,1) and (2.4,.8) .. (2.4,1.5);
  \node at (1.2,.35) {$\dots$};
  \node at (1.8,1.65) {$\dots$};
  \draw[thick] (3,0) -- (3,1.5);
  %% DOTS
  \filldraw  (0.02,2.1) circle (2pt);
 %% BLUE LINES
    \draw[blue, dotted] (-0.4,0) -- (3.4,0);
    \draw[blue, dotted] (-0.4,1.5) -- (3.4,1.5);
    %% TOP
   \draw[thick,->]  (3,1.5) .. controls ++(0,0.95) and ++(0,-0.95) .. (0.6,2.5);
   \draw[thick,->]  (2.4,1.5) .. controls ++(0,0.5) and ++(0,-0.5) .. (3,2.5);
   \draw[thick,->]  (0,1.5) -- (0,2.5);
   \draw[thick,->]  (0.6,1.5) .. controls ++(0,0.5) and ++(0,-0.5) .. (1.2,2.5);
   \draw[thick,->]  (1.2,1.5) .. controls ++(0,0.5) and ++(0,-0.5) .. (1.8,2.5);
   %% Bottom
   \draw[thick]  (3,0) .. controls ++(0,-0.95) and ++(0,+0.95) .. (0,-1);
   \draw[thick]  (0,0) .. controls ++(0,-0.5) and ++(0,+0.5) .. (0.6,-1);
   \draw[thick]  (0.6,0) .. controls ++(0,-0.5) and ++(0,+0.5) .. (1.2,-1);
    \draw[thick] (1.8,0) .. controls ++(0,-0.5) and ++(0,+0.5) .. (2.4,-1);
   \draw[thick]  (2.4,0) .. controls (2.4,-.5) and (3,-.5) .. (3,-1);
\end{tikzpicture}}
\; \;\;\right]
\]
\[
 = \;\;
 \left[ \;\;\;
\hackcenter{
\begin{tikzpicture}[scale=0.8]
  \draw[thick,->] (2.4,0) .. controls (2.4,1.25) and (0,.25) .. (0,2)
     node[pos=0.85, shape=coordinate](X){};
  \draw[thick,->] (0,0) .. controls (0,1) and (0.6,.8) .. (0.6,2);
  \draw[thick,->] (0.6,0) .. controls (0.6,1) and (1.2,.8) .. (1.2,2);
  \draw[thick,->] (1.8,0) .. controls (1.8,1) and (2.4,.8) .. (2.4,2);
  \node at (1.2,.35) {$\dots$};
  \node at (1.8,1.65) {$\dots$};
  \draw[thick,->] (-0.6,0) -- (-0.6,2);
  \filldraw  (-0.6,1.2) circle (2pt);
  \filldraw  (.05,1.6) circle (2pt);
\end{tikzpicture}}
\; \;\;\right]
\;\; + \;\;
\left[ \;\;\;
\hackcenter{
\begin{tikzpicture}[scale=0.8]
  \draw[thick] (2.4,0) .. controls (2.4,1.25) and (0,.25) .. (0,1.5)
     node[pos=0.85, shape=coordinate](X){};
  \draw[thick] (0,0) .. controls (0,1) and (0.6,.8) .. (0.6,1.5);
  \draw[thick] (0.6,0) .. controls (0.6,1) and (1.2,.8) .. (1.2,1.5);
  \draw[thick] (1.8,0) .. controls (1.8,1) and (2.4,.8) .. (2.4,1.5);
  \node at (1.2,.35) {$\dots$};
  \node at (1.8,1.65) {$\dots$};
  \draw[thick] (3,0) -- (3,1.5);
  %% DOTS
  \filldraw  (0.02,2.1) circle (2pt);
 %% BLUE LINES
    \draw[blue, dotted] (-0.4,0) -- (3.4,0);
    \draw[blue, dotted] (-0.4,1.5) -- (3.4,1.5);
    %% TOP
   \draw[thick,->]  (0,1.5) -- (0,2.5);
   \draw[thick,->]  (1.2,1.5) .. controls ++(0,0.5) and ++(0,-0.5) .. (0.6,2.5);
   \draw[thick,->]  (2.4,1.5) .. controls ++(0,0.5) and ++(0,-0.5) .. (1.8,2.5);
   \draw[thick,->]  (3,1.5) .. controls ++(0,0.5) and ++(0,-0.5) .. (2.4,2.5);
   \draw[thick,->]  (0.6,1.5) .. controls ++(0,0.95) and ++(0,-0.95) .. (3,2.5);
   %% Bottom
   \draw[thick]  (3,0) .. controls ++(0,-0.95) and ++(0,+0.95) .. (0,-1);
   \draw[thick]  (0,0) .. controls ++(0,-0.5) and ++(0,+0.5) .. (0.6,-1);
   \draw[thick]  (0.6,0) .. controls ++(0,-0.5) and ++(0,+0.5) .. (1.2,-1);
    \draw[thick] (1.8,0) .. controls ++(0,-0.5) and ++(0,+0.5) .. (2.4,-1);
   \draw[thick]  (2.4,0) .. controls (2.4,-.5) and (3,-.5) .. (3,-1);
\end{tikzpicture}}
\; \;\;\right]
\]
where the trace relation was used repeatedly on the second term to reorder the crossings at the top of the diagram.  The second term simplifies further using the trace relation to reduce the number of crossings so that the second diagram is equal to $(h_{n+1} \otimes x_1)$, completing the induction step.

Now consider the Jacobi identity:
\begin{equation*}
[[\beta_1, \beta_{m-1}],\beta_n]+[[\beta_{m-1},\beta_n],\beta_1]+[[\beta_n,\beta_1],\beta_{m-1}]=0.
\end{equation*}
Using the base case and the induction hypothesis this becomes
\begin{equation*}
(m-2)[\beta_m,\beta_n]+(n-m+1)[\beta_{m+n-1},\beta_1]+(1-n)[\beta_{n+1},\beta_{m-1}]=0.
\end{equation*}
Again by using the base case and the induction hypothesis this becomes
\begin{equation*}
(m-2)[\beta_m,\beta_n]+(n-m+1)(2-m-n)\beta_{m+n}+(1-n)(m-n-2)\beta_{m+n}=0.
\end{equation*}
It now follows easily that
$ [\beta_m,\beta_n]=(n-m)\beta_{m+n}$.
\end{proof}

We now have some relations leading to an action of a Virasoro algebra on a Heisenberg algebra.

\begin{lemma}
\label{lemman1m0}
Let $m,n \in \mathbb{Z} $ with $ mn>0$.  Then
\label{posvirposheis}
\begin{equation*}
[h_{n} \otimes x_1, h_m \otimes 1]=m(h_{m+n} \otimes 1).
\end{equation*}
\end{lemma}

\begin{proof}
We only consider the case that $ m $ and $n$ are positive.  The other case is similar.
Using the same method of proof as in Lemma~\ref{lemmam1n1} it is not difficult to show
\[
(h_m \otimes 1 )(h_n \otimes x_1 )
\;\; = \;\; \left[ \;\;\;
\hackcenter{
\begin{tikzpicture}[scale=0.8]
 %% Separate lines by 0.6
 %% DOWNWARD ORIENTED LINES
  \draw[thick,->] (3,0) .. controls ++(0,1.25) and ++(0,-1.75) .. (0,2);
  \draw[thick,->] (0,0) .. controls ++(0,1) and ++(0,-1.2) .. (.6,2);
  \draw[thick,->] (.6,0) .. controls ++(0,1) and ++(0,-1.2) .. (1.2,2);
  \draw[thick,->] (1.8,0) .. controls ++(0,1) and ++(0,-1.2) .. (2.4,2);
  \draw[thick,->] (2.4,0) .. controls ++(0,1) and ++(0,-1.2) .. (3,2);
  \node at (1.2,.35) {$\dots$};
  \node at (1.8,1.65) {$\dots$};
  %%
  %% UPWARD ORIENTED LINES
  \draw[thick,->] (6,0) .. controls ++(0,1.25) and ++(0,-1.75) .. (3.6,2);
  \draw[thick,->] (3.6,0) .. controls ++(0,1) and ++(0,-1.2) .. (4.2,2);
  \draw[thick,->] (4.2,0) .. controls ++(0,1) and ++(0,-1.2) .. (4.8,2);
  \draw[thick,->] (5.4,0) .. controls ++(0,1) and ++(0,-1.2) .. (6,2);
   \filldraw  (3.7,1.4) circle (2pt);
  \node at (4.8,.35) {$\dots$};
  \node at (5.4,1.65) {$\dots$};
\end{tikzpicture}}
\; \;\;\right]
 \;\; = \;\; \hspace{1in}
\]

\[
\left[ \;\;\;
\hackcenter{
\begin{tikzpicture}[scale=0.8]
 %% Separate lines by 0.6
 %% DOWNWARD ORIENTED LINES
  \draw[thick] (3,0) .. controls (3,1.25) and (0,.25) .. (0,1.5)
     node[pos=0.85, shape=coordinate](X){};
  \draw[thick] (0,0) .. controls (0,1) and (.6,.8) .. (.6,1.5);
  \draw[thick] (.6,0) .. controls (.6,1) and (1.2,.8) .. (1.2,1.5);
  \draw[thick] (1.8,0) .. controls (1.8,1) and (2.4,.8) .. (2.4,1.5);
  \draw[thick] (2.4,0) .. controls (2.4,1) and (3,.8) .. (3,1.5);
  \node at (1.2,.35) {$\dots$};
  \node at (1.8,1.15) {$\dots$};
  %%
  %% UPWARD ORIENTED LINES
  \draw[thick] (6,0) .. controls (6,1.25) and (3.6,.25) .. (3.6,1.5)
     node[pos=0.85, shape=coordinate](X){};
  \draw[thick] (3.6,0) .. controls (3.6,1) and (4.2,.8) .. (4.2,1.5);
  \draw[thick] (4.2,0) .. controls (4.2,1) and (4.8,.8) .. (4.8,1.5);
  \draw[thick] (5.4,0) .. controls (5.4,1) and (6,.8) .. (6,1.5);
  \node at (4.8,-.65) {$\dots$};
  \node at (5.4,1.85) {$\dots$};
  %%
  %% BLUE LINES
    \draw[blue, dotted] (-0.4,0) -- (6.4,0);
    \draw[blue, dotted] (-0.4,1.5) -- (6.4,1.5);
  %% TOP LEVEL
   \draw[thick,->]  (3,1.5) .. controls (3,2) and (3.6,2) .. (3.6,2.5);
   \draw[thick,->] (3.6,1.5) .. controls (3.6,2.5) and (0,1.5) .. (0,2.5);
   \draw[thick,->]  (2.4,1.5) .. controls (2.4,2) and (3,2) .. (3,2.5);
   \draw[thick,->]  (1.2,1.5) .. controls (1.2,2) and (1.8,2) .. (1.8,2.5);
   \draw[thick,->]  (.6,1.5) .. controls (.6,2) and (1.2,2) .. (1.2,2.5);
   \draw[thick,->]  (0,1.5) .. controls (0,2) and (.6,2) .. (.6,2.5);
   \draw[thick,->] (4.2,1.5) -- (4.2,2.5);
   \draw[thick,->] (4.8,1.5) -- (4.8,2.5);
   \draw[thick,->] (6,1.5) -- (6,2.5);
     \filldraw  (.15,2.2) circle (2pt);
  %% BOTTOM LEVEL
   \draw[thick]  (2.4,0) .. controls (2.4,-.5) and (3,-.5) .. (3,-1);
   \draw[thick]  (1.8,0) .. controls (1.8,-.5) and (2.4,-.5) .. (2.4,-1);
   \draw[thick]  (.6,0) .. controls (.6,-.5) and (1.2,-.5) .. (1.2,-1);
   \draw[thick]  (0,0) .. controls (0,-.5) and (.6,-.5) .. (.6,-1);
   \draw[thick]  (3,0) .. controls (3,-.5) and (3.6,-.5) .. (3.6,-1);
  \draw[thick ] (3.6,0) .. controls (3.6,-1) and (0,0) .. (0,-1);
   \draw[thick] (4.2,0) -- (4.2,-1);
   \draw[thick] (5.4,0) -- (5.4,-1);
   \draw[thick] (6,0) -- (6,-1);
\end{tikzpicture}}
\right]
\;\; - \;\;
m h_{n+m} \otimes 1.
\]
Then by sliding the crossings appearing in $h_n$ to the left of those appearing in $h_m$, the first term can be written as
\[
\left[ \;\;\;
\hackcenter{
\begin{tikzpicture}[scale=0.8]
 %% Separate lines by 0.6
 %% DOWNWARD ORIENTED LINES
  \draw[thick] (3,0) .. controls (3,1.25) and (0,.25) .. (0,1.5)
     node[pos=0.85, shape=coordinate](X){};
  \draw[thick] (0,0) .. controls (0,1) and (.6,.8) .. (.6,1.5);
  \draw[thick] (.6,0) .. controls (.6,1) and (1.2,.8) .. (1.2,1.5);
  \draw[thick] (1.8,0) .. controls (1.8,1) and (2.4,.8) .. (2.4,1.5);
  \draw[thick] (2.4,0) .. controls (2.4,1) and (3,.8) .. (3,1.5);
  \node at (1.2,.35) {$\dots$};
  \node at (1.8,1.15) {$\dots$};
  %%
  %% UPWARD ORIENTED LINES
  \draw[thick] (6,0) .. controls (6,1.25) and (3.6,.25) .. (3.6,1.5)
     node[pos=0.85, shape=coordinate](X){};
  \draw[thick] (3.6,0) .. controls (3.6,1) and (4.2,.8) .. (4.2,1.5);
  \draw[thick] (4.2,0) .. controls (4.2,1) and (4.8,.8) .. (4.8,1.5);
  \draw[thick] (5.4,0) .. controls (5.4,1) and (6,.8) .. (6,1.5);
  \node at (4.8,-.65) {$\dots$};
  \node at (5.4,1.85) {$\dots$};
  %%
  %% BLUE LINES
    \draw[blue, dotted] (-0.4,0) -- (6.4,0);
    \draw[blue, dotted] (-0.4,1.5) -- (6.4,1.5);
  %% TOP LEVEL
   \draw[thick,->]  (3,1.5) .. controls (3,2) and (3.6,2) .. (3.6,2.5);
   \draw[thick,->] (3.6,1.5) .. controls (3.6,2.5) and (0,1.5) .. (0,2.5);
   \draw[thick,->]  (2.4,1.5) .. controls (2.4,2) and (3,2) .. (3,2.5);
   \draw[thick,->]  (1.2,1.5) .. controls (1.2,2) and (1.8,2) .. (1.8,2.5);
   \draw[thick,->]  (.6,1.5) .. controls (.6,2) and (1.2,2) .. (1.2,2.5);
   \draw[thick,->]  (0,1.5) .. controls (0,2) and (.6,2) .. (.6,2.5);
   \draw[thick,->] (4.2,1.5) -- (4.2,2.5);
   \draw[thick,->] (4.8,1.5) -- (4.8,2.5);
   \draw[thick,->] (6,1.5) -- (6,2.5);
     \filldraw  (.15,2.2) circle (2pt);
  %% BOTTOM LEVEL
   \draw[thick]  (2.4,0) .. controls (2.4,-.5) and (3,-.5) .. (3,-1);
   \draw[thick]  (1.8,0) .. controls (1.8,-.5) and (2.4,-.5) .. (2.4,-1);
   \draw[thick]  (.6,0) .. controls (.6,-.5) and (1.2,-.5) .. (1.2,-1);
   \draw[thick]  (0,0) .. controls (0,-.5) and (.6,-.5) .. (.6,-1);
   \draw[thick]  (3,0) .. controls (3,-.5) and (3.6,-.5) .. (3.6,-1);
  \draw[thick ] (3.6,0) .. controls (3.6,-1) and (0,0) .. (0,-1);
   \draw[thick] (4.2,0) -- (4.2,-1);
   \draw[thick] (5.4,0) -- (5.4,-1);
   \draw[thick] (6,0) -- (6,-1);
\end{tikzpicture}}
\right]
\;\; = \;\;
\left[ \;\;\;
\hackcenter{
\begin{tikzpicture}[scale=0.8]
 %% Separate lines by 0.6
 %% DOWNWARD ORIENTED LINES
  \draw[thick] (3,0) .. controls (3,1.25) and (0,.25) .. (0,1.5)
     node[pos=0.85, shape=coordinate](X){};
  \draw[thick] (0,0) .. controls (0,1) and (.6,.8) .. (.6,1.5);
  \draw[thick] (.6,0) .. controls (.6,1) and (1.2,.8) .. (1.2,1.5);
  \draw[thick] (1.8,0) .. controls (1.8,1) and (2.4,.8) .. (2.4,1.5);
  \draw[thick] (2.4,0) .. controls (2.4,1) and (3,.8) .. (3,1.5);
  \node at (1.2,.35) {$\dots$};
  \node at (1.8,1.15) {$\dots$};
  %% UPWARD ORIENTED LINES
  \draw[thick] (-.6,0) .. controls (-.6,1.25) and (-3,.25) .. (-3,1.5)
     node[pos=0.85, shape=coordinate](X){};
  \draw[thick] (-3,0) .. controls (-3,1) and (-2.4,.8) .. (-2.4,1.5);
  \draw[thick] (-2.4,0) .. controls (-2.4,1) and (-1.8,.8) .. (-1.8,1.5);
  \draw[thick] (-1.2,0) .. controls (-1.2,1) and (-0.6,.8) .. (-0.6,1.5);
  \node at (-1.8,.35) {$\dots$};
  \node at (-1.2,1.15) {$\dots$};
  %% BLUE LINES
    \draw[blue, dotted] (-3.4,0) -- (3.4,0);
    \draw[blue, dotted] (-3.4,1.5) -- (3.4,1.5);
%  %% TOP LEVEL
   \draw[thick, ->] (-3,1.5) -- (-3,3);
   \draw[thick,->]  (-.6,1.5) .. controls ++(0,.65) and ++(0,-1) .. (3,3);
   \draw[thick,->]  (-1.8,1.5) .. controls ++(0,.85) and ++(0,-1) .. (1.8,3);
   \draw[thick,->]  (-2.4,1.5) .. controls ++(0,1.1) and ++(0,-1) .. (1.2,3);
   \draw[thick, ->]  (0,1.5) .. controls ++(0,.5) and ++(0,-.65) .. (-2.4,3);
   \draw[thick, ->]  (.6,1.5) .. controls ++(0,.5) and ++(0,-.65) .. (-1.8,3);
   \draw[thick, ->]  (1.2,1.5) .. controls ++(0,.5) and ++(0,-.65) .. (-1.2,3);
   \draw[thick, ->]  (2.4,1.5) .. controls ++(0,.5) and ++(0,-.65) .. (0,3);
   \draw[thick, ->]  (3,1.5) .. controls ++(0,.5) and ++(0,-.65) .. (0.6,3);
        \filldraw  (-3,2.2) circle (2pt);
  %% BOTTOM LEVEL
     \draw[thick] (-3,0) -- (-3,-1.5);
   \draw[thick]  (-.6,0) .. controls ++(0,-.65) and ++(0,1) .. (3,-1.5);
   \draw[thick]  (-1.2,0) .. controls ++(0,-.85) and ++(0,1) .. (2.4,-1.5);
   \draw[thick]  (-2.4,0) .. controls ++(0,-1.1) and ++(0,1) .. (1.2,-1.5);
      \draw[thick]  (0,0) .. controls ++(0,-.5) and ++(0,.65) .. (-2.4,-1.5);
      \draw[thick]  (.6,0) .. controls ++(0,-.5) and ++(0,.65) .. (-1.8,-1.5);
      \draw[thick]  (1.8,0) .. controls ++(0,-.5) and ++(0,.65) .. (-0.6,-1.5);
      \draw[thick]  (2.4,0) .. controls ++(0,-.5) and ++(0,.65) .. (0,-1.5);
      \draw[thick]  (3,0) .. controls ++(0,-.5) and ++(0,.65) .. (0.6,-1.5);
\end{tikzpicture}}
\right]
\]
which simplifies to $(h_n \otimes x_1)(h_m \otimes 1)$.

%
%Let $ \beta_n = h_n \otimes x_1 $ and $ \alpha_m = h_m \otimes 1$.
%We prove by induction on $n$ that $ [\beta_n, \alpha_m]=m \alpha_{n+m}$.
%The base case $[\beta_1, \alpha_m]$ is computed graphically as follows:
%
%\JS{Supply graphical proof for $ [h_1 \otimes x_1, h_m \otimes 1]$}.
%
%Consider the Jacobi identity:
%\begin{equation*}
%[[\beta_1,\beta_{n-1}],\alpha_m]+[[\beta_{n-1},\alpha_m],\beta_1]+[[\alpha_m,\beta_1],\beta_{n-1}]=0.
%\end{equation*}
%By Lemma ~\ref{lemmam1n1}, the induction hypothesis and the base case, this is equal to
%\begin{equation*}
%(n-2)[\beta_n,\alpha_m]+m[\alpha_{m+n-1},\beta_1]-m[\alpha_{m+1},\beta_{n-1}]=0.
%\end{equation*}
%Once again using the base and the induction hypothesis, this becomes
%\begin{equation*}
%(n-2)[\beta_n,\alpha_m]-m(m+n-1) \alpha_{m+n} + m(m+1) \alpha_{m+n}=0.
%\end{equation*}
%It now easy follows that
%$ [\beta_n,\alpha_m]=m \alpha_{m+n}$.
\end{proof}

\begin{lemma}
\label{-n1m1}
For integers $m,n \geq 1$ we have
\begin{equation*}
[h_{-m} \otimes x_1, h_n \otimes 1] =
\begin{cases}
n(h_{n-m} \otimes 1) & \text{ if } n > m \\
0 & \text{ if }  m \geq n  \geq 1.
\end{cases}
\end{equation*}
\end{lemma}

\begin{proof}
This is proven by direct computation.
  \[
(h_n \otimes 1 )(h_{-m} \otimes x_1 )
\;\; = \;\; \left[ \;
\hackcenter{
\begin{tikzpicture}[scale=0.8]
 %% Separate lines by 0.6
 %% UPWARD ORIENTED LINES
  \draw[thick,->] (3,0) .. controls (3,1.25) and (0,.25) .. (0,2);
  \draw[thick,->] (0,0) .. controls (0,1) and (.6,.8) .. (.6,2);
  \draw[thick,->] (.6,0) .. controls (.6,1) and (1.2,.8) .. (1.2,2);
  \draw[thick,->] (1.8,0) .. controls (1.8,1) and (2.4,.8) .. (2.4,2);
  \draw[thick,->] (2.4,0) .. controls (2.4,1) and (3,.8) .. (3,2);
  \node at (1.2,.35) {$\dots$};
  \node at (1.8,1.65) {$\dots$};
  %% DOWNWARD ORIENTED LINES
  \filldraw  (3.62,1.7) circle (2pt);
  \draw[thick,<-] (6,0) .. controls (6,1.25) and (3.6,.25) .. (3.6,2);
  \draw[thick,<-] (3.6,0) .. controls (3.6,1) and (4.2,.8) .. (4.2,2);
  \draw[thick,<-] (4.2,0) .. controls (4.2,1) and (4.8,.8) .. (4.8,2);
  \draw[thick,<-] (5.4,0) .. controls (5.4,1) and (6,.8) .. (6,2);
  \node at (4.8,.35) {$\dots$};
  \node at (5.4,1.65) {$\dots$};
\end{tikzpicture}}
\; \;\;\right]
 \;\; \refequal{\eqref{eq:rel2}} \left[ \;
\hackcenter{
\begin{tikzpicture}[scale=0.8]
 %% Separate lines by 0.6
 %% UPWARD ORIENTED LINES
  \draw[thick] (3,0) .. controls (3,1.25) and (0,.25) .. (0,1.5);
  \draw[thick] (0,0) .. controls (0,1) and (.6,.8) .. (.6,1.5);
  \draw[thick] (.6,0) .. controls (.6,1) and (1.2,.8) .. (1.2,1.5);
  \draw[thick] (1.8,0) .. controls (1.8,1) and (2.4,.8) .. (2.4,1.5);
  \draw[thick] (2.4,0) .. controls (2.4,1) and (3,.8) .. (3,1.5);
  \node at (1.2,.35) {$\dots$};
  \node at (1.8,1.15) {$\dots$};
  %%
  %% DOWNWARD ORIENTED LINES
  \draw[thick] (6,0) .. controls (6,1.25) and (3.6,.25) .. (3.6,1.5)
     node[pos=0.85, shape=coordinate](X){};
  \draw[thick] (3.6,0) .. controls (3.6,1) and (4.2,.8) .. (4.2,1.5);
  \draw[thick] (4.2,0) .. controls (4.2,1) and (4.8,.8) .. (4.8,1.5);
  \draw[thick] (5.4,0) .. controls (5.4,1) and (6,.8) .. (6,1.5);
  \node at (4.8,.35) {$\dots$};
  \node at (5.4,1.15) {$\dots$};
  %%
  %% BLUE LINES
    \draw[blue, dotted] (-0.4,0) -- (6.4,0);
    \draw[blue, dotted] (-0.4,1.5) -- (6.4,1.5);
  %% TOP LEVEL
   \draw[thick,->]  (3,1.5) .. controls (3,2) and (3.6,2) .. (3.6,2.5);
   \draw[thick] (3.6,1.5) .. controls (3.6,2) and (3,2) .. (3,2.5);
   \draw[thick,->] (0,1.5) -- (0,2.5);
   \draw[thick,->] (.6,1.5) -- (.6,2.5);
   \draw[thick,->] (1.2,1.5) -- (1.2,2.5);
   \draw[thick,->] (2.4,1.5) -- (2.4,2.5);
   \draw[thick] (4.2,1.5) -- (4.2,2.5);
   \draw[thick] (4.8,1.5) -- (4.8,2.5);
   \draw[thick] (6,1.5) -- (6,2.5);
   \filldraw  (3.55,1.7) circle (2pt);
  %% BOTTOM LEVEL
   \draw[thick]  (3,0) .. controls (3,-.5) and (3.6,-.5) .. (3.6,-1);
  \draw[thick ,->] (3.6,0) .. controls (3.6,-.5) and (3,-.5) .. (3,-1);
   \draw[thick] (0,0) -- (0,-1);
   \draw[thick] (.6,0) -- (.6,-1);
   \draw[thick] (1.8,0) -- (1.8,-1);
   \draw[thick] (2.4,0) -- (2.4,-1);
   \draw[thick,->] (4.2,0) -- (4.2,-1);
   \draw[thick,->] (5.4,0) -- (5.4,-1);
   \draw[thick,->] (6,0) -- (6,-1);
\end{tikzpicture}}
\right]
\]
\begin{equation} \label{eq:p1}
\;\; = \;\;
\left[ \;
\hackcenter{
\begin{tikzpicture}[scale=0.8]
 %% Separate lines by 0.6
 %% UPWARD ORIENTED LINES
  \draw[thick] (3,0) .. controls (3,1.25) and (0,.25) .. (0,1.5);
  \draw[thick] (0,0) .. controls (0,1) and (.6,.8) .. (.6,1.5);
  \draw[thick] (.6,0) .. controls (.6,1) and (1.2,.8) .. (1.2,1.5);
  \draw[thick] (1.8,0) .. controls (1.8,1) and (2.4,.8) .. (2.4,1.5);
  \draw[thick] (2.4,0) .. controls (2.4,1) and (3,.8) .. (3,1.5);
  \node at (1.2,.35) {$\dots$};
  \node at (1.8,1.15) {$\dots$};
  %%
  %% DOWNWARD ORIENTED LINES
  \draw[thick] (6,0) .. controls (6,1.25) and (3.6,.25) .. (3.6,1.5)
     node[pos=0.85, shape=coordinate](X){};
  \draw[thick] (3.6,0) .. controls (3.6,1) and (4.2,.8) .. (4.2,1.5);
  \draw[thick] (4.2,0) .. controls (4.2,1) and (4.8,.8) .. (4.8,1.5);
  \draw[thick] (5.4,0) .. controls (5.4,1) and (6,.8) .. (6,1.5);
  \node at (4.8,.35) {$\dots$};
  \node at (5.4,1.15) {$\dots$};
  %%
  %% BLUE LINES
    \draw[blue, dotted] (-0.4,0) -- (6.4,0);
    \draw[blue, dotted] (-0.4,1.5) -- (6.4,1.5);
  %% TOP LEVEL
   \draw[thick,->]  (3,1.5) .. controls (3,2) and (3.6,2) .. (3.6,2.5);
   \draw[thick] (3.6,1.5) .. controls (3.6,2) and (3,2) .. (3,2.5);
   \draw[thick,->] (0,1.5) -- (0,2.5);
   \draw[thick,->] (.6,1.5) -- (.6,2.5);
   \draw[thick,->] (1.2,1.5) -- (1.2,2.5);
   \draw[thick,->] (2.4,1.5) -- (2.4,2.5);
   \draw[thick] (4.2,1.5) -- (4.2,2.5);
   \draw[thick] (4.8,1.5) -- (4.8,2.5);
   \draw[thick] (6,1.5) -- (6,2.5);
   \filldraw  (3.05,2.25) circle (2pt);
  %% BOTTOM LEVEL
   \draw[thick]  (3,0) .. controls (3,-.5) and (3.6,-.5) .. (3.6,-1);
  \draw[thick ,->] (3.6,0) .. controls (3.6,-.5) and (3,-.5) .. (3,-1);
   \draw[thick] (0,0) -- (0,-1);
   \draw[thick] (.6,0) -- (.6,-1);
   \draw[thick] (1.8,0) -- (1.8,-1);
   \draw[thick] (2.4,0) -- (2.4,-1);
   \draw[thick,->] (4.2,0) -- (4.2,-1);
   \draw[thick,->] (5.4,0) -- (5.4,-1);
   \draw[thick,->] (6,0) -- (6,-1);
\end{tikzpicture}}
\right]
\;\; - \;\;
\left[ \;
\hackcenter{
\begin{tikzpicture}[scale=0.8]
 %% Separate lines by 0.6
 %% UPWARD ORIENTED LINES
  \draw[thick] (3,0) .. controls (3,1.25) and (0,.25) .. (0,1.5);
  \draw[thick] (0,0) .. controls (0,1) and (.6,.8) .. (.6,1.5);
  \draw[thick] (.6,0) .. controls (.6,1) and (1.2,.8) .. (1.2,1.5);
  \draw[thick] (1.8,0) .. controls (1.8,1) and (2.4,.8) .. (2.4,1.5);
  \draw[thick] (2.4,0) .. controls (2.4,1) and (3,.8) .. (3,1.5);
  \node at (1.2,.35) {$\dots$};
  \node at (1.8,1.15) {$\dots$};
  %%
  %% DOWNWARD ORIENTED LINES
  \draw[thick] (6,0) .. controls (6,1.25) and (3.6,.25) .. (3.6,1.5)
     node[pos=0.85, shape=coordinate](X){};
  \draw[thick] (3.6,0) .. controls (3.6,1) and (4.2,.8) .. (4.2,1.5);
  \draw[thick] (4.2,0) .. controls (4.2,1) and (4.8,.8) .. (4.8,1.5);
  \draw[thick] (5.4,0) .. controls (5.4,1) and (6,.8) .. (6,1.5);
  \node at (4.8,.35) {$\dots$};
  \node at (5.4,1.15) {$\dots$};
  %%
  %% BLUE LINES
    \draw[blue, dotted] (-0.4,0) -- (6.4,0);
    \draw[blue, dotted] (-0.4,1.5) -- (6.4,1.5);
  %% TOP LEVEL
   \draw[thick,->]  (3,1.5) .. controls ++(0,.4) and ++(0,.4) .. (3.6,1.5);
   \draw[thick,<-] (3.6,2.5) .. controls ++(0,-.4) and ++(0,-.4) .. (3,2.5);
   \draw[thick,->] (0,1.5) -- (0,2.5);
   \draw[thick,->] (.6,1.5) -- (.6,2.5);
   \draw[thick,->] (1.2,1.5) -- (1.2,2.5);
   \draw[thick,->] (2.4,1.5) -- (2.4,2.5);
   \draw[thick] (4.2,1.5) -- (4.2,2.5);
   \draw[thick] (4.8,1.5) -- (4.8,2.5);
   \draw[thick] (6,1.5) -- (6,2.5);
  %% BOTTOM LEVEL
   \draw[thick]  (3,0) .. controls (3,-.5) and (3.6,-.5) .. (3.6,-1);
  \draw[thick ,->] (3.6,0) .. controls (3.6,-.5) and (3,-.5) .. (3,-1);
   \draw[thick] (0,0) -- (0,-1);
   \draw[thick] (.6,0) -- (.6,-1);
   \draw[thick] (1.8,0) -- (1.8,-1);
   \draw[thick] (2.4,0) -- (2.4,-1);
   \draw[thick,->] (4.2,0) -- (4.2,-1);
   \draw[thick,->] (5.4,0) -- (5.4,-1);
   \draw[thick,->] (6,0) -- (6,-1);
\end{tikzpicture}}
\right]
\end{equation}
The second diagram is zero since the upper cup can be pulled to the bottom of the diagram using the trace, thereby creating a left twist curl.   Applying the same technique inductively, the first downward oriented strand and its dot can be brought all the way to the left of the diagram.
\[
\;\; = \;\;
\left[ \;
\hackcenter{
\begin{tikzpicture}[scale=0.8]
 %% Separate lines by 0.6
 %% UPWARD ORIENTED LINES
  \draw[thick] (3,0) .. controls ++(0,1.25) and ++(0,-1.75) .. (0,1.5);
  \draw[thick] (0,0) .. controls ++(0,1) and ++(0,-.7) .. (.6,1.5);
  \draw[thick] (.6,0) .. controls ++(0,1) and ++(0,-.7) .. (1.2,1.5);
  \draw[thick] (1.8,0) .. controls ++(0,1) and ++(0,-.7) .. (2.4,1.5);
  \draw[thick] (2.4,0) .. controls ++(0,1) and ++(0,-.7) .. (3,1.5);
  \node at (1.2,.35) {$\dots$};
  \node at (1.8,1.15) {$\dots$};
  %%
  %% DOWNWARD ORIENTED LINES
  \draw[thick] (6,0) .. controls ++(0,1.25) and ++(0,-1.75) .. (3.6,1.5) ;
  \draw[thick] (3.6,0) .. controls ++(0,1) and ++(0,-.7) .. (4.2,1.5);
  \draw[thick] (4.2,0) .. controls ++(0,1) and ++(0,-.7) .. (4.8,1.5);
  \draw[thick] (5.4,0) .. controls ++(0,1) and ++(0,-.7) .. (6,1.5);
  \node at (4.8,-.65) {$\dots$};
  \node at (5.4,1.85) {$\dots$};
  %%
  %% BLUE LINES
    \draw[blue, dotted] (-0.4,0) -- (6.4,0);
    \draw[blue, dotted] (-0.4,1.5) -- (6.4,1.5);
  %% TOP LEVEL
   \draw[thick,->]  (3,1.5) .. controls ++(0,0.5) and ++(0,-.5) .. (3.6,2.5);
   \draw[thick] (3.6,1.5) .. controls ++(0,1) and ++(0,-1) .. (0,2.5);
   \draw[thick,->]  (2.4,1.5) .. controls ++(0,0.5) and ++(0,-.5) .. (3,2.5);
   \draw[thick,->]  (1.2,1.5) .. controls ++(0,0.5) and ++(0,-.5) .. (1.8,2.5);
   \draw[thick,->]  (.6,1.5) .. controls ++(0,0.5) and ++(0,-.5) .. (1.2,2.5);
   \draw[thick,->]  (0,1.5) .. controls ++(0,0.5) and ++(0,-.5) .. (.6,2.5);
   \draw[thick] (4.2,1.5) -- (4.2,2.5);
   \draw[thick] (4.8,1.5) -- (4.8,2.5);
   \draw[thick] (6,1.5) -- (6,2.5);
     \filldraw  (.15,2.2) circle (2pt);
  %% BOTTOM LEVEL
   \draw[thick]  (2.4,0) .. controls ++(0,-.5) and ++(0,.5) .. (3,-1);
   \draw[thick]  (1.8,0) .. controls ++(0,-.5) and ++(0,.5) .. (2.4,-1);
   \draw[thick]  (.6,0) .. controls ++(0,-.5) and ++(0,.5) .. (1.2,-1);
   \draw[thick]  (0,0) .. controls ++(0,-.5) and ++(0,.5) .. (.6,-1);
   \draw[thick]  (3,0) .. controls ++(0,-.5) and ++(0,.5) .. (3.6,-1);
  \draw[thick ,->] (3.6,0) .. controls ++(0,-1) and ++(0,+1) .. (0,-1);
   \draw[thick,->] (4.2,0) -- (4.2,-1);
   \draw[thick,->] (5.4,0) -- (5.4,-1);
   \draw[thick,->] (6,0) -- (6,-1);
\end{tikzpicture}}
\right]
\]
and the result claim follows by lemma ~\ref{lem:alem2}.
\end{proof}

\begin{lemma}
\label{n1-m1}
For integers $m,n \geq 1$ we have
\begin{equation*}
[h_{n} \otimes x_1, h_{-m} \otimes 1]=
\begin{cases}
-2m(h_{n-m} \otimes 1) & \text{ if } n  > m \geq 1 \\
0 & \text{ if } n = m \geq 1 \\
-m(h_{n-m} \otimes 1) & \text{ if } 1 \leq n < m.
\end{cases}
\end{equation*}
\end{lemma}

\begin{proof}
For this proof let $ \beta_n=h_n \otimes x_1 $ and $ \alpha_n = h_n \otimes 1$.
We prove this by induction on $m$.  For $m=1$ we give a graphical proof.  Consider $(h_{-1} \otimes 1)(h_n \otimes x_1)$
\[
 \left[ \;\;\;
\hackcenter{
\begin{tikzpicture}[scale=0.8]
  \draw[thick,->] (3,0) .. controls ++(0,1.25) and (0,.25) .. (0,2);
  \draw[thick,->] (0.0,0) .. controls ++(0,1) and ++(0,-1.2) .. (0.6,2);
  \draw[thick,->] (0.6,0) .. controls ++(0,1) and ++(0,-1.2) .. (1.2,2);
  \draw[thick,->] (1.2,0) .. controls ++(0,1) and ++(0,-1.2) .. (1.8,2);
  \draw[thick,->] (2.4,0) .. controls ++(0,1) and ++(0,-1.2) .. (3,2);
  \node at (1.8,.35) {$\dots$};
  \node at (2.4,1.65) {$\dots$};
  \draw[thick,<-] (-0.6,0) -- (-0.6,2);
  \filldraw  (.05,1.6) circle (2pt);
\end{tikzpicture}}
 \;\;\right]
\;\; \refequal{\eqref{eq:rel2}} \;\;
 \left[ \;\;\;
\hackcenter{
\begin{tikzpicture}[scale=0.8]
  \draw[thick] (3.0,0) .. controls ++(0,1.25) and ++(0,-1.1) .. (0,1.5);
  \draw[thick] (0.0,0) .. controls ++(0,.5) and ++(0,-.6) .. (0.6,1.5);
  \draw[thick] (0.6,0) .. controls ++(0,.5) and ++(0,-.6) .. (1.2,1.5);
  \draw[thick] (1.2,0) .. controls ++(0,.5) and ++(0,-.6) .. (1.8,1.5);
  \draw[thick] (2.4,0) .. controls ++(0,.5) and ++(0,-.6) .. (3,1.5);
  \node at (1.8,.35) {$\dots$};
  \node at (2.4,1.35) {$\dots$};
  \draw[thick,<-] (-0.6,0) -- (-0.6,1.5);
  \filldraw  (.1,1.2) circle (2pt);
    %%
  %% BLUE LINES
    \draw[blue, dotted] (-1,0) -- (3.4,0);
    \draw[blue, dotted] (-1,1.5) -- (3.4,1.5);
  %% TOP LINES
   \draw[thick,->] (0.0,1.5) .. controls ++(0,.45) and ++(0,-.6) .. (-0.6,2.5);
   \draw[thick] (-0.6,1.5) .. controls ++(0,.45) and ++(0,-.6) .. (0.0,2.5);
   \draw[thick,->] (0.6,1.5) --  (0.6,2.5);
   \draw[thick,->] (1.2,1.5) --  (1.2,2.5);
   \draw[thick,->] (1.8,1.5) --  (1.8,2.5);
   \draw[thick,->] (3,1.5) --  (3,2.5);
   %% bottom LINES
   \draw[thick,<-] (0.0,0) .. controls ++(0,-.45) and ++(0,.6) .. (-0.6,-1);
   \draw[thick,->] (-0.6,0) .. controls ++(0,-.45) and ++(0,.6) .. (0.0,-1);
   \draw[thick,<-] (0.6,0) --  (0.6,-1);
   \draw[thick,<-] (1.2,0) --  (1.2,-1);
   \draw[thick,<-] (2.4,0) --  (2.4,-1);
   \draw[thick,<-] (3.0,0) --  (3,-1);
\end{tikzpicture}}
\;\;\right]
\;\; + \;\;
 \left[ \;\;\;
\hackcenter{
\begin{tikzpicture}[scale=0.8]
  \draw[thick] (3.0,0) .. controls ++(0,1.25) and ++(0,-1.1) .. (0,1.5);
  \draw[thick] (0.0,0) .. controls ++(0,.5) and ++(0,-.6) .. (0.6,1.5);
  \draw[thick] (0.6,0) .. controls ++(0,.5) and ++(0,-.6) .. (1.2,1.5);
  \draw[thick] (1.2,0) .. controls ++(0,.5) and ++(0,-.6) .. (1.8,1.5);
  \draw[thick] (2.4,0) .. controls ++(0,.5) and ++(0,-.6) .. (3,1.5);
  \node at (1.8,.35) {$\dots$};
  \node at (2.4,1.35) {$\dots$};
  \draw[thick] (-0.6,0) -- (-0.6,1.5);
  \filldraw  (.1,1.2) circle (2pt);
    %%
  %% BLUE LINES
    \draw[blue, dotted] (-1,0) -- (3.4,0);
    \draw[blue, dotted] (-1,1.5) -- (3.4,1.5);
  %% TOP LINES
  % \draw[thick,->] (0.0,1.5) .. controls ++(0,.45) and ++(0,-.6) .. (-0.6,2.5);
   \draw[thick] (-0.6,1.5) .. controls ++(0,.35) and ++(0,.35) .. (0.0,1.5);
   \draw[thick,->] (0.6,1.5) --  (0.6,2.5);
   \draw[thick,->] (1.2,1.5) --  (1.2,2.5);
   \draw[thick,->] (1.8,1.5) --  (1.8,2.5);
   \draw[thick,->] (3,1.5) --  (3,2.5);
   %% bottom LINES
   %\draw[thick,<-] (0.0,0) .. controls ++(0,-.45) and ++(0,.6) .. (-0.6,-1);
   \draw[thick,->] (-0.6,0) .. controls ++(0,-.35) and ++(0,-.35) .. (0.0,0);
   \draw[thick,<-] (0.6,0) --  (0.6,-1);
   \draw[thick,<-] (1.2,0) --  (1.2,-1);
   \draw[thick,<-] (2.4,0) --  (2.4,-1);
   \draw[thick,<-] (3.0,0) --  (3,-1);
\end{tikzpicture}}
\;\;\right]
\]
\[
\;\; \refequal{\eqref{eq:nil-dot}} \;\;
 \left[ \;\;\;
\hackcenter{
\begin{tikzpicture}[scale=0.8]
  \draw[thick] (3.0,0) .. controls ++(0,1.25) and ++(0,-1.1) .. (0,1.5);
  \draw[thick] (0.0,0) .. controls ++(0,.5) and ++(0,-.6) .. (0.6,1.5);
  \draw[thick] (0.6,0) .. controls ++(0,.5) and ++(0,-.6) .. (1.2,1.5);
  \draw[thick] (1.2,0) .. controls ++(0,.5) and ++(0,-.6) .. (1.8,1.5);
  \draw[thick] (2.4,0) .. controls ++(0,.5) and ++(0,-.6) .. (3,1.5);
  \node at (1.8,.35) {$\dots$};
  \node at (2.4,1.35) {$\dots$};
  \draw[thick,<-] (-0.6,0) -- (-0.6,1.5);
  \filldraw  (-0.55,2.2) circle (2pt);
    %%
  %% BLUE LINES
    \draw[blue, dotted] (-1,0) -- (3.4,0);
    \draw[blue, dotted] (-1,1.5) -- (3.4,1.5);
  %% TOP LINES
   \draw[thick,->] (0.0,1.5) .. controls ++(0,.45) and ++(0,-.6) .. (-0.6,2.5);
   \draw[thick] (-0.6,1.5) .. controls ++(0,.45) and ++(0,-.6) .. (0.0,2.5);
   \draw[thick,->] (0.6,1.5) --  (0.6,2.5);
   \draw[thick,->] (1.2,1.5) --  (1.2,2.5);
   \draw[thick,->] (1.8,1.5) --  (1.8,2.5);
   \draw[thick,->] (3,1.5) --  (3,2.5);
   %% bottom LINES
   \draw[thick,<-] (0.0,0) .. controls ++(0,-.45) and ++(0,.6) .. (-0.6,-1);
   \draw[thick,->] (-0.6,0) .. controls ++(0,-.45) and ++(0,.6) .. (0.0,-1);
   \draw[thick,<-] (0.6,0) --  (0.6,-1);
   \draw[thick,<-] (1.2,0) --  (1.2,-1);
   \draw[thick,<-] (2.4,0) --  (2.4,-1);
   \draw[thick,<-] (3.0,0) --  (3,-1);
\end{tikzpicture}}
\;\;\right]
\;\; + \;\;
 \left[ \;\;\;
\hackcenter{
\begin{tikzpicture}[scale=0.8]
  \draw[thick] (3.0,0) .. controls ++(0,1.25) and ++(0,-1.1) .. (0,1.5);
  \draw[thick] (0.0,0) .. controls ++(0,.5) and ++(0,-.6) .. (0.6,1.5);
  \draw[thick] (0.6,0) .. controls ++(0,.5) and ++(0,-.6) .. (1.2,1.5);
  \draw[thick] (1.2,0) .. controls ++(0,.5) and ++(0,-.6) .. (1.8,1.5);
  \draw[thick] (2.4,0) .. controls ++(0,.5) and ++(0,-.6) .. (3,1.5);
  \node at (1.8,.35) {$\dots$};
  \node at (2.4,1.35) {$\dots$};
  \draw[thick] (-0.6,0) -- (-0.6,1.5);
    %%
  %% BLUE LINES
    \draw[blue, dotted] (-1,0) -- (3.4,0);
    \draw[blue, dotted] (-1,1.5) -- (3.4,1.5);
  %% TOP LINES
   \draw[thick,->] (0.0,2.5) .. controls ++(0,-.35) and ++(0,-.35) .. (-0.6,2.5);
   \draw[thick] (-0.6,1.5) .. controls ++(0,.35) and ++(0,.35) .. (0.0,1.5);
   \draw[thick,->] (0.6,1.5) --  (0.6,2.5);
   \draw[thick,->] (1.2,1.5) --  (1.2,2.5);
   \draw[thick,->] (1.8,1.5) --  (1.8,2.5);
   \draw[thick,->] (3,1.5) --  (3,2.5);
   %% bottom LINES
   \draw[thick,<-] (0.0,0) .. controls ++(0,-.45) and ++(0,.6) .. (-0.6,-1);
   \draw[thick,->] (-0.6,0) .. controls ++(0,-.45) and ++(0,.6) .. (0.0,-1);
   \draw[thick,<-] (0.6,0) --  (0.6,-1);
   \draw[thick,<-] (1.2,0) --  (1.2,-1);
   \draw[thick,<-] (2.4,0) --  (2.4,-1);
   \draw[thick,<-] (3.0,0) --  (3,-1);
\end{tikzpicture}}
\;\;\right]
\;\; + \;\; h_{n-1} \otimes 1
\]
where we used \eqref{eq:dotted-curl} to simplify the dotted left twist curl.  Using the trace relation, the second term above is also equal to $h_{n-1}\otimes 1$.  For the first term we use the first equation in \eqref{eq:rel2} to slide the upward strands to the left.
\[
 \left[ \;\;\;
\hackcenter{
\begin{tikzpicture}[scale=0.8]
  \draw[thick] (3.0,0) .. controls ++(0,1.25) and ++(0,-1.1) .. (0,1.5);
  \draw[thick] (0.0,0) .. controls ++(0,.5) and ++(0,-.6) .. (0.6,1.5);
  \draw[thick] (0.6,0) .. controls ++(0,.5) and ++(0,-.6) .. (1.2,1.5);
  \draw[thick] (1.2,0) .. controls ++(0,.5) and ++(0,-.6) .. (1.8,1.5);
  \draw[thick] (2.4,0) .. controls ++(0,.5) and ++(0,-.6) .. (3,1.5);
  \node at (1.8,.35) {$\dots$};
  \node at (2.4,1.35) {$\dots$};
  \draw[thick,<-] (-0.6,0) -- (-0.6,1.5);
  \filldraw  (-0.55,2.2) circle (2pt);
    %%
  %% BLUE LINES
    \draw[blue, dotted] (-1,0) -- (3.4,0);
    \draw[blue, dotted] (-1,1.5) -- (3.4,1.5);
  %% TOP LINES
   \draw[thick,->] (0.0,1.5) .. controls ++(0,.45) and ++(0,-.6) .. (-0.6,2.5);
   \draw[thick] (-0.6,1.5) .. controls ++(0,.45) and ++(0,-.6) .. (0.0,2.5);
   \draw[thick,->] (0.6,1.5) --  (0.6,2.5);
   \draw[thick,->] (1.2,1.5) --  (1.2,2.5);
   \draw[thick,->] (1.8,1.5) --  (1.8,2.5);
   \draw[thick,->] (3,1.5) --  (3,2.5);
   %% bottom LINES
   \draw[thick,<-] (0.0,0) .. controls ++(0,-.45) and ++(0,.6) .. (-0.6,-1);
   \draw[thick,->] (-0.6,0) .. controls ++(0,-.45) and ++(0,.6) .. (0.0,-1);
   \draw[thick,<-] (0.6,0) --  (0.6,-1);
   \draw[thick,<-] (1.2,0) --  (1.2,-1);
   \draw[thick,<-] (2.4,0) --  (2.4,-1);
   \draw[thick,<-] (3.0,0) --  (3,-1);
\end{tikzpicture}}
\;\;\right]
\;\; = \;\;
 \left[ \;\;\;
\hackcenter{
\begin{tikzpicture}[scale=0.8]
  \draw[thick] (3.0,0) .. controls ++(0,1.25) and ++(0,-1.1) .. (0,1.5);
  \draw[thick] (0.0,0) .. controls ++(0,.5) and ++(0,-.6) .. (0.6,1.5);
  \draw[thick] (0.6,0) .. controls ++(0,.5) and ++(0,-.6) .. (1.2,1.5);
  \draw[thick] (1.2,0) .. controls ++(0,.5) and ++(0,-.6) .. (1.8,1.5);
  \draw[thick] (2.4,0) .. controls ++(0,.5) and ++(0,-.6) .. (3,1.5);
  \node at (1.8,.35) {$\dots$};
  \node at (2.4,1.35) {$\dots$};
  \draw[thick,<-] (-0.6,0) -- (-0.6,1.5);
  \filldraw  (-0.55,2.2) circle (2pt);
    %%
  %% BLUE LINES
    \draw[blue, dotted] (-1,0) -- (3.4,0);
    \draw[blue, dotted] (-1,1.5) -- (3.4,1.5);
  %% TOP LINES
   \draw[thick,->] (0.0,1.5) .. controls ++(0,.45) and ++(0,-.6) .. (-0.6,2.5);
   \draw[thick] (-0.6,1.5) .. controls ++(0,.85) and ++(0,-1) .. (3.0,2.5);
   \draw[thick,->] (0.6,1.5).. controls ++(0,.45) and ++(0,-.6) .. (-0.0,2.5);
   \draw[thick,->] (1.2,1.5) .. controls ++(0,.45) and ++(0,-.6) .. (0.6,2.5);
   \draw[thick,->] (1.8,1.5) .. controls ++(0,.45) and ++(0,-.6) .. (1.2,2.5);
   \draw[thick,->] (3,1.5) .. controls ++(0,.45) and ++(0,-.6) .. (2.4,2.5);
   %% bottom LINES
   \draw[thick] (0.0,0) .. controls ++(0,-.45) and ++(0,.6) .. (-0.6,-1);
   \draw[thick,->]   (-0.6,0) .. controls ++(0,-.85) and ++(0,1) .. (3.0,-1);
   \draw[thick] (0.6,0) .. controls ++(0,-.45) and ++(0,.6) .. (-0.0,-1);
   \draw[thick] (1.2,0) .. controls ++(0,-.45) and ++(0,.6) .. (0.6,-1);
   \draw[thick] (2.4,0) .. controls ++(0,-.45) and ++(0,.6) .. (1.8,-1);
   \draw[thick] (3,0)   .. controls ++(0,-.45) and ++(0,.6) .. (2.4,-1);
\end{tikzpicture}}
\;\;\right]
\;\; = \;\;
 \left[ \;\;\;
\hackcenter{
\begin{tikzpicture}[scale=0.8]
  \draw[thick,->] (3,0) .. controls ++(0,1.25) and (0,.25) .. (0,2);
  \draw[thick,->] (0.0,0) .. controls ++(0,1) and ++(0,-1.2) .. (0.6,2);
  \draw[thick,->] (0.6,0) .. controls ++(0,1) and ++(0,-1.2) .. (1.2,2);
  \draw[thick,->] (1.2,0) .. controls ++(0,1) and ++(0,-1.2) .. (1.8,2);
  \draw[thick,->] (2.4,0) .. controls ++(0,1) and ++(0,-1.2) .. (3,2);
  \node at (1.8,.35) {$\dots$};
  \node at (2.4,1.65) {$\dots$};
  \draw[thick,<-] (3.6,0) -- (3.6,2);
  \filldraw  (.05,1.6) circle (2pt);
\end{tikzpicture}}
 \;\;\;\right]
\]
where the first equality holds since all the resolutions terms
\[
 \left[ \;\;\;
\hackcenter{
\begin{tikzpicture}[scale=0.8]
  \draw[thick] (5.4,0) .. controls ++(0,1.25) and ++(0,-1.1) .. (0,1.5);
  \draw[thick] (0.0,0) .. controls ++(0,.5) and ++(0,-.6) .. (0.6,1.5);
  \draw[thick] (0.6,0) .. controls ++(0,.5) and ++(0,-.6) .. (1.2,1.5);
    \node at (1.2,.35) {$\dots$};
  \node at (1.8,1.35) {$\dots$};
  \draw[thick] (1.8,0) .. controls ++(0,.5) and ++(0,-.6) .. (2.4,1.5);
  \draw[very thick,red] (2.4,0) .. controls ++(0,.5) and ++(0,-.6) .. (3,1.5);
  \draw[thick] (3,0) .. controls ++(0,.5) and ++(0,-.6) .. (3.6,1.5);
  \draw[thick] (3.6,0) .. controls ++(0,.5) and ++(0,-.6) .. (4.2,1.5);
  \node at (4.2,.35) {$\dots$};
  \node at (4.8,1.35) {$\dots$};
  \draw[thick] (4.8,0) .. controls ++(0,.5) and ++(0,-.6) .. (5.4,1.5);
  \draw[very thick,red,<-] (-0.6,0) -- (-0.6,1.5);
  \filldraw  (-0.55,2.2) circle (2pt);
    %%
  %% BLUE LINES
    \draw[blue, dotted] (-1,0) -- (3.4,0);
    \draw[blue, dotted] (-1,1.5) -- (3.4,1.5);
  %% TOP LINES
   \draw[thick,->] (0.0,1.5) .. controls ++(0,.45) and ++(0,-.6) .. (-0.6,2.5);
   \draw[very thick,red,] (-0.6,1.5) .. controls ++(0,.75) and ++(0.4,1) .. (3.0,1.5);
   \draw[thick,->] (0.6,1.5).. controls ++(0,.45) and ++(0,-.6) .. (-0.0,2.5);
   \draw[thick,->] (1.2,1.5) .. controls ++(0,.45) and ++(0,-.6) .. (0.6,2.5);
   \draw[thick,->] (2.4,1.5) .. controls ++(0,.45) and ++(0,-.6) .. (1.8,2.5);
   \draw[thick,->] (3.6,1.5) -- (3.6,2.5);
   \draw[thick,->] (4.2,1.5) -- (4.2,2.5);
   \draw[thick,->] (5.4,1.5) -- (5.4,2.5);
   %% bottom LINES
   \draw[thick] (0.0,0) .. controls ++(0,-.45) and ++(0,.6) .. (-0.6,-1);
   \draw[very thick,red,->]   (-0.6,0) .. controls ++(0,-.75) and ++(0.2,-1) .. (3.0,0);
   \draw[thick] (0.6,0) .. controls ++(0,-.45) and ++(0,.6) .. (-0.0,-1);
   \draw[thick] (1.8,0) .. controls ++(0,-.45) and ++(0,.6) .. (1.2,-1);
   \draw[very thick,red] (2.4,0) .. controls ++(0,-.45) and ++(0,.6) .. (1.8,-1);
   \draw[thick] (3.6,0) -- (3.6,-1);
   \draw[thick] (4.8,0) -- (4.8,-1);
   \draw[thick] (5.4,0) -- (5.4,-1);
\end{tikzpicture}}
\;\;\right]
\]
contain left twist curls.  Thus, we have proven the base case of our induction.

Graphically it is easy to see that for some constant $ \gamma_{n,-m}$ that
\begin{equation*}
[\beta_n, \alpha_{-m}]=\gamma_{n,-m} \alpha_{n-m}.
\end{equation*}
In order to compute the constant $\gamma_{n,-m}$ consider the Jacobi identity
\begin{equation}
\label{n1-m1eq1}
[\beta_n,[\beta_{-1},\alpha_{-m+1}]]+[\beta_{-1},[\alpha_{-m+1},\beta_n]]+[\alpha_{-m+1},[\beta_n,\beta_{-1}]]=0.
\end{equation}
Now applying Lemma ~\ref{posvirposheis}, ~\eqref{n1-m1eq1} becomes
\begin{equation*}
(1-m)[\beta_n,\alpha_{-m}]-\gamma_{n,-m+1}[\beta_{-1},\alpha_{n-m+1}]+(n+1)[\beta_{n-1},\alpha_{-m+1}]=0.
\end{equation*}
Thus
\begin{align*}
[\beta_n,\alpha_{-m}] &= \frac{\gamma_{n,-m+1}[\beta_{-1},\alpha_{n-m+1}]-(n+1) \gamma_{n-1,-m+1} \alpha_{n-m}}   {1-m} \\
&= \frac{ (n-m+1)\gamma_{n,-m+1}-(n+1) \gamma_{n-1,-m+1}}{1-m}\alpha_{n-m}.
\end{align*}
The lemma now easily follows by induction.
\end{proof}

Along with Lemma ~\ref{lemmam1n1}, the next lemma will lead to Virasoro relations.

\begin{lemma}
\label{lemma-m1n1}
Let $m$ and $n$ be positive integers and $ T=\text{min}(m,n)$.  Then
\begin{equation*}
[h_{-m} \otimes x_1, h_n \otimes x_1]=(n+m)(h_{n-m} \otimes x_1) - \sum_{j=1}^{T-1} j (h_{n-j} \otimes 1)(h_{-m+j} \otimes 1).
\end{equation*}
\end{lemma}

\begin{proof}
For this proof let $ \alpha_n = h_n \otimes 1 $ and $ \beta_n = h_n \otimes x_1$.
We proceed by induction on $m$.  The base case is $m=1$. The proof is similar to Lemma~\ref{n1-m1}. We begin by using the first equation in \eqref{eq:rel2} and the trace relation.
\[
 \left[ \;\;\;
\hackcenter{
\begin{tikzpicture}[scale=0.8]
  \draw[thick,->] (3,0) .. controls ++(0,1.25) and (0,.25) .. (0,2);
  \draw[thick,->] (0.0,0) .. controls ++(0,1) and ++(0,-1.2) .. (0.6,2);
  \draw[thick,->] (0.6,0) .. controls ++(0,1) and ++(0,-1.2) .. (1.2,2);
  \draw[thick,->] (1.2,0) .. controls ++(0,1) and ++(0,-1.2) .. (1.8,2);
  \draw[thick,->] (2.4,0) .. controls ++(0,1) and ++(0,-1.2) .. (3,2);
  \node at (1.8,.35) {$\dots$};
  \node at (2.4,1.65) {$\dots$};
  \draw[thick,<-] (-0.6,0) -- (-0.6,2);
  \filldraw  (-0.6,1.2) circle (2pt);
  \filldraw  (.05,1.6) circle (2pt);
\end{tikzpicture}}
 \;\;\right]
\;\; \refequal{\eqref{eq:rel2}} \;\;
 \left[ \;\;\;
\hackcenter{
\begin{tikzpicture}[scale=0.8]
  \draw[thick] (3.0,0) .. controls ++(0,1.25) and ++(0,-1.1) .. (0,1.5);
  \draw[thick] (0.0,0) .. controls ++(0,.5) and ++(0,-.6) .. (0.6,1.5);
  \draw[thick] (0.6,0) .. controls ++(0,.5) and ++(0,-.6) .. (1.2,1.5);
  \draw[thick] (1.2,0) .. controls ++(0,.5) and ++(0,-.6) .. (1.8,1.5);
  \draw[thick] (2.4,0) .. controls ++(0,.5) and ++(0,-.6) .. (3,1.5);
  \node at (1.8,.35) {$\dots$};
  \node at (2.4,1.35) {$\dots$};
  \draw[thick,<-] (-0.6,0) -- (-0.6,1.5);
  \filldraw  (.1,1.2) circle (2pt);
    \filldraw  (-0.6,1.2) circle (2pt);
    %%
  %% BLUE LINES
    \draw[blue, dotted] (-1,0) -- (3.4,0);
    \draw[blue, dotted] (-1,1.5) -- (3.4,1.5);
  %% TOP LINES
   \draw[thick,->] (0.0,1.5) .. controls ++(0,.45) and ++(0,-.6) .. (-0.6,2.5);
   \draw[thick] (-0.6,1.5) .. controls ++(0,.45) and ++(0,-.6) .. (0.0,2.5);
   \draw[thick,->] (0.6,1.5) --  (0.6,2.5);
   \draw[thick,->] (1.2,1.5) --  (1.2,2.5);
   \draw[thick,->] (1.8,1.5) --  (1.8,2.5);
   \draw[thick,->] (3,1.5) --  (3,2.5);
   %% bottom LINES
   \draw[thick,<-] (0.0,0) .. controls ++(0,-.45) and ++(0,.6) .. (-0.6,-1);
   \draw[thick,->] (-0.6,0) .. controls ++(0,-.45) and ++(0,.6) .. (0.0,-1);
   \draw[thick,<-] (0.6,0) --  (0.6,-1);
   \draw[thick,<-] (1.2,0) --  (1.2,-1);
   \draw[thick,<-] (2.4,0) --  (2.4,-1);
   \draw[thick,<-] (3.0,0) --  (3,-1);
\end{tikzpicture}}
\;\;\right]
\;\; + \;\;
 \left[ \;\;\;
\hackcenter{
\begin{tikzpicture}[scale=0.8]
  \draw[thick] (3.0,0) .. controls ++(0,1.25) and ++(0,-1.1) .. (0,1.5);
  \draw[thick] (0.0,0) .. controls ++(0,.5) and ++(0,-.6) .. (0.6,1.5);
  \draw[thick] (0.6,0) .. controls ++(0,.5) and ++(0,-.6) .. (1.2,1.5);
  \draw[thick] (1.2,0) .. controls ++(0,.5) and ++(0,-.6) .. (1.8,1.5);
  \draw[thick] (2.4,0) .. controls ++(0,.5) and ++(0,-.6) .. (3,1.5);
  \node at (1.8,.35) {$\dots$};
  \node at (2.4,1.35) {$\dots$};
  \draw[thick] (-0.6,0) -- (-0.6,1.5);
  \filldraw  (.1,1.2) circle (2pt);
    \filldraw  (-0.6,1.2) circle (2pt);
    %%
  %% BLUE LINES
    \draw[blue, dotted] (-1,0) -- (3.4,0);
    \draw[blue, dotted] (-1,1.5) -- (3.4,1.5);
  %% TOP LINES
  % \draw[thick,->] (0.0,1.5) .. controls ++(0,.45) and ++(0,-.6) .. (-0.6,2.5);
   \draw[thick] (-0.6,1.5) .. controls ++(0,.35) and ++(0,.35) .. (0.0,1.5);
   \draw[thick,->] (0.6,1.5) --  (0.6,2.5);
   \draw[thick,->] (1.2,1.5) --  (1.2,2.5);
   \draw[thick,->] (1.8,1.5) --  (1.8,2.5);
   \draw[thick,->] (3,1.5) --  (3,2.5);
   %% bottom LINES
   %\draw[thick,<-] (0.0,0) .. controls ++(0,-.45) and ++(0,.6) .. (-0.6,-1);
   \draw[thick,->] (-0.6,0) .. controls ++(0,-.35) and ++(0,-.35) .. (0.0,0);
   \draw[thick,<-] (0.6,0) --  (0.6,-1);
   \draw[thick,<-] (1.2,0) --  (1.2,-1);
   \draw[thick,<-] (2.4,0) --  (2.4,-1);
   \draw[thick,<-] (3.0,0) --  (3,-1);
\end{tikzpicture}}
\;\;\right]
\]
Using \eqref{eq:dotted-curl} the second term is equal to $h_{n-1}\otimes x_1$.  For the first term we slide the dots upward  producing
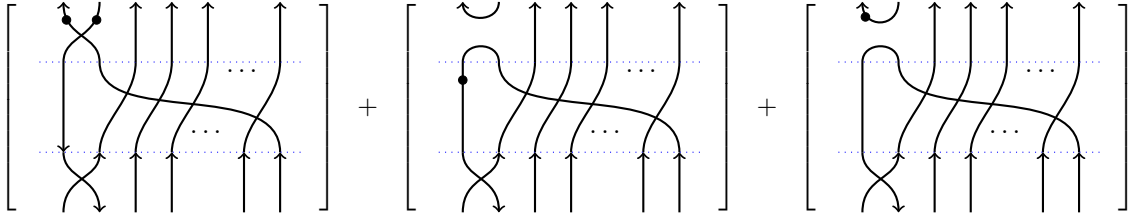
\[
 \left[ \;\;\;
\hackcenter{
\begin{tikzpicture}[scale=0.8]
  \draw[thick] (3.0,0) .. controls ++(0,1.25) and ++(0,-1.1) .. (0,1.5);
  \draw[thick] (0.0,0) .. controls ++(0,.5) and ++(0,-.6) .. (0.6,1.5);
  \draw[thick] (0.6,0) .. controls ++(0,.5) and ++(0,-.6) .. (1.2,1.5);
  \draw[thick] (1.2,0) .. controls ++(0,.5) and ++(0,-.6) .. (1.8,1.5);
  \draw[thick] (2.4,0) .. controls ++(0,.5) and ++(0,-.6) .. (3,1.5);
  \node at (1.8,.35) {$\dots$};
  \node at (2.4,1.35) {$\dots$};
  \draw[thick,<-] (-0.6,0) -- (-0.6,1.5);
  \filldraw  (-0.55,2.2) circle (2pt);
  \filldraw  (-0.05,2.2) circle (2pt);
    %%
  %% BLUE LINES
    \draw[blue, dotted] (-1,0) -- (3.4,0);
    \draw[blue, dotted] (-1,1.5) -- (3.4,1.5);
  %% TOP LINES
   \draw[thick,->] (0.0,1.5) .. controls ++(0,.45) and ++(0,-.6) .. (-0.6,2.5);
   \draw[thick] (-0.6,1.5) .. controls ++(0,.45) and ++(0,-.6) .. (0.0,2.5);
   \draw[thick,->] (0.6,1.5) --  (0.6,2.5);
   \draw[thick,->] (1.2,1.5) --  (1.2,2.5);
   \draw[thick,->] (1.8,1.5) --  (1.8,2.5);
   \draw[thick,->] (3,1.5) --  (3,2.5);
   %% bottom LINES
   \draw[thick,<-] (0.0,0) .. controls ++(0,-.45) and ++(0,.6) .. (-0.6,-1);
   \draw[thick,->] (-0.6,0) .. controls ++(0,-.45) and ++(0,.6) .. (0.0,-1);
   \draw[thick,<-] (0.6,0) --  (0.6,-1);
   \draw[thick,<-] (1.2,0) --  (1.2,-1);
   \draw[thick,<-] (2.4,0) --  (2.4,-1);
   \draw[thick,<-] (3.0,0) --  (3,-1);
\end{tikzpicture}}
\;\;\right]
\;\; + \;\;
 \left[ \;\;\;
\hackcenter{
\begin{tikzpicture}[scale=0.8]
  \draw[thick] (3.0,0) .. controls ++(0,1.25) and ++(0,-1.1) .. (0,1.5);
  \draw[thick] (0.0,0) .. controls ++(0,.5) and ++(0,-.6) .. (0.6,1.5);
  \draw[thick] (0.6,0) .. controls ++(0,.5) and ++(0,-.6) .. (1.2,1.5);
  \draw[thick] (1.2,0) .. controls ++(0,.5) and ++(0,-.6) .. (1.8,1.5);
  \draw[thick] (2.4,0) .. controls ++(0,.5) and ++(0,-.6) .. (3,1.5);
  \node at (1.8,.35) {$\dots$};
  \node at (2.4,1.35) {$\dots$};
  \draw[thick] (-0.6,0) -- (-0.6,1.5);
  \filldraw  (-0.6,1.2) circle (2pt);
    %%
  %% BLUE LINES
    \draw[blue, dotted] (-1,0) -- (3.4,0);
    \draw[blue, dotted] (-1,1.5) -- (3.4,1.5);
  %% TOP LINES
   \draw[thick,->] (0.0,2.5) .. controls ++(0,-.35) and ++(0,-.35) .. (-0.6,2.5);
   \draw[thick] (-0.6,1.5) .. controls ++(0,.35) and ++(0,.35) .. (0.0,1.5);
   \draw[thick,->] (0.6,1.5) --  (0.6,2.5);
   \draw[thick,->] (1.2,1.5) --  (1.2,2.5);
   \draw[thick,->] (1.8,1.5) --  (1.8,2.5);
   \draw[thick,->] (3,1.5) --  (3,2.5);
   %% bottom LINES
   \draw[thick,<-] (0.0,0) .. controls ++(0,-.45) and ++(0,.6) .. (-0.6,-1);
   \draw[thick,->] (-0.6,0) .. controls ++(0,-.45) and ++(0,.6) .. (0.0,-1);
   \draw[thick,<-] (0.6,0) --  (0.6,-1);
   \draw[thick,<-] (1.2,0) --  (1.2,-1);
   \draw[thick,<-] (2.4,0) --  (2.4,-1);
   \draw[thick,<-] (3.0,0) --  (3,-1);
\end{tikzpicture}}
\;\;\right]
\;\; + \;\;
 \left[ \;\;\;
\hackcenter{
\begin{tikzpicture}[scale=0.8]
  \draw[thick] (3.0,0) .. controls ++(0,1.25) and ++(0,-1.1) .. (0,1.5);
  \draw[thick] (0.0,0) .. controls ++(0,.5) and ++(0,-.6) .. (0.6,1.5);
  \draw[thick] (0.6,0) .. controls ++(0,.5) and ++(0,-.6) .. (1.2,1.5);
  \draw[thick] (1.2,0) .. controls ++(0,.5) and ++(0,-.6) .. (1.8,1.5);
  \draw[thick] (2.4,0) .. controls ++(0,.5) and ++(0,-.6) .. (3,1.5);
  \node at (1.8,.35) {$\dots$};
  \node at (2.4,1.35) {$\dots$};
  \draw[thick] (-0.6,0) -- (-0.6,1.5);
  \filldraw  (-0.55,2.25) circle (2pt);
    %%
  %% BLUE LINES
    \draw[blue, dotted] (-1,0) -- (3.4,0);
    \draw[blue, dotted] (-1,1.5) -- (3.4,1.5);
  %% TOP LINES
   \draw[thick,->] (0.0,2.5) .. controls ++(0,-.45) and ++(0,-.45) .. (-0.6,2.5);
   \draw[thick] (-0.6,1.5) .. controls ++(0,.35) and ++(0,.35) .. (0.0,1.5);
   \draw[thick,->] (0.6,1.5) --  (0.6,2.5);
   \draw[thick,->] (1.2,1.5) --  (1.2,2.5);
   \draw[thick,->] (1.8,1.5) --  (1.8,2.5);
   \draw[thick,->] (3,1.5) --  (3,2.5);
   %% bottom LINES
   \draw[thick,<-] (0.0,0) .. controls ++(0,-.45) and ++(0,.6) .. (-0.6,-1);
   \draw[thick,->] (-0.6,0) .. controls ++(0,-.45) and ++(0,.6) .. (0.0,-1);
   \draw[thick,<-] (0.6,0) --  (0.6,-1);
   \draw[thick,<-] (1.2,0) --  (1.2,-1);
   \draw[thick,<-] (2.4,0) --  (2.4,-1);
   \draw[thick,<-] (3.0,0) --  (3,-1);
\end{tikzpicture}}
\;\;\right]
\]
The second and third term are both equal to $h_{n-1}\otimes x_1$.  For the second diagram pull the cup to the bottom of the diagram producing a right twist curl (which is just another dot), then apply \eqref{eq:dotted-curl}.  Likewise, the third diagram simplifies using the trace relation.

To complete the proof of this claim we must slide the downward oriented line to the right of the upward oriented strands using the first equation in \eqref{eq:rel2}.  Crossing the downward oriented strand past the $i$th upward oriented strand produces the following sum.
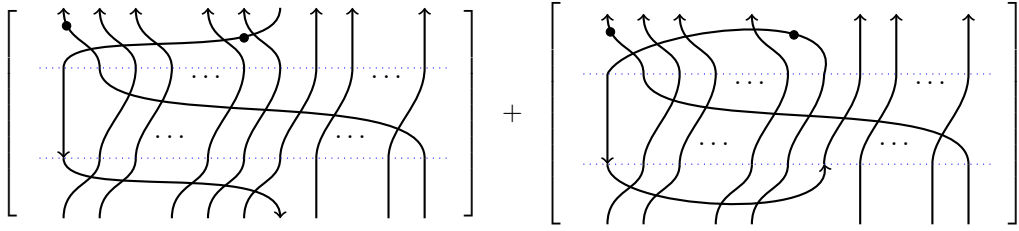
\[
 \left[ \;\;\;
\hackcenter{
\begin{tikzpicture}[scale=0.8]
  \draw[thick] (5.4,0) .. controls ++(0,1.25) and ++(0,-1.1) .. (0,1.5);
  \draw[thick] (0.0,0) .. controls ++(0,.5) and ++(0,-.6) .. (0.6,1.5);
  \draw[thick] (0.6,0) .. controls ++(0,.5) and ++(0,-.6) .. (1.2,1.5);
    \node at (1.2,.35) {$\dots$};
  \node at (1.8,1.35) {$\dots$};
  \draw[thick] (1.8,0) .. controls ++(0,.5) and ++(0,-.6) .. (2.4,1.5);
  \draw[thick] (2.4,0) .. controls ++(0,.5) and ++(0,-.6) .. (3,1.5);
  \draw[thick] (3,0) .. controls ++(0,.5) and ++(0,-.6) .. (3.6,1.5);
  \draw[thick] (3.6,0) .. controls ++(0,.5) and ++(0,-.6) .. (4.2,1.5);
  \node at (4.2,.35) {$\dots$};
  \node at (4.8,1.35) {$\dots$};
  \draw[thick] (4.8,0) .. controls ++(0,.5) and ++(0,-.6) .. (5.4,1.5);
  \draw[thick,<-] (-0.6,0) -- (-0.6,1.5);
  \filldraw  (-0.55,2.2) circle (2pt);
    \filldraw  (2.4,2.0) circle (2pt);
    %%
  %% BLUE LINES
    \draw[blue, dotted] (-1,0) -- (5.8,0);
    \draw[blue, dotted] (-1,1.5) -- (5.8,1.5);
  %% TOP LINES
   \draw[thick,->] (0.0,1.5) .. controls ++(0,.45) and ++(0,-.6) .. (-0.6,2.5);
   \draw[thick] (-0.6,1.5) .. controls ++(0,.75) and ++(0.0,-1) .. (3.0,2.5);
   \draw[thick,->] (0.6,1.5).. controls ++(0,.45) and ++(0,-.6) .. (-0.0,2.5);
   \draw[thick,->] (1.2,1.5) .. controls ++(0,.45) and ++(0,-.6) .. (0.6,2.5);
   \draw[thick,->] (2.4,1.5) .. controls ++(0,.45) and ++(0,-.6) .. (1.8,2.5);
   \draw[thick,->] (3,1.5) .. controls ++(0,.45) and ++(0,-.6) .. (2.4,2.5);
   \draw[thick,->] (3.6,1.5) -- (3.6,2.5);
   \draw[thick,->] (4.2,1.5) -- (4.2,2.5);
   \draw[thick,->] (5.4,1.5) -- (5.4,2.5);
   %% bottom LINES
   \draw[thick] (0.0,0) .. controls ++(0,-.45) and ++(0,.6) .. (-0.6,-1);
   \draw[thick,->]   (-0.6,0) .. controls ++(0,-.75) and ++(0.0,1) .. (3.0,-1);
   \draw[thick] (0.6,0) .. controls ++(0,-.45) and ++(0,.6) .. (-0.0,-1);
   \draw[thick] (1.8,0) .. controls ++(0,-.45) and ++(0,.6) .. (1.2,-1);
   \draw[thick] (2.4,0) .. controls ++(0,-.45) and ++(0,.6) .. (1.8,-1);
   \draw[thick] (3,0) .. controls ++(0,-.45) and ++(0,.6) .. (2.4,-1);
   \draw[thick] (3.6,0) -- (3.6,-1);
   \draw[thick] (4.8,0) -- (4.8,-1);
   \draw[thick] (5.4,0) -- (5.4,-1);
\end{tikzpicture}}
\;\;\right]
\;\; + \;\;
 \left[ \;\;\;
\hackcenter{
\begin{tikzpicture}[scale=0.8]
  \draw[thick] (5.4,0) .. controls ++(0,1.25) and ++(0,-1.1) .. (0,1.5);
  \draw[thick] (0.0,0) .. controls ++(0,.5) and ++(0,-.6) .. (0.6,1.5);
  \draw[thick] (0.6,0) .. controls ++(0,.5) and ++(0,-.6) .. (1.2,1.5);
    \node at (1.2,.35) {$\dots$};
  \node at (1.8,1.35) {$\dots$};
  \draw[thick] (1.8,0) .. controls ++(0,.5) and ++(0,-.6) .. (2.4,1.5);
  \draw[thick] (2.4,0) .. controls ++(0,.5) and ++(0,-.6) .. (3,1.5);
  \draw[thick] (3,0) .. controls ++(0,.5) and ++(0,-.6) .. (3.6,1.5);
  \draw[thick] (3.6,0) .. controls ++(0,.5) and ++(0,-.6) .. (4.2,1.5);
  \node at (4.2,.35) {$\dots$};
  \node at (4.8,1.35) {$\dots$};
  \draw[thick] (4.8,0) .. controls ++(0,.5) and ++(0,-.6) .. (5.4,1.5);
  \draw[thick,<-] (-0.6,0) -- (-0.6,1.5);
  \filldraw  (-0.55,2.2) circle (2pt);
    \filldraw  (2.5 ,2.15) circle (2pt);
    %%
  %% BLUE LINES
    \draw[blue, dotted] (-1,0) -- (5.8,0);
    \draw[blue, dotted] (-1,1.5) -- (5.8,1.5);
  %% TOP LINES
   \draw[thick,->] (0.0,1.5) .. controls ++(0,.45) and ++(0,-.6) .. (-0.6,2.5);
   \draw[thick] (-0.6,1.5) .. controls ++(.4,.75) and ++(0.4,1.2) .. (3.0,1.5);
   \draw[thick,->] (0.6,1.5).. controls ++(0,.45) and ++(0,-.6) .. (-0.0,2.5);
   \draw[thick,->] (1.2,1.5) .. controls ++(0,.45) and ++(0,-.6) .. (0.6,2.5);
   \draw[thick,->] (2.4,1.5) .. controls ++(0,.45) and ++(0,-.6) .. (1.8,2.5);
   \draw[thick,->] (3.6,1.5) -- (3.6,2.5);
   \draw[thick,->] (4.2,1.5) -- (4.2,2.5);
   \draw[thick,->] (5.4,1.5) -- (5.4,2.5);
   %% bottom LINES
   \draw[thick] (0.0,0) .. controls ++(0,-.45) and ++(0,.6) .. (-0.6,-1);
   \draw[thick,->]   (-0.6,0) .. controls ++(0,-.75) and ++(0.2,-1) .. (3.0,0);
   \draw[thick] (0.6,0) .. controls ++(0,-.45) and ++(0,.6) .. (-0.0,-1);
   \draw[thick] (1.8,0) .. controls ++(0,-.45) and ++(0,.6) .. (1.2,-1);
   \draw[thick] (2.4,0) .. controls ++(0,-.45) and ++(0,.6) .. (1.8,-1);
   \draw[thick] (3.6,0) -- (3.6,-1);
   \draw[thick] (4.8,0) -- (4.8,-1);
   \draw[thick] (5.4,0) -- (5.4,-1);
\end{tikzpicture}}
\;\;\right]
\]
The second diagram contains a left twist curl with single interior dot.  Using using \eqref{eq:dotted-curl} with $a=1$ to simplify this dotted curl, the second term above is  equal to $h_{n-1}\otimes x_1$.  The crossing resolution term   resulting from sliding the dot upward in the first term produces a diagram containing a left twist curl, so that these terms all vanish.  Continuing this process of moving the downward strand to to the right using the first equation in \eqref{eq:rel2} and the dot slide equation we see that $(h_{-1}\otimes x_1)(h_n\otimes x_1)$ is equal to
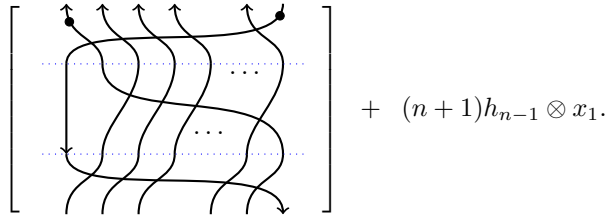
\[
 \left[ \;\;\;
\hackcenter{
\begin{tikzpicture}[scale=0.8]
  \draw[thick] (3.0,0) .. controls ++(0,1.25) and ++(0,-1.1) .. (0,1.5);
  \draw[thick] (0.0,0) .. controls ++(0,.5) and ++(0,-.6) .. (0.6,1.5);
  \draw[thick] (0.6,0) .. controls ++(0,.5) and ++(0,-.6) .. (1.2,1.5);
  \draw[thick] (1.2,0) .. controls ++(0,.5) and ++(0,-.6) .. (1.8,1.5);
  \draw[thick] (2.4,0) .. controls ++(0,.5) and ++(0,-.6) .. (3,1.5);
  \node at (1.8,.35) {$\dots$};
  \node at (2.4,1.35) {$\dots$};
  \draw[thick,<-] (-0.6,0) -- (-0.6,1.5);
  \filldraw  (-0.55,2.2) circle (2pt);
   \filldraw  (2.95,2.3) circle (2pt);
    %%
  %% BLUE LINES
    \draw[blue, dotted] (-1,0) -- (3.4,0);
    \draw[blue, dotted] (-1,1.5) -- (3.4,1.5);
  %% TOP LINES
   \draw[thick,->] (0.0,1.5) .. controls ++(0,.45) and ++(0,-.6) .. (-0.6,2.5);
   \draw[thick] (-0.6,1.5) .. controls ++(0,.85) and ++(0,-1) .. (3.0,2.5);
   \draw[thick,->] (0.6,1.5).. controls ++(0,.45) and ++(0,-.6) .. (-0.0,2.5);
   \draw[thick,->] (1.2,1.5) .. controls ++(0,.45) and ++(0,-.6) .. (0.6,2.5);
   \draw[thick,->] (1.8,1.5) .. controls ++(0,.45) and ++(0,-.6) .. (1.2,2.5);
   \draw[thick,->] (3,1.5) .. controls ++(0,.45) and ++(0,-.6) .. (2.4,2.5);
   %% bottom LINES
   \draw[thick] (0.0,0) .. controls ++(0,-.45) and ++(0,.6) .. (-0.6,-1);
   \draw[thick,->]   (-0.6,0) .. controls ++(0,-.85) and ++(0,1) .. (3.0,-1);
   \draw[thick] (0.6,0) .. controls ++(0,-.45) and ++(0,.6) .. (-0.0,-1);
   \draw[thick] (1.2,0) .. controls ++(0,-.45) and ++(0,.6) .. (0.6,-1);
   \draw[thick] (2.4,0) .. controls ++(0,-.45) and ++(0,.6) .. (1.8,-1);
   \draw[thick] (3,0)   .. controls ++(0,-.45) and ++(0,.6) .. (2.4,-1);
\end{tikzpicture}}
\;\;\right]
\;\; + \;\;
(n+1) h_{n-1} \otimes x_1.
\]
completing the proof of $m=1$ case.

Now assume the lemma is true for $[\beta_{-m+1}, \beta_n]$.
The Jacobi identity gives
\begin{equation}
\label{-m1n1jac1}
0= [[\beta_{-1}, \beta_{-m+1}],\beta_n]+[[\beta_{-m+1},\beta_n],\beta_{-1}]+[[\beta_n, \beta_{-1}],\beta_{-m+1}].
\end{equation}
By the base case and the inductive step ~\eqref{-m1n1jac1} is:

\begin{equation*}
0=[[\beta_{-1}, \beta_{-m+1}],\beta_n]
+[(n+m-1)\beta_{n-m+1}+\sum_{j=1}^{m-2} j \alpha_{n-j} \alpha_{-m+1+j},\beta_{-1}]
+(n+1)[\beta_{-m+1}, \beta_{n-1}].
\end{equation*}

Another application of the base case and the inductive steps gives:

\begin{align}
\label{-m1n1eq2}
0= [[\beta_{-1}, \beta_{-m+1}],\beta_n]
&+(n+m)(m-2) \beta_{n-m}
+ \sum_{j=1}^{m-2} j (\alpha_{n-j} \alpha_{-m+1+j} \beta_{-1}-\beta_{-1} \alpha_{n-j} \alpha_{-m+1+j})  \\
%+(n+1)(n+m-2) \beta_{n-m}
&+(n+1) \sum_{j=1}^{m-2} j \alpha_{n-1-j} \alpha_{-m+1+j} \nonumber
\end{align}

By Lemma ~\ref{-n1m1}, we move $\beta_{-1} $ and  ~\eqref{-m1n1eq2} becomes
\begin{align}
\label{-m1n1eq3}
0 = [[\beta_{-1}, \beta_{-m+1}],\beta_n]
& +(m+n)(m-2) \beta_{n-m}
+ \sum_{j=1}^{m-2} j \alpha_{n-j} \alpha_{-m+1+j} \beta_{-1}
-\sum_{j=1}^{m-2} j \alpha_{n-j} \beta_{-1} \alpha_{-m+1+j} \\
&-\sum_{j=1}^{m-2} j(n-j) \alpha_{n-j-1} \alpha_{-m+1+j}
+ \sum_{j=1}^{m-2} (n+1)j \alpha_{n-1-j} \alpha_{-m+1+j}. \nonumber
\end{align}
Again moving the element $\beta_{-1} $ to the right ~\ref{-m1n1eq3} becomes using Lemma ~\ref{-n1m1}
\begin{align*}
0=[[\beta_{-1}, \beta_{-m+1}],\beta_n]
&+(m+n)(m-2) \beta_{n-m}
+ \sum_{j=1}^{m-2} (n+1)j \alpha_{n-1-j} \alpha_{-m+1+j} \\
&+ \sum_{j=1}^{m-2} j \alpha_{n-j} \alpha_{-m+1+j} \beta_{-1}
-\sum_{j=1}^{m-2} j \alpha_{n-j}  \alpha_{-m+1+j} \beta_{-1} \\
&-\sum_{j=1}^{m-2} j(-m+1+j) \alpha_{n-j} \alpha_{-m+j}
-\sum_{j=1}^{m-2} j(n-j) \alpha_{n-j-1} \alpha_{-m+1+j}.
\end{align*}
Simplifying the above gives
\begin{equation*}
0= [[\beta_{-1}, \beta_{-m+1}],\beta_n]
+(m+n)(m-2) \beta_{n-m}
+\sum_{j=1}^{m-2} j(j+1) \alpha_{n-1-j} \alpha_{-m+1+j}
-\sum_{j=1}^{m-2} j(-m+1+j) \alpha_{n-j} \alpha_{-m+j}.
\end{equation*}
Combining the last two sums above gives
\begin{align*}
0= [[\beta_{-1}, \beta_{-m+1}],\beta_n]
&+(m+n)(m-2) \beta_{n-m}
+\sum_{j=1}^{m-3} (m-2)(j+1) \alpha_{n-1-j} \alpha_{-m+1+j} \\
&+(m-2) \alpha_{n-1} \alpha_{-m+1}
+(m-2)(m-1) \alpha_{n-m+1} \alpha_{-1}.
\end{align*}
Thus
\begin{equation*}
[[\beta_{-1}, \beta_{-m+1}],\beta_n]=
-(m-2)((m+n) \beta_{n-m} - \sum_{j=1}^{m-1} j \alpha_{n-j} \alpha_{-m+j}).
\end{equation*}
It follows from Lemma ~\ref{lemmam1n1} that
$ [\beta_{-1},\beta_{-m+1}]=(-m+2) \beta_{-m}$.
Substituting this into the above equation gives the desired result.
\end{proof}

\begin{remark}
A slight modification to the $m=1$ argument of Lemma~\ref{lemma-m1n1} can be used to prove that for $n,m>1$ the following relations:
\begin{align}
[h_n \otimes x_1^p, h_1 \otimes x_1^m]&=-\sum_{j=2}^n \sum_{a+b=m-1} h_{n+1} \otimes x_1^p x_j^a x_{j+1}^b
+ \sum_{a+b=p-1} h_{n+1} \otimes x_1^a x_2^{b+m} - \sum_{c+d=m-1} h_{n+1} \otimes x_1^c x_2^{p+d}, \nn
\\
 % \label{n1m2}
[h_n \otimes x_1, h_m \otimes x_1^2] &=(m-2n)(h_{n+m} \otimes x_1^2)
+ \sum_{j=1}^n (2n-j)(h_j \otimes x_1)(h_{n+m-j} \otimes 1)
- \sum_{j=1}^{n-1} j(h_{m+j} \otimes x_1)(h_{n-j} \otimes 1),
\nn
\end{align}
hold in $\Tr(\H)$. Though we will not need these relations in $ \Tr(\H)$ in order to identify it with a quotient of $\W$, they can be helpful in computing explicit formulas for commutators.
\end{remark}

%Define elements $h_{l,k}$ for $l>0$ as follows:
%\begin{align}
%\label{hlkdef}
%h_{\pm l,0} &= h_{\pm l} \otimes 1 \\ \nonumber
%h_{0,0} &= 0 \\ \nonumber
%h_{l,1} & = \frac{-1}{l}(h_l \otimes (x_1+\cdots+x_l))+\frac{l-1}{2}(h_l \otimes 1) \\  \nonumber
%h_{-l,1} & = \frac{-1}{l}(h_{-l} \otimes (x_1+\cdots+x_l))-\frac{l+1}{2}(h_{-l} \otimes 1) \\  \nonumber
%h_{0,1} &=-c_0 \\ \nonumber
%h_{0,2} &=c_0+c_1 \nonumber
%\end{align}
We define:
\begin{align}
\label{heivirformulas}
L_l^{} &= |\frac{1}{l}|(h_{-l} \otimes (x_1 + \cdots + x_{|l|}) \hspace{.2in} (l \neq 0) \\ \nonumber
L_0 &= c_0 \\ \nonumber
b_l &= h_{-l} \otimes 1. \nonumber
\end{align}

\begin{proposition}
\label{heisvirasoroHH0}
The elements $ L_l^{} $ for $ l \in \mathbb{Z} $ and $b_l$ for $ l \neq 0 $  generate a Heisenberg-Virasoro algebra with central charge one.  That is
\begin{align*}
[b_k, b_l]&=k \delta_{k,-l} \\
[L_k^{}, L_l^{}]&=(k-l) L_{k+l}^{} + \frac{k^3-k}{12} \delta_{k,-l} \\
[L_l^{}, b_k]&=-k(b_{l+k}).
\end{align*}
\end{proposition}

\begin{proof}
This follows from Lemmas ~\ref{origheisrelations}, ~\ref{lemmam1n1},  ~\ref{lemman1m0},  ~\ref{-n1m1}, ~\ref{n1-m1} and ~\ref{lemma-m1n1}.
\end{proof}

\subsection{Trace of the degenerate affine Hecke algebra}
\label{HHDAHA}

The trace of $DH_n$ may be computed by a theorem of Solleveld ~\cite{S} which reduces the problem to computing the trace of a semi-direct product via a spectral sequence argument.

\begin{theorem}\cite[Theorem 3.4]{S}
\label{HHDH}
The trace of the degenerate affine Hecke algebra is isomorphic to the trace of the semi-direct product of the symmetric group and a polynomial algebra:
\begin{equation*}
\Tr(DH_n) \cong \Tr(S_n  \ltimes \mathbb{C}[x_1, \ldots, x_n]).
\end{equation*}
\end{theorem}

Now let $P(n)$ be the set of partitions of $n$ and, for a partition $ \lambda $ of $n$, let $p_i(\lambda)$ be the number of times the number $i$ occurs in $\lambda$.

\begin{theorem}\cite[Theorem 3.1]{EO}, \cite[Section 1]{S}
\label{HHSD}
We have
\begin{equation*}
\Tr(S_n \ltimes \mathbb{C}[x_1, \ldots, x_n]) \cong
\bigoplus_{\lambda \in P(n)} \bigotimes_{i \geq 1} S^{p_i(\lambda)}  \mathbb{C}[x]
\end{equation*}
where $ S^k \mathbb{C}[x]$ is the space of $ S_k $ invariants of $\mathbb{C}[x_1, \ldots, x_k]$.
\end{theorem}

Theorems ~\ref{HHDH} and ~\ref{HHSD} together determine the trace of the degenerate affine Hecke algebra.

\begin{example}
The degenerate affine Hecke algebra of rank two $ DH_2$ is generated by the group algebra of the symmetric group $S_2$ and the polynomial algebra $\mathbb{C}[x_1,x_2]$.  If $s_1$ is the generator of $S_2$ then the only additional relation is $s_1 x_1= x_2 s_1 +1$.

There are partitions $ \lambda=(2)$ and $\lambda=(1,1) $ of $2$.
By the above theorems $\Tr(DH_2)$ is a direct sum of subspaces
\begin{equation*}
\bigotimes_{i \geq 1} S^{p_i((2))} \mathbb{C}[x] \bigoplus
\bigotimes_{i \geq 1} S^{p_i((1,1))} \mathbb{C}[x]
\cong
\mathbb{C}[x] \bigoplus S^2 \mathbb{C}[x_1,x_2].
\end{equation*}
This corresponds to a splitting of the trace
\begin{equation*}
\mathbb{C}\langle x_1^a s_1 \rangle_{a \in \mathbb{Z}_{\geq 0}} \bigoplus
S^2 \mathbb{C}[x_1,x_2].
\end{equation*}
\end{example}

\subsection{$\Tr(\H)$ as a vector space}
\label{HHVS}

\begin{lemma}\label{lem:invert}
If $f,g\in DH_n$ with $fg = 1$, then in fact $f,g\in \mathbb{C}[S_n]\subset DH_n$.
\end{lemma}
\begin{proof}
We consider the non-negative integral filtration on $DH_n$, with associated graded $\mbox{gr}(DH_n) \cong \mathbb{C}[S_n]\rtimes \mathbb{C}[x_1,\dots,x_n]$.  The degree 0 part of this filtration is precisely $\mathbb{C}[S_n]$. Now $\mbox{gr}(f)\mbox{gr}(g) = 1$, which implies that $\mbox{gr}(f)$ and $\mbox{gr}(g)$ are in $\mathbb{C}[S_n]$. Thus $f$ and $g$ are in the degree 0 part of the filtration, as desired.
\end{proof}

\begin{lemma}
The indecomposable objects of $\H'$ are of the form $\P^m\Q^n$ for $m,n \in \mathbb{Z}_{\geq 0}$.
\end{lemma}
\begin{proof}
An indecomposable object must be of the form $\P^m\Q^n$ because a subsequence $\Q\P$ produces a decomposition
$\P\Q \oplus \Id$.

There is a morphism from $\P^m\Q^n$ to $\P^a\Q^b$ if and only if $m-n=a-b$.
If $a+b \neq m+n$ then the composition $\P^a\Q^b \rightarrow \P^m\Q^n \rightarrow \P^a\Q^b$ produces cups and caps or circles.
By Proposition ~\ref{KhovEndThmPmQn} it follows that this composition cannot be the identity.

Now suppose there are maps
$ f \colon \P^m \Q^n \rightarrow \P^m \Q^n $
and
$ g \colon \P^m \Q^n \rightarrow \P^m \Q^n $ such that $ gf$ is the identity.
We claim that in fact $gf$ is the identity, too.
%By the argument about
Proposition ~\ref{KhovEndThmPmQn} implies $f$ is a monoidal composition of $ f_1 $ and $ f_2 $ where
$ f_1 \colon \P^m \rightarrow \P^m $ and $ f_2 \colon \Q^n \rightarrow \Q^n $.
Therefore, $ f_1 $ can be identified with an element in $DH_m$ and $ f_2$ may be identified with an element in $DH_n^{op}$.
Similarly, $g$ is a monoidal composition of $g_1$ and $g_2$ where
$ g_1 \colon \P^m \rightarrow \P^m $ and $ g_2 \colon \Q^n \rightarrow \Q^n $, so that
 $g_1$ can be identified with an element in $DH_m$ and $g_2$ may be identified with an element in
$DH_n^{op}$.
But now, by Lemma \ref{lem:invert}, $f_1$ and $g_1$ must belong to $\mathbb{C}[S_m]$.  Since $\mathbb{C}[S_m]$ is semi-simple,
$ g_1 f_1 $ is the identity if and only $f_1 g_1$ is the identity.
Thus $f$ is an isomorphism, and $\P^m\Q^n$ is indecomposable.
\end{proof}

\begin{lemma}
\label{HHVSLEMMA}
There is an isomorphism
\begin{equation*}
\Tr(\H) \cong \left( \bigoplus_{\substack{m \in \mathbb{Z}_{\geq 0} \\ n \in \mathbb{Z}_{\geq 0}}}
\bigoplus_{\substack{\mu \in P(m) \\ \lambda \in P(n)}}
\bigotimes_{\substack{i \geq 1 \\ j \geq 1}} S^{p_i(\mu)}  \mathbb{C}[x] \otimes S^{p_j(\lambda)}  \mathbb{C}[x] \right)
\otimes \mathbb{C}[c_0, c_1, \ldots].
\end{equation*}
\end{lemma}

\begin{proof}
By Proposition ~\ref{HH0HequalsHH0Hprime}, $ \Tr(\H) \cong \Tr(\H')$ so by the definition of $ \Tr(\H') $ and by Proposition~\ref{prop:indecomposables} it suffices to consider endomorphisms of all indecomposable objects ($m$ products of $ \P$ followed by $n$ products of $\Q$) modulo the ideal $\mathcal{I}$ defined in Section \eqref{HHconventions}.
That is
\begin{equation*}
\Tr(\H') \cong \bigoplus_{m,n \geq 0} \End(\P^m\Q^n) / \mathcal{I}.
\end{equation*}
It is clear that this ideal is equal to the ideal $J_{m,n}$ from Proposition ~\ref{KhovEndThmPmQn} plus the ideal generated by $fg-gf$ where $f,g \colon \P^mQ^n \rightarrow \P^m\Q^n$.
Using the trace relation together with the relations in $\H$, any map contained in the ideal $J_{m,n}$ can be reduced to a sum of endomorphisms of $\P^{m'} \Q^{n'}$ for some $m'$, $n'$ that are not in the ideal $J_{m,n}$.
By applying Proposition ~\ref{KhovEndThmPmQn}
\begin{equation*}
\Tr(\H) \cong \bigoplus_{m,n \geq 0} \Tr(DH_m \otimes DH_n^{op} \otimes \mathbb{C}[c_0,c_1,\ldots]).
\end{equation*}
The Lemma now follows from the results in Section ~\ref{HHDAHA}.
\end{proof}

\subsection{Towards the trace of $\H$ as an algebra}
The trace of a general category is only a vector space  The monoidal structure of $\mathcal{H}$ endows $\Tr(\H)$ with a product.  Given morphisms $ f \colon X \rightarrow X' $ and $ g \colon Y \rightarrow Y'$ then $ fg \colon XY \rightarrow X'Y'$ and we may define $[f][g]:=[fg]$.

Denote by $\Tr^{>}(\H)$ (resp. $ \Tr^{<}(\H) $) the subalgebra of $ \Tr(\H) $ generated by $ h_l \otimes x_1^k $ (resp. $ h_{-l} \otimes x_1^k $) for $ l \geq 1, k \geq 0 $. Moreover, let $\Tr^0(\H) := \Tr(\H^{0})$.

\begin{lemma}
\label{HH0graded}
The algebra $ \Tr(\H) $ is $\Z$-graded where $h_r \otimes x_1^k $ is in degree $r$.
\end{lemma}

\begin{proof}
This is clear by an inspection of the defining relations of the category $\H'$.
\end{proof}
We call this $ \Z$-grading the rank grading because it is related to the degenerate affine Hecke algebra associated to a Lie group of a particular rank.

\begin{lemma}
\label{HH0filtered}
The algebra $\Tr(\H)$ is $\Z_{\geq 0}$-filtered where $ h_r \otimes x_1^k $ is in degree $ k$.
\end{lemma}
\begin{proof}
This follows from the fact that a dot may slide through a crossing plus or minus a correction term which is a resolution of the crossing containing no dots.
\end{proof}

For $\omega \in \{<,>,0\}$, the associated graded of $ \Tr^{\omega}(\H) $ with respect to the $\Z_{\geq 0}$-filtration will be denoted by $gr \Tr^{\omega}(\H) $. This is a $ (\Z \times \Z_{\geq 0})$-graded algebra. We denote the subspace in degree $(r,k)$ by $ gr \Tr^{\omega}(\H)[r, k] $ and define the Poincar\'e series by
\begin{equation*}
P_{\Tr^{\omega}(\H)}(t,q)=\sum_{r \in \Z} \sum_{k \in \Z_{\geq 0}} \dim gr \Tr^{\omega}(\H)[r, k] t^r q^k.
\end{equation*}

\begin{corollary}
\label{poincareH}
These Poincar\'e series are given by
\begin{equation*}
P_{\Tr^>(\H)}=\prod_{r>0} \prod_{k \geq 0} \frac{1}{1-t^r q^k} \hspace{.3in}
P_{\Tr^<(\H)}=\prod_{r<0} \prod_{k \geq 0} \frac{1}{1-t^r q^k}.
\end{equation*}
\end{corollary}
\begin{proof}
We prove the first equality (the second follows similarly). Under the isomorphism from Lemma ~\ref{HHVSLEMMA} $\Tr^>(\H)$ corresponds to
\begin{equation}\label{eq:1}
\bigoplus_{n \ge 0, \lambda \in P(n)} \bigotimes_{i \geq 1} S^{p_i(\lambda)}  \mathbb{C}[x].
\end{equation}
Now $S^{p_i(\lambda)} \mathbb{C}[x] \cong \mathbb{C}[h_1, \dots, h_{p_i(\lambda)}]$ where the bi-degree of the symmetric functions $h_i$ is $(|\l|,i)$. It follows that the generating series of (\ref{eq:1}) is given by $\prod_{j \ge 1} \phi(q^j,t)$ where
$$\phi(q,t) := 1 + \frac{q}{(1-t)} + \frac{q^2}{(1-t)(1-t^2)} + \frac{q^3}{(1-t)(1-t^2)(1-t^3)} + \dots$$
But $\phi(q,t)$ is just the basic hypergeometric series $_1 \phi_0(0;t,q)$ which is known to equal $\prod_{j \ge 0} \frac{1}{1-qt^j}$. The result follows.
\end{proof}

\begin{lemma} \label{eq:lemma-hnc0}
For $ n>0$ we have
\label{c0actingonhn}
\begin{align*}
[h_n \otimes x_1^a,c_0] &=-n(h_{n} \otimes x_1^a) \\
[h_{-n} \otimes x_1^a, c_0]&=n(h_{-n} \otimes x_1^a).
\end{align*}
\end{lemma}

\begin{proof}
This is an easy graphical calculation using computations from ~\cite[Section 2]{K}.
\end{proof}

\begin{remark}
Lemma~\ref{eq:lemma-hnc0} can be generalized to prove the more general identities
\begin{align}
  [h_n \otimes x_1^a,c_0^r] &=\sum_{j=1}^r (-1)^j \binom{r}{j} n^j c_0^{r-j} (h_n \otimes x_1^a), \text{ and}
\nn \\
[h_{-n} \otimes x_1^a,c_0^r] &=\sum_{j=1}^r \binom{r}{j} n^j c_0^{r-j} (h_{-n} \otimes x_1^a).
\nn
\end{align}
\end{remark}

\begin{lemma}
For $n>0$ there are equalities
\label{c1actingonhn}
\begin{align*}
[h_n \otimes x_1^a, c_1]&=-2n(h_n \otimes x_1^{a+1}) + \sum_{j=1}^{n-1} 2(n-j) (h_j \otimes x_1^a)(h_{n-j} \otimes 1) \\
[h_{-n} \otimes x_1^a, c_1]&=2n(h_{-n} \otimes x_1^{a+1}) + \sum_{j=1}^{n-1} 2(n-j) (h_{-j} \otimes x_1^a)(h_{-n+j} \otimes 1).
\end{align*}
\end{lemma}

\begin{proof}
This too is an easy graphical calculation using computations from ~\cite[Section 2]{K}.
\end{proof}

\begin{lemma}
For non-negative integers $a$ and $b$ there is an equality
\label{lemma-1b1a}
\begin{equation*}
[h_{-1} \otimes x_1^b, h_1 \otimes x_1^a]=\tilde{c}_{a+b}+\sum_{l=0}^{a+b-2} (a+b-1-l) \tilde{c}_l c_{a+b-2-l}.
\end{equation*}
Here we define $c_{-2}=-1$ and $c_{-n}=0$ for $ n \in \mathbb{N}$ where $ n \neq 2$.
\end{lemma}

\begin{proof}
The element $ (h_1 \otimes x_1^a)(h_{-1} \otimes x_1^b) $ is equal to
\[
 \left[ \;
\hackcenter{
\begin{tikzpicture}[scale=0.8]
 \draw[thick,->] (-0.6,0) -- (-0.6,2);
  \draw[thick,<-] (0.0,0) -- (-0.0,2);
  \filldraw  (-0.6,1.2) circle (2pt);
  \filldraw   (0,1.2) circle (2pt);
  \node at (-.85,1.3) {$\scs a$};
  \node at (0.25,1.3) {$\scs b$};
\end{tikzpicture}}
  \;\right]
  \;\; = \;\;
  \left[ \; \;
\hackcenter{
\begin{tikzpicture}[scale=0.8]
 \draw[thick,->] (-0.6,0) .. controls ++(0,.4) and ++(0,-.5) .. (-0.0,1);
  \draw[thick,<-] (0.0,0) .. controls ++(0,.4) and ++(0,-.5) .. (-0.6,1);
   \draw[thick,<-] (-0.6,1) .. controls ++(0,.4) and ++(0,-.5) .. (-0.0,2);
  \draw[thick,->] (0.0,1) .. controls ++(0,.4) and ++(0,-.5) .. (-0.6,2);
  \filldraw  (-0.6,1.7) circle (2pt);
  \filldraw   (0,1.7) circle (2pt);
  \node at (-.85,1.8) {$\scs a$};
  \node at (0.25,1.8) {$\scs b$};
\end{tikzpicture}}
  \;\right]
 \;\; \refequal{\eqref{eq:inddot}} \;\;
  \left[ \;\;
\hackcenter{
\begin{tikzpicture}[scale=0.8]
 \draw[thick,->] (-0.6,0) .. controls ++(0,.4) and ++(0,-.5) .. (-0.0,1);
  \draw[thick,<-] (0.0,0) .. controls ++(0,.4) and ++(0,-.5) .. (-0.6,1);
   \draw[thick,<-] (-0.6,1) .. controls ++(0,.4) and ++(0,-.5) .. (-0.0,2);
  \draw[thick,->] (0.0,1) .. controls ++(0,.4) and ++(0,-.5) .. (-0.6,2);
  \filldraw  (0,1.2) circle (2pt);
  \filldraw   (0,1.7) circle (2pt);
  \node at (0.2,1.3) {$\scs a$};
  \node at (0.25,1.8) {$\scs b$};
\end{tikzpicture}}
  \;\right]
\;\; - \;\;
\sum_{j=0}^{a-1}
  \left[ \;\;
\hackcenter{
\begin{tikzpicture}[scale=0.8]
 \draw[thick,->] (-0.6,0) .. controls ++(0,.4) and ++(0,-.5) .. (-0.0,1);
  \draw[thick,<-] (0.0,0) .. controls ++(0,.4) and ++(0,-.5) .. (-0.6,1);
   \draw[thick,<-] (-0.6,1) .. controls ++(0,.4) and ++(0, .4) .. (-0.0,1);
  \draw[thick,->] (0.0,2) .. controls ++(0,-.5) and ++(0,-.5) .. (-0.6,2);
  \filldraw  (0,1.2) circle (2pt);
  \filldraw   (-0.1,1.7) circle (2pt);
  \node at (0.2,1.3) {$\scs j$};
  \node at (0.8,1.8) {$\scs a+b-1-j$};
\end{tikzpicture}}
  \;\right]
\]
\[
\;\; = \;\;
  \left[ \;\;
\hackcenter{
\begin{tikzpicture}[scale=0.8]
 \draw[thick,->] (-0.6,0) .. controls ++(0,.4) and ++(0,-.5) .. (-0.0,1);
  \draw[thick,<-] (0.0,0) .. controls ++(0,.4) and ++(0,-.5) .. (-0.6,1);
   \draw[thick ] (-0.6,1) .. controls ++(0,.4) and ++(0,-.5) .. (-0.0,2);
  \draw[thick,->] (0.0,1) .. controls ++(0,.4) and ++(0,-.5) .. (-0.6,2);
  \filldraw  (0,1.2) circle (2pt);
  \filldraw   (-0.6,1.2) circle (2pt);
  \node at (0.2,1.3) {$\scs a$};
  \node at (-0.85,1.3) {$\scs b$};
\end{tikzpicture}}
  \;\right]
\;\; - \;\;
\sum_{j=0}^{b-1}
  \left[ \;\;
\hackcenter{
\begin{tikzpicture}[scale=0.8]
 \draw[thick,->] (-0.6,0) .. controls ++(0,.4) and ++(0,-.5) .. (-0.0,1);
  \draw[thick,<-] (0.0,0) .. controls ++(0,.4) and ++(0,-.5) .. (-0.6,1);
   \draw[thick,<-] (-0.6,1) .. controls ++(0,.4) and ++(0, .4) .. (-0.0,1);
  \draw[thick,->] (0.0,2) .. controls ++(0,-.5) and ++(0,-.5) .. (-0.6,2);
  \filldraw  (0,1.2) circle (2pt);
  \filldraw   (-0.1,1.7) circle (2pt);
  \node at (0.45,1.3) {$\scs a+j$};
  \node at (0.8,1.7) {$\scs  b-1-j$};
\end{tikzpicture}}
  \;\right]
\;\; - \;\;
\sum_{j=0}^{a-1}
  \left[ \;\;
\hackcenter{
\begin{tikzpicture}[scale=0.8]
 \draw[thick,->] (-0.6,0) .. controls ++(0,.4) and ++(0,-.5) .. (-0.0,1);
  \draw[thick,<-] (0.0,0) .. controls ++(0,.4) and ++(0,-.5) .. (-0.6,1);
   \draw[thick,<-] (-0.6,1) .. controls ++(0,.4) and ++(0, .4) .. (-0.0,1);
  \draw[thick,->] (0.0,0) .. controls ++(0,-.5) and ++(0,-.5) .. (-0.6,0);
  \filldraw  (0,1.2) circle (2pt);
  \filldraw   (-0.1,-.3) circle (2pt);
  \node at (0.2,1.3) {$\scs j$};
  \node at (0.8,-.2) {$\scs a+b-1-j$};
\end{tikzpicture}}
  \;\right]
\]
which can be simplified further, again using equation \eqref{eq:inddot}.  Thus
\begin{equation*}
(h_1 \otimes x_1^a)(h_{-1} \otimes x_1^b) =
(h_1 \otimes x_1^a)(h_{-1} \otimes x_1^b)
-\tilde{c}_{a+b}
-\sum_{j=0}^{b-1} \sum_{k=0}^{a+j-1} \tilde{c}_k c_{a+b-2-k}
-\sum_{j=0}^{a-1} \sum_{k=0}^{j-1} \tilde{c}_k c_{a+b-2-k}.
\end{equation*}
Combining the last two double summations into a single summation gives the lemma.
\end{proof}

\begin{lemma}
$\Tr(\H)$ is generated by $h_{-1} \otimes 1, h_1 \otimes 1, (c_0+c_1) $.
%\item $\Tr^>(\H)$ is generated by $h_1 \otimes x_1^l $ for $ l \geq 0$.
%\item $\Tr^<(\H) $ is generated by $h_{-1} \otimes x_1^l$ for $ l \geq 0$.
%\end{enumerate}
\end{lemma}

\begin{proof}
By Lemmas ~\ref{c0actingonhn} and ~\ref{c1actingonhn} we have
\begin{equation*}
ad(c_0+c_1)(h_1 \otimes 1) = 2(h_1 \otimes x_1) + h_1 \otimes 1
\end{equation*}
allowing us to generate $h_1 \otimes x_1$.

By Proposition ~\ref{posvirposheis} we have $ad^n(h_1 \otimes x_1)(h_1 \otimes 1)=n!(h_n \otimes 1)$ allowing us to also generate $h_n \otimes 1$. Next we could generate $h_n \otimes x_1^k$ for all $n \geq 1, k \geq 0$ inductively by using Lemmas ~\ref{c0actingonhn} and ~\ref{c1actingonhn} to calculate $[h_n \otimes x_1^{k-1}, c_0 + c_1]$.

In a similar manner using Lemmas ~\ref{c0actingonhn} and ~\ref{c1actingonhn} we can generate $h_{-n} \otimes x_1^k$ for $ n \geq 1, k \geq 0$. In particular we can generate $h_{-1} \otimes x_1^b $ and $ h_1 \otimes x_1^a$. Finally we can generate $c_n $ for all $n \geq 0 $ using Lemma ~\ref{lemma-1b1a}.
\end{proof}

\subsection{A representation of the category $\H'$} \label{sec:fock}
Let $\mathcal{S}_n'$ denote the category whose objects are all $(\mathbb{C}[S_r], \mathbb{C}[S_n])$ bimodules
for some $ r \in \mathbb{Z}_{\geq 0}$ corresponding to induction and restriction functors.
The morphisms between two bimodules $B$ and $B'$ is the space of bimodule maps between $B$ and $B'$.
Note that if the bimodules $B$ and $B'$ are not bimodules over the same algebra then the space of morphisms is zero.

For each $n$ there is a functor $ \mathcal{F}_n' \colon \H' \rightarrow \mathcal{S}_n'$.
The object $P$ is mapped to the $(\mathbb{C}[S_{n+1}], \mathbb{C}[S_n])$-bimodule $ \mathbb{C}[S_{n+1}]$.
The object $Q$ is mapped to the $(\mathbb{C}[S_{n-1}], \mathbb{C}[S_n])$-bimodule $ \mathbb{C}[S_{n}]$
where we view $ \mathbb{C}[S_{n-1}] $ as embedded in $\mathbb{C}[S_{n}]$ by mapping the simple transpositions
$s_i$ to $s_i$ for $i=1,\ldots,n-2$.
We then extend this action monoidally.
For details on how the functor $\mathcal{F}_n'$ acts on morphisms see ~\cite[Section 3.3]{K}.
Let
\begin{equation*}
V = \bigoplus_n \Tr(\mathcal{S}_n') = \bigoplus_n \Tr(\mathbb{C}[S_n]-mod).
\end{equation*}
This is naturally a module for $ \Tr(\H) \cong \Tr(\H')$ since the category $ \mathcal{S}_n' $ is a representation of the category $ \H $.  For the analogous statement for $2$-categories see ~\cite[Section 6.0.3]{BGHL}.

\begin{proposition}
\label{formofV}
As a vector space $V$ is isomorphic to $ \mathbb{C}[h_{1} \otimes 1, h_{2} \otimes 1, \ldots]$.
\end{proposition}

\begin{proof}
As a vector space $\Tr(\mathbb{C}[S_n]) $ is spanned by conjugacy classes of $S_n$ which in turn is isomorphic to the span of irreducible characters of $S_n$.
Thus $V$ is spanned by the irreducible characters of all symmetric groups.  As an algebra under induction it's generated by irreducible representations corresponding to single cycles.  These are precisely the elements $h_1, h_2, \ldots $ in $ \Tr(\mathbb{C}[S_n])$.
\end{proof}

\begin{proposition}\label{prop:irrep}
The $\Tr(\H)$-module $ V $ is irreducible.
\end{proposition}

\begin{proof}
It was proved in ~\cite{K} that $V$ categorifies Fock space for the Heisenberg algebra $\mathfrak{h}$.
Since $\mathfrak{h}$ acts irreducibly on Fock space and $\mathfrak{h}$ is a subalgebra of $\Tr(\H)$,
it follows that $V$ must be an irreducible $\Tr(\H)$-module.
\end{proof}

In order to prove the next proposition we recall that the structure of $\Tr(\H)$ as a vector space was computed in section ~\ref{HHVS}.  We follow closely the proof of \cite[Proposition F.6]{SV} where the faithfullness of $\W/(C-1,w_{0,0})$ on $ \mathcal{V}_{1,0} $ is given.

\begin{proposition}\label{prop:faithful}
The action of $ \Tr(\H) $ on $V$ is faithful.
\end{proposition}

\begin{proof}
In what follows we denote by $\Tr(\H)[r, \le k]$ the set of elements in $\Z$-degree $r$ and $ \Z_{\geq 0}$-degree $\le k$. Let $I$ be the annihilitator of $ \Tr(\H)$ in $V$.  There is a filtration on $I$: $ I_0 \subset I_1 \subset \cdots $. Since $ \oplus_{r} \Tr(\H)[r, \leq 0] $ is isomorphic to a Heisenberg algebra and it is known that the Heisenberg algebra acts faithfully on $V$, it must be the case that
$$ I \cap (\oplus_{r} \Tr(\H)[r, \leq 0]) = \{ 0 \}.$$
Thus $I_0=\{ 0 \}$.  If we assume that $I$ is non-zero then the minimal $n$ such that $I_n \neq \{ 0 \}$ is greater than zero.

Notice that for $ n \geq 1 $ that
$$ [h_l \otimes 1, \Tr(\H)[r, \leq n]] \subset \Tr(\H)[r+l, \leq n-1].$$
It follows that for $ n \geq 1$ that
$ [I_n, \Tr(\H)[r, \leq 0]]=0$ so that each filtered component $I_n$ is contained in the centralizer the $ \oplus_r \Tr(\H)[r, \leq 0] $ in $ \Tr(\H)$.

Now we compute the centralizer of $ \oplus_r \Tr(\H)[r, \leq 0] $ in $ \Tr(\H)$ and show that it must be trivial.
If $ g \in \Tr(\H)$ then we let its image in the associated graded be denoted by $ \overline{g}$.
It follows from Lemma ~\ref{HHVSLEMMA} that the associated graded is isomorphic to a polynomial algebra $ \mathbb{C}[\overline{h_r \otimes x_1^s}]$
for $ r \in \mathbb{Z}, s \in \mathbb{Z}_{\geq 0} $ and $ (r,s) \neq (0,0)$.
This algebra is bigraded with element $ \overline{h_r \otimes x_1^s} $ in degree $(r,s)$.

Consider the map for $n \geq 1$
\begin{equation*}
\zeta_n=[h_l \otimes 1, \bullet] \colon \bigoplus_{r} \Tr(\H)[r, \leq n]/\Tr(\H)[r, \leq n-1] \rightarrow
\bigoplus_{r} \Tr(\H)[r, \leq n-1]/\Tr(\H)[r, \leq n-2].
\end{equation*}

The space $ \Tr(\H)[r, \leq n]/\Tr(\H)[r, \leq n-1] $ is the degree $(r,n)$ piece of the associated graded algebra.
It is easy to see that
$$ \zeta_l(\overline{h_r \otimes x_1^k})=-kl(\overline{h_{l+r} \otimes x_1^{k-1}}).$$
Thus the common kernel of the operators $ \cap_l \ker(\zeta_l) $ is zero if $ n \geq 1$. This implies that the centralizer of $ \oplus_r \Tr(\H)[r, \leq 0] $ in $ \Tr(\H)$ is contained in $ \oplus_r \Tr(\H)[r, \leq 0] $.  But this subalgebra has trivial center.
\end{proof}

\subsection{$\Tr(\H)$ as an algebra}

There is an isomorphism of vector spaces
\begin{equation}
\label{Fockspaceiso}
\bar{\psi} \colon
V=\mathbb{C}[h_{1} \otimes 1, h_{2} \otimes 1, \ldots] \rightarrow
\mathcal{V}_{1,0}=\mathbb{C}[w_{-1,0},w_{-2,0},\ldots]
\end{equation}
where
\begin{equation*}
(h_{l_1} \otimes 1) \cdots  (h_{l_r} \otimes 1) \mapsto (w_{-l_1,0}) \cdots (w_{-l_r,0}).
\end{equation*}

\begin{lemma}
The map $ \bar{\psi} $ in ~\eqref{Fockspaceiso} commutes with the action of Heisenberg subalgebras in $ V$ and $ \mathcal{V}_{1,0} $.  That is for any $ v \in V$
\begin{equation*}
\bar{\psi}((h_r \otimes 1) v)= w_{-r,0} \bar{\psi}(v).
\end{equation*}
\end{lemma}

\begin{proof}
Propositions ~\ref{formofV10} and ~\ref{formofV} provide vector space realizations of $\mathcal{V}_{1,0} $ and $ V$ respectively.  The lemma now follows easily from the definition of $ \bar{\psi} $ and the Heisenberg relations calculated in ~\eqref{Walgheisrelations} and Lemma ~\ref{origheisrelations}.
\end{proof}

\begin{lemma}
\label{-c0w01same}
For any $ v \in V$ we have $\bar{\psi}(c_0 v) =  -w_{0,1} \bar{\psi}(v)$.
\end{lemma}
\begin{proof}
This follows easily using ~\eqref{Walgl1k0} to compute $[w_{0,1},w_{n,0}]$, Lemma \ref{c0actingonhn} to compute $ [-c_0,h_n \otimes 1]$ and the definition of the map $\bar{\psi}$.
\end{proof}

\begin{lemma} \label{c1c0w02}
For any $ v \in V$ we have $\bar{\psi}((c_1+c_0) v) =  w_{0,2} \bar{\psi}(v)$.
\end{lemma}
\begin{proof}
By Lemma ~\ref{c1actingonhn} and ~\eqref{heivirformulas} $[c_1,h_n \otimes 1]=2nL_{-n}$. It is well known that the Virasoro operator $2nL_{-n}$ acts on Fock space by $ n \sum_{k \in \mathbb{Z}} (h_{n-k} \otimes 1)(h_k \otimes 1)$. Now it is easy to see that the operator $c_1$ on $V$ acts as
\begin{equation*}
\sum_{k,l>0} ((h_{l} \otimes 1)(h_{k} \otimes 1)(h_{-l-k} \otimes 1)+(h_{l+k} \otimes 1)(h_{-l} \otimes 1)(h_{-k} \otimes 1)).
\end{equation*}
The result now follows from Lemma ~\ref{nablaDAHA}.
\end{proof}

Now we extend the map $ \bar{\psi} $ to a map
\begin{equation}
\label{defpsi}
\psi \colon \Tr(\H) \rightarrow  \W/(C-1,w_{0,0})
\end{equation}
by mapping generators
\begin{equation*}
h_{-1} \otimes 1 \mapsto w_{1,0}
\hspace{.4in}
h_{1} \otimes 1 \mapsto w_{-1,0}
\hspace{.4in}
c_1+c_0 \mapsto w_{0,2}.
\end{equation*}

\begin{lemma}
\label{psifiltered}
The map $ \psi $ is a map of $ \Z$-graded and $ \Z_{\geq 0}$-filtered algebras.
\end{lemma}
\begin{proof}
This follows since the actions of $ \Tr(\H) $ and $ \W/(C-1,w_{0,0}) $  on $V$ and $ \mathcal{V}_{1,0} $  respectively are faithful.
\end{proof}

\begin{remark}
The map $\psi$ maps the Virasoro subalgebra of $ \W/(C-1,w_{0,0}) $ to the Virasoro subalgebra of $ \Tr(\H)$ from Proposition ~\ref{heisvirasoroHH0}.
\end{remark}

\begin{lemma}
The map $\psi$ is surjective.
\end{lemma}
\begin{proof}
This follows directly from the definition of $\psi$ and Lemma ~\ref{lemmaf.5sv} which states that $w_{1,0}, w_{-1,0}, w_{0,2} $ generate $\W$.
\end{proof}

\begin{theorem}
\label{psiiso}
The map $ \psi $ is an isomorphism of algebras.
\end{theorem}

\begin{proof}
We first show that $\psi $ restricts to an isomorphism
\begin{equation*}
\psi^0 \colon \Tr^0(\H) \rightarrow \W^0 /(C-1,w_{0,0}).
\end{equation*}
For some $ a_0, \ldots, a_{l-1}, b_0, \ldots, b_{l-1} \in \mathbb{C}$ we have
\begin{align}
\psi(h_1 \otimes x_1^l + \sum_{k=0}^{l-1} a_k(h_1 \otimes x_1^k)) &=\psi(ad^l(\frac{c_1+c_0}{2})(h_1 \otimes 1)) \\ \nonumber
&=2^{-l} ad^l(w_{0,2})(w_{-1,0}) \\ \nonumber
&=(-1)^l w_{-1,l} + \sum_{k=0}^{l-1} b_k(w_{-1,k}).
\end{align}
The first equality follows from repeated use of Lemmas ~\ref{c0actingonhn} and ~\ref{c1actingonhn}.  The second equality follows from Lemma ~\ref{c1c0w02} and the third equality follows from \cite[(F.9)]{SV}. So
\begin{equation*}
\psi(h_1 \otimes x_1^l)=(-1)^l w_{-1,l} + \sum_{k=0}^{l-1} \alpha_k w_{-1,k}
\end{equation*}
for some $ \alpha_0, \ldots, \alpha_{l-1} \in \mathbb{C}$. By Lemma~\eqref{lemma-1b1a},
\begin{equation} \label{local-101l}
[h_{-1} \otimes 1, h_1 \otimes x_1^l]=\tilde{c}_l + \sum_{k=0}^{l-2} (l-1-k) \tilde{c}_k c_{l-2-k}.
\end{equation}
Hence we have
\begin{equation*}
\psi([h_{-1} \otimes 1, h_1 \otimes x_1^l])
= [w_{1,0}, \psi(h_1 \otimes x_1^l)]
=[w_{1,0},w_{-1,l} + \sum_{k=0}^{l-1} \alpha_k w_{-1,k}].
\end{equation*}
By Lemma ~\ref{Walg-1a1b}
\begin{equation}
\label{localw10-1l}
[w_{1,0}, w_{-1,l}]=-lw_{0,l-1} - \sum_{r=2}^l \binom{l}{r} w_{0,l-r} - 1.
\end{equation}
so that ~\eqref{local-101l} and ~\eqref{localw10-1l} give us
\begin{equation*}
\psi(c_{l-2})=-lw_{0,l-1} + \sum_{k=0}^{l-1} d_k w_{0,k-1} + d
\end{equation*}
for some $d_0,\ldots, d_{l-1}, d \in \mathbb{C}$.
Thus $ \psi $ restricts to an isomorphism $ \psi_0$.

Next observe that
\begin{align*}
\psi(\Tr^{>}(\H)) &\subset \W^> /(C-1,w_{0,0}) \\
\psi(\Tr^{<}(\H)) &\subset \W^< /(C-1,w_{0,0}).
\end{align*}
Since
\begin{align*}
\Tr(\H) &= \Tr^>(\H) \otimes \Tr^0(\H) \otimes \Tr^<(\H) \\
\W /(C-1,w_{0,0}) &= \W^> /(C-1,w_{0,0}) \otimes \W^0 /(C-1,w_{0,0}) \otimes \W^< /(C-1,w_{0,0})
\end{align*}
it suffices to show that $ \psi $ restricted to $ \Tr^>(\H) $ is an isomorphism (the restriction to $ \Tr^<(\H) $ is similar). Since we know that $ \psi $ is surjective we just need that the graded dimensions of $ \Tr^>(\H) $ and $ \W^> /(C-1,w_{0,0}) $ are the same.  This follows from Proposition ~\ref{poincareW} and Corollary ~\ref{poincareH}.
\end{proof}

\subsection{The action of $\Tr(\H)$ on the center}
Recall, by Proposition \ref{KhovEndThm}, that the center of $\H$ is isomorphic to a polynomial algebra in countably many variables:
$$
	Z(\H) = \End_{\H}(\id) \cong \mathbb{C}[c_0,c_1,c_2,\dots].
$$
Since all objects of $\H$ have biadjoints and all adjunctions in $\H$ are cyclic, it follows that the trace $\Tr(\H)$ acts on $Z(\H)$ (see for example \cite[Section 6]{BGHL} or \cite[Section 9]{BHLW}). Graphically, an element of the center $X\in Z(\H)$ can be represented by a closed diagram and an element of the trace by a diagram on the annulus:
\[
f \; = \;\;
\hackcenter{\begin{tikzpicture}[scale=0.9]
      \path[draw,blue, very thick, fill=blue!10]
        (-2.3,-.6) to (-2.3,.6) .. controls ++(0,1.85) and ++(0,1.85) .. (2.5,.6)
         to (2.5,-.6)  .. controls ++(0,-1.85) and ++(0,-1.85) .. (-2.3,-.6);
        \path[draw, blue, very thick, fill=white]
            (-0.2,0) .. controls ++(0,.4) and ++(0,.4) .. (0.4,0)
            .. controls ++(0,-.4) and ++(0,-.4) .. (-0.2,0);
    \draw[very thick] (-1.65,-.7) -- (-1.65, .7).. controls ++(0,1.15) and ++(0,1.15) .. (1.85,.7)
        to (1.85,-.7) .. controls ++(0,-1.15) and ++(0,-1.15) .. (-1.65,-.7);
    \draw[very thick] (-1.1,-.55) -- (-1.1,.55) .. controls ++(0,.75) and ++(0,.75) .. (1.3,.55)
        to (1.3,-.55) .. controls ++(0,-.75) and ++(0,-.75) .. (-1.1, -.55);
    \draw[very thick] (-.55,-.4) -- (-.55,.4) .. controls ++(0,.45) and ++(0,.45) .. (.75,.4)
        to (.75, -.4) .. controls ++(0,-.45) and ++(0,-.45) .. (-.55,-.4);
    \draw[fill=white!20,] (-1.8,-.25) rectangle (-.4,.25);
    \node () at (-1,0) {$f$};
\end{tikzpicture}} \;\; \in \Tr(\H)
\qquad \qquad
X \; =\;\;
\hackcenter{\begin{tikzpicture}[scale = 0.9]
      \path[draw,blue, dashed,very thick,fill=red!10]
        (-2.6,-.6) to (-2.6,.6) .. controls ++(0,1.85) and ++(0,1.85) .. (2.6,.6)
         to (2.6,-.6)  .. controls ++(0,-1.85) and ++(0,-1.85) .. (-2.6,-.6);
    %% TOP
    \draw[very thick] (-2.1,-.55)-- (-2.1,.65) .. controls ++(0,1.15) and ++(0,1.15) .. (2.1,.65)--(2.1,-.55);
    \draw[very thick]  (.9, .25).. controls ++(0,.35) and ++(0,.35) .. (1.5,.25);
    \draw[very thick] (-1.5,.55) .. controls ++(0,.65) and ++(0,.65) .. (.3,.55);
    \draw[very thick] (-.3,.25) .. controls ++(0,.35) and ++(0,.35) .. (-.9,.25);
     %% BOTTOM
    \draw [very thick] (2.1,-.55) .. controls ++(0,-.65) and ++(0,-.65) .. (.3,-.55)-- (.3,.55);
    \draw[very thick]  (1.5,-.25) .. controls ++(0,-.35) and ++(0,-.35) .. (.9, -.25);
    \draw[very thick]  (-1.5, -.25) .. controls ++(0,-.35) and ++(0,-.35) .. (-.9,-.25);
    \draw[very thick]  (-.3,-.25) -- (-.3, -.55) .. controls ++(0,-.65) and ++(0,-.65) .. (-2.1,-.55);
    \draw [very thick] (-1.5,.25)-- (-1.5,.55);
    \draw[fill=white!20,] (-2.3,-.25) rectangle (2.3,.25);
    \node () at (0,0) {$g$};
\end{tikzpicture}}   \;\;\in Z(\H).
 \]
This action of $\Tr(\H)$ on $Z(\H)$ is given by placing the diagram for $X$ in the interior of the annulus
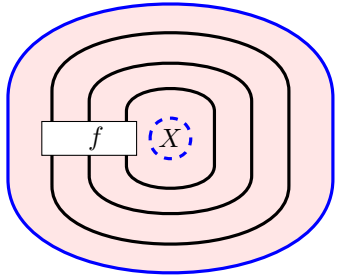
\[
\hackcenter{\begin{tikzpicture}[scale=0.9]
      \path[draw,blue, very thick, fill=red!10]
        (-2.3,-.6) to (-2.3,.6) .. controls ++(0,1.85) and ++(0,1.85) .. (2.5,.6)
         to (2.5,-.6)  .. controls ++(0,-1.85) and ++(0,-1.85) .. (-2.3,-.6);
        \path[draw, blue, very thick, dashed, fill=red!10]
            (-0.2,0) .. controls ++(0,.4) and ++(0,.4) .. (0.4,0)
            .. controls ++(0,-.4) and ++(0,-.4) .. (-0.2,0);
    \draw[very thick] (-1.65,-.7) -- (-1.65, .7).. controls ++(0,1.15) and ++(0,1.15) .. (1.85,.7)
        to (1.85,-.7) .. controls ++(0,-1.15) and ++(0,-1.15) .. (-1.65,-.7);
    \draw[very thick] (-1.1,-.55) -- (-1.1,.55) .. controls ++(0,.75) and ++(0,.75) .. (1.3,.55)
        to (1.3,-.55) .. controls ++(0,-.75) and ++(0,-.75) .. (-1.1, -.55);
    \draw[very thick] (-.55,-.4) -- (-.55,.4) .. controls ++(0,.45) and ++(0,.45) .. (.75,.4)
        to (.75, -.4) .. controls ++(0,-.45) and ++(0,-.45) .. (-.55,-.4);
    \draw[fill=white!20,] (-1.8,-.25) rectangle (-.4,.25);
    \node () at (-1,0) {$f$};
    \node () at (.1,0) {$X$};
\end{tikzpicture}}
 \]
and interpreting the result as a new central element in $Z(\H)$. As a consequence of Theorem \ref{thm:main} and its proof, we have the following.

\begin{corollary}\label{cor:Wrep}
The algebra $\Tr(\H)$ acts irreducibly and faithfully on $Z(\H)$.  Under the isomorphism $\Tr(\H) \cong \W/I$, this representation is isomorphic to the canonical level one representation $ \mathcal{V}_{1,0} $ of $\W$.
\end{corollary}
\begin{proof}
The proof that $\Tr(\H)$ acts irreducibly and faithfully on $Z(\H)$ follows exactly as in Propositions \ref{prop:irrep} and \ref{prop:faithful}. Those propositions used the action of $\Tr(\H)$ on $\Tr(\mathbb{C}[S_n]\dmod)$, but we could have equally well worked with the center of $\H$.  The identification $\End_{\H}(\id)\cong \mathcal{V}_{1,0}$ follows immediately.
\end{proof}

\bigskip

\end{document}